\documentclass[a4paper, abstract]{scrartcl}
\usepackage[scale=0.75]{geometry}

\usepackage{setspace,graphicx,tikz,tabularx} 

\usepackage{amsmath, amsthm, bm}

\usepackage{textcomp}
\usepackage[semibold]{libertine}
\usepackage[libertine, upint]{newtxmath}

\usepackage[scr=euler, scrscaled = 1.0, cal = stixplain, bb = dsfontserif]{mathalpha}

\usepackage[utf8]{inputenc} 
\usepackage[T1]{fontenc} 
\usepackage[ngerman,english]{babel}
\usepackage[draft=false,babel,tracking=true,kerning=true,spacing=true]{microtype} 

\usepackage{xparse,upgreek,xcolor,csquotes,float,array,subcaption,relsize, nicefrac, etoolbox,lastpage, accents,enumerate}
\usepackage[colorinlistoftodos]{todonotes}

\usepackage[unicode=true]{hyperref}					
\hypersetup{
     colorlinks   = true,
     linkcolor    = blue,
     citecolor	  = blue,
     urlcolor	  = blue,
     pdfauthor={Sascha Robert Gaudlitz},
}	

\usepackage[backend=biber, 
sorting=nyt,
bibstyle = authoryear,
citestyle=authoryear-comp,
maxbibnames = 99,
giveninits=true,
date = year,
doi = false,
isbn = false,
url = false,
]{biblatex}

\newcommand{\citep}{\autocite}
\newcommand{\citet}{\textcite}

\AtBeginRefsection{\GenRefcontextData{sorting=ynt}}
\AtEveryCite{\localrefcontext[sorting=ynt]}

\renewbibmacro{in:}{%
  \ifentrytype{article}
    {}
    {\bibstring{in}%
     \printunit{\intitlepunct}}}

\DeclareFieldFormat{pages}{#1}		

	\usepackage{scrlayer-scrpage}
	\automark{section}
	
	\setkomafont{pagehead}{\footnotesize\normalfont}
	
	\usepackage{float}
	\usepackage{subcaption}	
	
	\usepackage{xurl}
\usepackage{xparse}
\usepackage{amsmath, mathtools}
	
\allowdisplaybreaks

\DeclareMathOperator*{\argmin}{argmin}

\DeclareMathOperator{\var}{Var}

\newcommand{\MTkillspecial}[1]{
	\begingroup%
	\catcode`\&=9%
	\let\\\relax%
	\scantokens{#1}%
	\endgroup%
}
\newcommand{\MTemptyplaceholder}{\:\cdot\:}
\DeclarePairedDelimiter\abs\lvert\rvert
\reDeclarePairedDelimiterInnerWrapper\abs{star}{%
	\mathopen{#1\vphantom{\MTkillspecial{#2}}\kern-\nulldelimiterspace\right.}%
	\ifblank{#2}{\MTemptyplaceholder}{#2}%
	\mathclose{\left.\kern-\nulldelimiterspace\vphantom{\MTkillspecial{#2}}#3}%
}

\DeclarePairedDelimiterXPP\iprodWrapper[2]{}{\langle}{\rangle}{}{
	\ifblank{#1}{\MTemptyplaceholder}{#1},
	\ifblank{#2}{\MTemptyplaceholder}{#2}
}
\NewDocumentCommand\iprod{ s o m m }{
	\IfBooleanTF {#1}
	{ \iprodWrapper*{#3}{#4} }
	{ \IfNoValueTF{#2}
		{
			\iprodWrapper{#3}{#4} 
		}
		{
			\iprodWrapper[#2]{#3}{#4}
		}
	}
}

\DeclarePairedDelimiterXPP\normWrapper[1]{}\lVert\rVert{}{\ifblank{#1}{\MTemptyplaceholder}{#1}}

\NewDocumentCommand\norm{ s o m }{
	\IfBooleanTF {#1}
	{ \normWrapper*{#3}}
	{ \IfNoValueTF{#2}
		{
			\normWrapper{#3}
		}
		{
			\normWrapper[#2]{#3}
		}
	}
}

\newcommand{\tnorm}[1]{{\left\vert\kern-0.25ex\left\vert\kern-0.25ex\left\vert #1 
    \right\vert\kern-0.25ex\right\vert\kern-0.25ex\right\vert}}

\providecommand\given{}

\DeclarePairedDelimiterX\Set[1]\{\}{%
	\renewcommand\given{\SetSymbol[\delimsize]}
	#1
}

\DeclarePairedDelimiterXPP\ProbWrapper[2]{#1}(){}{
	\renewcommand\given{\nonscript\:\delimsize\vert\nonscript\:\mathopen{}}
	#2
}

\NewDocumentCommand\prob{ s O{} O{} O{} m }{
	\ifblank {#5}{\mathrm{P}_{#2}^{#3}}
		{\IfBooleanTF {#1}
			{ \ProbWrapper*{\mathrm{P}_{#2}^{#3}}{#5} }
			{ \IfNoValueTF{#4}
				{
					\ProbWrapper{\mathrm{P}_{#2}^{#3}}{#5}
				}
				{
					\ProbWrapper[#4]{\mathrm{P}_{#2}^{#3}}{#5}
				}
			}
		}
}
\NewDocumentCommand\qrob{ s O{} O{} O{} m }{
	\ifblank {#5}{\mathrm{Q}_{#2}^{#3}}
		{\IfBooleanTF {#1}
			{ \ProbWrapper*{\mathrm{Q}_{#2}^{#3}}{#5} }
			{ \IfNoValueTF{#4}
				{
					\ProbWrapper{\mathrm{Q}_{#2}^{#3}}{#5}
				}
				{
					\ProbWrapper[#4]{\mathrm{Q}_{#2}^{#3}}{#5}
				}
			}
		}
}
\DeclarePairedDelimiterXPP\EVWrapper[2]{#1}[]{}{
	\renewcommand\given{\mathrel{}\mathclose{}\delimsize\vert\mathopen{}\mathrel{}}
	#2
}
\NewDocumentCommand\EV{ s O{} O{} O{} m }{
	\ifblank {#5}{\mathrm{E}_{#2}^{#3}}
		{\IfBooleanTF {#1}
			{ \EVWrapper*{\mathrm{E}_{#2}^{#3}}{#5} }
			{ \IfNoValueTF{#4}
				{
					\EVWrapper{\mathrm{E}_{#2}^{#3}}{#5}
				}
				{
					\EVWrapper[#4]{\mathrm{E}_{#2}^{#3}}{#5}
				}
			}
		}
}

\renewcommand{\underbar}[1]{\underaccent{\bar}{#1}}
\providecommand{\transpose}{^{\intercal}}
\newcommand{\inv}{^{{\mathsmaller{-1}}}}

\newcommand{\identity}{\mathbb{I}}		
\newcommand{\indicator}{\mathbb{1}}		

\makeatletter
\providecommand*{\bigcupdot}{%
	\mathop{%
		\vphantom{\bigcup}%
		\mathpalette\@bigcupdot{}%
	}%
}
\newcommand*{\@bigcupdot}[2]{%
	\ooalign{%
		$\m@th#1\bigcup$\cr
		\sbox0{$#1\bigcup$}%
		\dimen@=\ht0 %
		\advance\dimen@ by -\dp0 %
		\sbox0{\scalebox{2}{$\m@th#1\cdot$}}%
		\advance\dimen@ by -\ht0 %
		\dimen@=.5\dimen@
		\hidewidth\raise\dimen@\box0\hidewidth
	}%
}
\makeatother

\newcommand{\D}{\ensuremath{\mathrm{d}}}

\newcommand{\semigroup}{\mathcal{S}}

\let\Pr=P

\makeatletter
\newcommand{\subalign}[1]{%
	\vcenter{%
		\Let@ \restore@math@cr \default@tag
		\baselineskip\fontdimen10 \scriptfont\tw@
		\advance\baselineskip\fontdimen12 \scriptfont\tw@
		\lineskip\thr@@\fontdimen8 \scriptfont\thr@@
		\lineskiplimit\lineskip
		\ialign{\hfil$\m@th\scriptstyle##$&$\m@th\scriptstyle{}##$\crcr
			#1\crcr
		}%
	}
}
\makeatother

	\renewcommand{\phi}{\varphi}			
	\renewcommand{\epsilon}{\varepsilon}	
	\renewcommand{\theta}{\vartheta}		
	\renewcommand{\Delta}{\varDelta}		
	
\newcommand{\smallo}{
	  \mathchoice
	    {{\scriptstyle\mathcal{O}}}
	    {{\scriptstyle\mathcal{O}}}
	    {{\scriptscriptstyle\mathcal{O}}}
	    {\scalebox{.7}{$\scriptscriptstyle\mathcal{O}$}}
	  }	
	  
	\newcounter{thc}[section]
	\numberwithin{thc}{section}
	\numberwithin{equation}{section}
	
	\theoremstyle{plain}
	
		\newtheorem{corollary}[thc]{Corollary}
		\newtheorem{proposition}[thc]{Proposition}
		\newtheorem*{proposition*}{Proposition}
		\newtheorem{theorem}[thc]{Theorem}
		\newtheorem{lemma}[thc]{Lemma}
		\newtheorem*{lemma*}{Lemma}
		\newtheorem{assumption}[thc]{Assumption}
		
	\theoremstyle{definition}
	
		\newtheorem{definition}[thc]{Definition}
		
		\newtheorem{example}[thc]{Example}
		
		\newtheorem{remark}[thc]{Remark}

\setkomafont{sectioning}{\normalcolor\bfseries}

\title{Non--parametric inference on the reaction term in semi--linear SPDEs with spatial ergodicity}
\author{Sascha Gaudlitz\thanks{Humboldt--Universität zu Berlin, Germany \\ Email: sascha.gaudlitz<at>hu-berlin.de}}

\date{}
\begin{document}

\maketitle
\begin{abstract}
	This paper discusses the non-parametric estimation of a non-linear reaction term in a semi-linear parabolic stochastic partial differential equation (SPDE). The estimator's consistency is due to the spatial ergodicity of the SPDE while the time horizon remains fixed.  
	The analysis of the estimation error requires the concentration of spatial averages of non-linear transformations of the SPDE. The method developed in this paper combines the Clark-Ocone formula from Malliavin calculus with the Markovianity of the SPDE and density estimates. The resulting variance bound utilises the averaging effect of the conditional expectation in the Clark-Ocone formula.
The method is applied to two realistic asymptotic regimes. The focus is on a coupling between the diffusivity and the noise level, where both tend to zero. Secondly, the observation of a fixed SPDE on a growing spatial observation window is considered. Furthermore, the concentration of the occupation time around the occupation measure is proved.
\end{abstract}


\section{Introduction}\label{sec:Balances_Intro}

  The main contribution of this work amounts to non-parametrically estimating the non-linear reaction term in semi-linear SPDEs of the type
	\begin{equation}
	\D X_t = \nu A_t X_t\,\D t + f\circ X_t\,\D t + \sigma\, \D W_t,\quad 0\le t\le T,\label{eq:Balanced_SPDE}
\end{equation}
	on a bounded domain $\Lambda\subset \mathbb{R}^d$, $d\in\mathbb{N}$. Here, $(A_t)_{t\in[0,T]}$ is a family of unbounded operators on $L^2(\Lambda)$, for example $A_tz=\operatorname{div}(\omega(t,\MTemptyplaceholder)\nabla z)$ for $\omega\colon [0,T]\times \Lambda \to\mathbb{R}_{>0}$. $(W_t)_{t\in[0,T]}$ is the driving Gaussian process on a probability space $(\Omega,\mathcal{F},\prob*{})$ and $(\mathcal{F}_t)_{t\ge 0}$ is the natural filtration generated by $(W_t)_{t\in[0,T]}$.  We refer to $\nu>0$ as the \emph{diffusivity} and to $\sigma>0$ as the \emph{noise level} of the system. The unknown non-linear function $f\colon \mathbb{R}\to\mathbb{R}$ models local reactions (\emph{reaction function}) and shall be estimated based on observing the process $(X_t)_{t\in[0,T]}$ continuously in time and space. We refer to Section \ref{sec:Setting} for the detailed model assumptions.

		\subsection{Statistical methodology and contributions}\label{subsec:Balanced_BalancedJustification}
		
		In this work, some $x_0$ in the \emph{state space} $\mathbb{R}$ of $X_t(y)$ is fixed and the local reaction $f(x_0)$ is estimated consistently on a finite time horizon $T<\infty$.
		 
		\subsubsection*{Non-parametric regression} Informally, we can rewrite the estimation of the reaction function $f$ in the SPDE \eqref{eq:Balanced_SPDE} into a non-parametric regression framework:
		\begin{equation*}
			\D X_t - \nu A_t X_t\,\D t = f\circ X_t\,\D t + \sigma\,\D W_t,\quad 0\le t\le T,
		\end{equation*}
		where $\D X_t - \nu A_t X_t\,\D t$ corresponds to the \emph{observable}, $f\circ X_t\,\D t$ to the \emph{signal} and $\sigma\,\D W_t$ to the \emph{noise}. From a statistical perspective, the problem of estimating $f$ is thus directly linked to non-parametric regression with \emph{random design}. Importantly, the non-i.i.d.-design is given by evaluations of $X=(X_t(y))_{t\in [0,T],y\in\Lambda}$, depends on the unknown $f$ and is complicated by the non-Gaussianity of $X$. The statistical approach to estimating $f(x_0)$ uses local information of $f$ around $x_0$. To this end, consider a \emph{kernel} $K\colon \mathbb{R}\to\mathbb{R}$ with compact support and its localised version $K_h(x)\coloneqq K((x-x_0)/h)$ with \emph{bandwidth} $h>0$. A canonical choice for estimating $f(x_0)$ is the Nadaraya-Watson estimator
		\begin{equation*}
			\hat{f}(x_0)_{h}^{\operatorname{NW}} \coloneqq \frac{\int_0^T \iprod*{K_h\circ X_t}{\D X_t - \nu A_t X_t\,\D t}_{L^2(\Lambda)}}{\int_0^T \int_\Lambda K_h(X_t(y))\,\D y\D t}.
		\end{equation*}
		Its analysis is closely linked to understanding the design $X$, in particular the marginal densities of $X$ and the behaviour of averages. The main statistical contribution of this work amounts to showing that concentration results for spatial averages of the non-Markovian process $y\mapsto X_t(y)$ are obtainable and that they can enable the consistent estimation of the reaction function $f$. This is exemplified in two concrete settings, which are outlined in the next paragraph.
		
		\subsubsection*{The asymptotic regimes} The Girsanov theorem (e.g.\ Theorem 10.18 of \citet{DaPrato2014}) implies that the laws of $X$ solving the SPDE \eqref{eq:Balanced_SPDE} with globally Lipschitz-continuous reaction functions $f$ and $g$ on the path space $C([0,T],L^2(\Lambda))$ are typically absolutely continuous, thus rendering the model statistically non-identifiable. Consequently, the need for asymptotic regimes affecting not only the observation scheme, but also the law of $X$ itself, arises.
		
		The statistical inference on $f$ was first considered by \citet{Ibragimov1999, Ibragimov2000, Ibragimov2001} and \citet{Ibragimov2003}, who assume that $f(x)$ can be written as $\theta \bar{f}(x)$ with a known reaction function $\bar{f}$, and estimate the reaction intensity $\theta$ in the small noise regime $\sigma\to 0$. This regime corresponds to an asymptotically deterministic design, where $\var(X_t(y))\to 0$. \citet{hildebrandtNonparametricCalibrationStochastic2023} estimate $f$ non-parametrically in the (temporally) ergodic setting $T\to\infty$ and \citet{goldysParameterEstimationControlled2002} conduct parametric estimation. In both works, the design density converges to an ergodic density and the variance of $X_t(y)$ is of the order of 1, i.e.\ $\var(X_t(y))\sim 1$. \citet{gaudlitzEstimationReactionTerm2023} observed that a small diffusivity level $\nu$ allows for the consistent estimation of the reaction intensity $\theta$, provided that the reaction function $\bar{f}$ is known. This asymptotic regime amounts to a design with exploding variance, where $\var(X_t(y))\to\infty$.
		
		 The focus of this work is on the asymptotic regime $\nu\to 0$, which is realistic in applications. Small values of $\nu$ yield the important class of \emph{diffusion-limited reactions} in physical chemistry \citep{riceDiffusionlimitedReactions1985}. When modelling biodiversity, small values of $\nu$ allow for the coexistence of populations \citep{groseljHowTurbulenceRegulates2015}. Further examples with small or medium diffusivity include the works of \citet{Soh2010, Alonso2018,flemmingHowCorticalWaves2020, Altmeyer2020b}. The noise level $\sigma=\sigma(\nu)\to 0$ is coupled to $\nu$ and tends to zero such that the design does not degenerate, i.e.\ $\var(X_t(y))\sim 1$. Since the latter is important when modelling random phenomena with SPDEs at small diffusivity, the coupling of $\sigma$ to $\nu$ arises naturally.
		 In computational neuroscience, a small noise level $\sigma$ increases the transmission probability \citep{tuckwellStochasticPartialDifferential2013}. A large noise level $\sigma$ hides the non-linear behaviour induced by $f$ \citep{Pasemann2021}. Viewing SPDEs as a method to quantify model uncertainty for PDEs additionally motivates small values of $\sigma$.
		 
		 Using the concrete example of an SPDE with fixed diffusivity level $\nu=1$ and noise level $\sigma=1$, which is observed on a growing spatial observation window, we show that the methodology developed in this work is not specific to the $\nu\to 0$ and $\sigma(\nu)\to 0$ asymptotics.
		
		
		\subsubsection*{Statistical results} The (asymptotic) \emph{spatial ergodicity} of the process $y\mapsto X_t(y)$ for fixed time $0<t\le T$ gives rise to the concentration of spatial averages of the SPDE (Section \ref{sec:Density}) and leads to the consistency of the estimator for $f(x_0)$ (Theorems \ref{thm:Nonparametric} and \ref{thm:Nonparametric_GrowingDomains}). If $\D W_t / \D t$ is space-time white noise, a central limit theorem (Corollary \ref{cor:CLT}) and the minimax-optimality of the convergence rate (Proposition \ref{prop:LowerBound}) are established. The estimator can be computed efficiently from discrete observations using weighted least squares (see Remarks \ref{rmk:PracticalImplementation} and Subsection \ref{subsec:Balanced_Numerics}).
		
		\subsection{Probabilistic challenges and contributions}\label{subsec:Balanced_Probchallenges}	
		
		 A precise control of the design and of the estimation error is achieved by using the spatial ergodicity of the SPDE. The fundamental phenomenon of spatial ergodicity for SPDEs has first been observed by \citet{chenSpatialErgodicitySPDEs2021} and has led to numerous convergence results for spatial averages of the type $\int_{[0,N]^d} g(v_t(y))\,\D y$ as $N\to\infty$, where $v$ solves the stochastic heat equation on $\mathbb{R}^d$ \citep{khoshnevisanSpatialStationarityErgodicity2021, kimLimitTheoremsTimedependent2022, kuzgunConvergenceDensitiesSpatial2022, chenSpatialErgodicityCentral2022, chenCentralLimitTheorems2022a, chenCentralLimitTheorems2023}. In the setting presented in this work, the spatial ergodicity occurs asymptotically as $\nu\to 0$ and, in principle, the previous results could be adapted to this case. Importantly, they are not suited to perform non-parametric estimation, since they are not uniform in the Lipschitz-constant of $g$. We tackle the following challenges, which are of independent statistical and probabilistic interest.
		
	\subsubsection*{Concentration of spatial averages}
	The analysis of the estimator  $\hat{f}(x_0)_{h}^{\operatorname{NW}}$ requires an in-depth understanding of the fluctuations of functionals of the spatial process $y\mapsto X_t(y)$ of the type
	\begin{equation}
		\int_\Lambda K_h(X_t(y))\,\D y, \quad h>0,\quad 0\le t\le T,\label{eq:deltasequence}
	\end{equation}
	where $K_h$ behaves similarly to a delta-sequence as $h\to 0$. If $y\mapsto X_t(y)$ was generated by a stochastic (fractional) differential equation (SDE), the fluctuations could be controlled using well-established tools. These include the local time \citep{Kutoyants2013, RevuzYor1999}, the It\^o-formula \citep{nicklNonparametricStatisticalInference2020}, mixing properties \citep{castellanaSmoothedProbabilityDensity1986, comteSuperOptimalRates2005}, martingale approximations \citep{aeckerle-willemsConcentrationScalarErgodic2021, trottnerConcentrationAnalysisMultivariate2023}, the (fractional) Meyer inequality \citep{comteNonparametricEstimationFractional2019, huDriftParameterEstimation2019} and the spectral gap property of the transition semi-group \citep{dalalyanAsymptoticStatisticalEquivalence2007}. Even though first results in this direction have been obtained for the process $y\mapsto X_t(y)$ for $d=1$ or in the Gaussian case by \citet{tudorChaosExpansionRegularity2013, sunQuadraticCovariationsSolution2020, chenSpatialErgodicitySPDEs2021, boufoussiLocalTimesSystems2023, boufoussiBesovRegularityLocal2023, huStochasticHeatEquation2009}, none of these tools seem available in the required generality. As previous research on the local time for random fields (e.g.\ \citet{gemanOccupationDensities1980, khoshnevisanMultiparameterProcessesIntroduction2002}) has primarily focused on Gaussian fields, its application to our setting is challenging. The spatial ergodicity approach of \citet{chenSpatialErgodicitySPDEs2021} and the variance bound of Lemma 2.12 of \citet{gaudlitzEstimationReactionTerm2023}, which are both based on the Poincaré inequality, can be adapted to the setting presented here, but only yield a variance bound for \eqref{eq:deltasequence} of the order of $\mathcal{O}(\sigma(\nu)^2 h^{-2})$ and $\mathcal{O}(\sigma(\nu)^2 h^{-1})$, respectively. As $h\to 0$ is required to reduce the bias of the estimator $\hat{f}(x_0)_{h}^{\operatorname{NW}}$, these bounds are not sufficiently sharp.
	
	Instead, a novel approach employing the Clark-Ocone formula together with upper bounds on the (Lebesgue-) density of $X_t(y)$, for $0<t\le T$ and $y\in\Lambda$, is introduced, which yields the bound
	\begin{equation*}
		\var\left(\int_\Lambda K_h(X_t(y))\,\D y\right) = \mathcal{O}(\sigma(\nu)^2),\quad h>0,\quad 0\le t\le T,
\end{equation*}
in Proposition \ref{prop:VarianceBound}. The first observation in the proof is that the upper bound for the density of $X_t(y)$ does not depend on the (deterministic) initial condition. In a second step, we use these density bounds to exploit the averaging effect of the conditional expectation in the Clark-Ocone formula and obtain the bound for the variance. This is a general approach since it only requires the following properties of the process $X$: As a process in time it needs to be Markovian, allow for a Clark-Ocone formula and obey upper and lower bounds of the marginal Lebesgue-densities. A related approach has been used by \citet{kohatsu-higaApproximationsNonsmoothIntegral2014} for controlling the discretisation error of time averages of functionals of one-dimensional diffusions. In the setting presented in this work, the infinite-dimensional nature of the SPDE \eqref{eq:SPDE1} significantly complicates the treatment of the Malliavin derivative and the densities compared to the setting considered by \citet{kohatsu-higaApproximationsNonsmoothIntegral2014}. The Clark-Ocone formula has been used in a statistical context, e.g.\ by \citet{gobetLocalAsymptoticMixed2001}, but - to the best of the author's knowledge - this is the first time the  averaging effect of the conditional expectation in the integrand of the Clark-Ocone formula is used explicitly.

As a by-product, Lemma \ref{lem:Occupationtime} shows that this technique can be used to prove concentration results for the \emph{occupation time} of the process $y\mapsto X_t(y)$. Extensions of the method beyond parabolic equations pose interesting questions for future research.
		\subsubsection*{Density bounds using minimal regularity of $f$} From a statistical perspective, a key challenge is to obtain convergence rates of the estimator that improve with higher smoothness of the function $f$ around $x_0$. This is usually achieved by using  higher-order kernels (see Proposition 1.65 of \citet{Kutoyants2013} or Section 1.2 of \citet{Tsybakov2009}) or by using local polynomials (see \citet{fanLocalPolynomialModelling2017} or Section 1.6 of \citet{Tsybakov2009}). The former requires that if $f$ has $k\in\mathbb{N}$ bounded derivatives, then the marginal density of $X_t(y)$ also has $k$ bounded derivatives for $0<t\le T$ and $y\in\Lambda$. No such results seem available for SPDEs. Most quantitative results on the densities of SPDEs proceed similarly to Theorem 2.1.4 of \citet{nualartMalliavinCalculusRelated2006} and require $k+2$ bounded derivatives of $f$ for $k$ bounded derivatives of the density \citep{carmonaRandomNonlinearWave1988a, milletStochasticWaveEquation1999, marquez-carrerasStochasticPartialDifferential2001, dalangPotentialTheoryHyperbolic2004, SanzSole2005, nualartExistenceSmoothnessDensity2007a, dalangHittingProbabilitiesSystems2009,  marinelliExistenceRegularityDensity2013}. Assuming smoothness of $f$ with bounded derivatives of all orders, Gaussian lower and upper density bounds are proven by \citet{nualartGaussianEstimatesDensity2012a}.
		
		Regularity results for the density for non-smooth $f$ are given by \citet{romitoSimpleMethodExistence2018} and \citet{nourdinDensityFormulaConcentration2009}. In the latter, requiring only global Lipschitz-continuity of $f$, the authors bound the densities (but not their derivatives) from below and above. Their approach has been applied by \citet{nualartGaussianDensityEstimates2009a, nualartOptimalGaussianDensity2011} to SPDEs, and we extend their results to more general operators $A_t$, domains $\Lambda$, and spatial covariance structures of $\D W_t/\D t$ (Proposition \ref{prop:DensityBounds}). 
			The lower and upper bounds for the densities in Proposition \ref{prop:DensityBounds} pave the way for other statistical methods, compare Assumption (E) and its discussion of \citet{hildebrandtNonparametricCalibrationStochastic2023}.		
			
\subsection{Outline}
	
	Section \ref{sec:Setting} introduces the analytical setting and the model assumptions. The main results on the non-parametric estimation constitute Section \ref{sec:Balanced_MainResults}. The bounds for the marginal densities of $X_t(y)$ for $0<t\le T$, $y\in\Lambda$, and the concentration results due to (asymptotic) spatial ergodicity are included in Section \ref{sec:Density}. In Section \ref{sec:GrowingObs} we show that the methodology developed in this work extends beyond the case $\nu\to 0$, namely to the setting of a semi-linear stochastic heat equation on $\mathbb{R}$ with fixed diffusivity $\nu=1$ and noise level $\sigma = 1$, which is observed on a growing spatial observation window. The appendix contains the remaining proofs. 
	
	\section{Setting and Assumptions}\label{sec:Setting}

For $d\in\mathbb{N}$ let $\Lambda\subset\mathbb{R}^d$ be open and bounded with Lipschitz boundary, and endow $\mathcal{H}\coloneqq L^2(\Lambda)$ with the standard inner product $\iprod{}{}$. Consider the SPDE
\begin{equation}
	\D X_t = \nu A_tX_t\,\D t + F(X_t)\,\D t + \sigma \D W_t,\label{eq:SPDE1}
\end{equation}
on $\Lambda$ with deterministic and continuous initial condition $X_0\in C(\Lambda)$ and with either Dirichlet or Neumann boundary conditions. For a bounded linear operator $B\colon \mathcal{H}\to\mathcal{H}$, the process $(W_t)_{t\in[0,T]}$ is a cylindrical $Q$-Wiener process with covariance operator $Q=B B^\ast$ on $\mathcal{H}$. $F$ is of Nemytskii-type, i.e.\ $F(z)=f\circ z$  for a non-linear function $f\colon \mathbb{R}\to\mathbb{R}$.

Let $\mathbb{H}$ be the closure of $\mathcal{H}$ with respect to the norm $\norm{}_{\mathbb{H}}$, which is induced by the inner product $\iprod{}{}_{\mathbb{H}} \coloneqq \iprod{Q \MTemptyplaceholder}{}$. The formal time-derivative $\D W_t/\D t$ can be interpreted as an isonormal Gaussian process $\mathcal{W}$ on $\mathfrak{H}\coloneqq L^2([0,T],\mathbb{H})$, where the spatial covariance is encoded in the norm of $\mathbb{H}$. For a linear bounded operator $L\colon\mathcal{H}\to\mathcal{H}$, let $\norm*{L}$ denote its operator norm, and deduce $\norm*{\xi}_{\mathbb{H}}\le \norm*{B}\norm*{\xi}$ for any $\xi\in \mathcal{H}$. Let $\identity$ be the identity operator on $\mathcal{H}$. We write $a\lesssim b$ (or $b\gtrsim a$) if there exists a constant $0<C<\infty$ depending only on non-asymptotic quantities such that $a\le Cb$ and $a\sim b$ if $a\lesssim b$ and $a\gtrsim b$. We denote by $\xrightarrow{\prob*{}}$ probabilistic and by $\xrightarrow{d}$ distributional convergence of random variables. For equality in distribution under a law $\qrob{}$ we write $\overset{\qrob{}}{=}$. Define the following H\"older-classes. For $\beta=1$, $\Sigma(\beta,L)$ denotes the class of globally Lipschitz-continuous functions $\mathbb{R}\to\mathbb{R}$ with Lipschitz-norm $L$. For $1<\beta\le 2$, $\Sigma(\beta,L)$ denotes the class of differentiable functions $\mathbb{R}\to\mathbb{R}$, whose derivative has $(\beta-1)$-H\"older-norm $L$. The indicator function for a set $M$ is denoted by $\indicator_M$. For $a,b\in\mathbb{R}$ let $a\wedge b\coloneqq \min(a,b)$ and $a\vee b \coloneqq \max(a,b)$. For two sets $M$, $N\subset\mathbb{R}^d$ we define their distance as $\operatorname{dist}(M,N)\coloneqq \inf\Set{\abs{m-n}\given m\in M, n\in N}$. The Lebesgue-measure of a Borel-set $M\subset\mathbb{R}^d$ is denoted by $\abs{M}$.
		
		 We proceed by collecting and discussing the main model conditions. This is complemented by concrete examples satisfying these assumptions (Example \ref{examp:main}). Since we consider the small diffusivity regime $\nu\to 0$, we restrict $0<\nu\le\bar{\nu}$ for some arbitrary but fixed $\bar{\nu}>0$.
\begin{assumption}[\textcolor{black}{Well-posedness}]\label{assump:WellPosedness}~
	\begin{enumerate}[(a)]
	\item\label{num:AssumpWellposedness_a} The reaction function $f$ is globally Lipschitz-continuous and $f'$ denotes its almost everywhere existing derivative. $\norm{f'}_\infty<\infty$ denotes the Lipschitz-constant of $f$.
	\item\label{num:AssumpWellposedness} The operator families $(A_t)_{t\in[0,T]}$ and $(A_t^\ast)_{t\in[0,T]}$ live on common domains $\operatorname{dom}(A_t)=\operatorname{dom}(A_0)\subset \mathcal{H}$ and $\operatorname{dom}(A_t^\ast)=\operatorname{dom}(A_0^\ast)\subset \mathcal{H}$, respectively, for all $0\le t\le T$ with $C_c^\infty(\Lambda)\subset \operatorname{dom}(A_0^\ast)$. Furthermore, for every $0<\nu\le\bar{\nu}$, the linear and deterministic equation
	\begin{equation*}
		 \frac{\partial}{\partial t}u_t= \nu A_t u_t,\quad 0\le \underbar{t}\le t\le T,\quad u_{\smash{\underbar{t}}}=\xi\in \operatorname{dom}(A_0),
\end{equation*}
 is solved by
	\begin{equation*}
		u_t = \semigroup_{\nu,t,\underbar{t}}\xi \coloneqq \int_\Lambda G_{\nu,t,\underbar{t}}(\MTemptyplaceholder,\eta)\xi(\eta)\,\D\eta,\quad 0\le \underbar{t}<t\le T,
	\end{equation*}
	for a non-negative \emph{Green function} $G_{\nu,\MTemptyplaceholder,\smash{\underbar{t}}}(\MTemptyplaceholder,\MTemptyplaceholder)\colon (\smash{\underbar{t}},T]\times \Lambda\times\Lambda\to \mathbb{R}_{\ge 0}$. Assume that there exists a constant $C_0>0$ such that
		\begin{equation*}
			\max\left(\norm*{\semigroup_{\nu,t,\smash{\underbar{t}}}},\norm*{G_{\nu,t,\smash{\underbar{t}}}(y,\MTemptyplaceholder)}_{L^1(\Lambda)}\right)\le C_0<\infty,\quad 0\le \smash{\underbar{t}}<t\le T,\quad y\in\Lambda,
		\end{equation*}
		and that
		\begin{equation}
			\frac{\partial}{\partial t}\semigroup_{\nu,t,\smash{\underbar{t}}}^\ast \xi = \nu A_t^\ast \semigroup_{\nu,t,\smash{\underbar{t}}}^\ast \xi =  \nu \semigroup_{\nu,t,\smash{\underbar{t}}}^\ast A_t^\ast\xi ,\quad \xi\in \operatorname{dom}(A_0^\ast),\quad 0\le \smash{\underbar{t}}<t\le T,\label{eq:Commuting}
		\end{equation}
		for all $0<\nu\le \bar{\nu}$.
		\item\label{num:AssumpWellposedNess_c} For all $t\ge 0$ and $0<\nu\le\bar{\nu}$ we have
		\begin{equation*}
			\sup_{y\in\Lambda}\int_0^t\norm*{G_{\nu,t,s}(y,\MTemptyplaceholder)}_{\mathbb{H}}^2\,\D s<\infty.
		\end{equation*}
\end{enumerate}
\end{assumption}

\begin{remark}[On the consequences of Assumption \ref{assump:WellPosedness} {\hyperref[assump:WellPosedness]{(well-posedness)}}]~
\begin{enumerate}[(a)]\label{rmk:Well_posedness}
	\item We can consider leading order operators of divergence type $A_t=\operatorname{div}(\omega_t\nabla)$ for temporally and spatially varying diffusivity $\omega_t(y)$ with ellipticity condition bounds $0<\underbar{\omega}\le \omega_t(y)$ in Examples \ref{examp:main} (\ref{num_Examp_a}) (\ref{num_Examp_b}) and (\ref{num_Examp_c}) (below).
	 \item By standard arguments (e.g.\ Theorem 2.4.3 of \citet{nualartMalliavinCalculusRelated2006}
	 or Theorem 1 of \citet{Ibragimov2003}), Assumption \ref{assump:WellPosedness} \hyperref[assump:WellPosedness]{(well-posedness)} ensures that the semi-linear SPDE \eqref{eq:SPDE1} admits a Markovian random field solution, in the sense that for all $0\le \underbar{t}< t\le T$ and $y\in\Lambda$ the equality
\begin{align}
\begin{split}
	X_t(y) &= \sigma \int_{\underbar{t}}^t\int_\Lambda G_{\nu,t,s}(y,\eta)\mathcal{W}(\D\eta,\D s)\\
	 &\quad +\int_{\underbar{t}}^t\int_\Lambda G_{\nu,t,s}(y,\eta)f(X_s(\eta))\,\D \eta\D s + \int_\Lambda G_{\nu,t,\underbar{t}}(y,\eta)X_{\underbar{t}}(\eta)\,\D\eta
	\end{split}\label{eq:RandomField1}
\end{align}
	is satisfied. We give a proof in Lemma \ref{lem:aux_WellPosedness} (\ref{num:aux_wellposed_a}) for completeness.
	\item The Markovianity of $t\mapsto(t,X_t)$ plays an important role in the following. For a starting time $0\le \smash{\underbar{t}} <T$ and an initial condition $\xi\in C(\Lambda)$ we denote by $\prob*[\smash{(\smash{\underbar{t}},\xi)}]{}$ the law of $(X_t)_{t\in [0,T-\smash{\underbar{t}}]}$ started at $(\smash{\underbar{t}},\xi)$. For notational simplicity, we abbreviate $\prob*{}\coloneqq \prob*[\smash{(0,X_0)}]{}$.
	\item Similarly to Proposition 2.4.4 of \citet{nualartMalliavinCalculusRelated2006}
	, the Malliavin derivative $\mathcal{D}X_t(y) = \mathcal{D}_{\MTemptyplaceholder,\MTemptyplaceholder} X_t(y)\in\mathfrak{H}$ of $X_t(y)$ under $\prob*[\smash{(\underbar{t},\xi)}]{}$ exists for all starting times $0\le \smash{\underbar{t}}<T$, deterministic initial conditions $\xi\in C(\Lambda)$, times $0< t\le T-\smash{\underbar{t}}$ and locations $y\in \Lambda$, and it satisfies
\begin{align}
\begin{split}
	\mathcal{D}_\tau X_t(y)&= \mathcal{D}_{\tau,\MTemptyplaceholder} X_t(y) \\
	&=\sigma G_{\nu,\smash{\underbar{t}}+t,\smash{\underbar{t}}+\tau}(y,\MTemptyplaceholder)+ \int_{\tau}^{t}\int_\Lambda G_{\nu,\smash{\underbar{t}}+t,\smash{\underbar{t}}+s}(y,\eta)f'(X_s(\eta))\mathcal{D}_\tau X_s(\eta)\,\D \eta\D s
	\end{split}\label{eq:MalliavinDerivative}
	\end{align}
	for $0\le \tau<t$ and $\mathcal{D}_\tau X_t(y)=0$ for $t\le \tau \le T-\smash{\underbar{t}}$. We give a proof in Lemma \ref{lem:aux_WellPosedness} (\ref{num:aux_wellposed_b}) for completeness. If $f$ is globally Lipschitz-continuous but not differentiable, then $f'(X_t(y))$ is an adapted process satisfying the bound $\sup_{0\le t\le T-\smash{\underbar{t}}, y\in\Lambda}\abs{f'(X_t(y))}\le \norm{f'}_\infty$ almost surely.\label{num:MalliavinDeriv}
	\item In Lemma \ref{lem:aux_WellPosedness} (\ref{num:aux_wellposed_c}) we show that 
	 the random field solution is also an analytically weak solution, in the sense that
	\begin{equation}
		\iprod*{X_t}{\phi}-\iprod*{X_0}{\phi}=\int_0^t \left(\iprod{X_s}{\nu A_s^\ast\phi} + \iprod*{F(X_s)}{\phi}\right)\,\D s + \sigma \int_0^t \iprod*{\phi}{\D W_s},\quad 0\le t\le T,\label{eq:WeakSolution}
	\end{equation}
	for all $\phi\in C_c^\infty(\Lambda)\subset \operatorname{dom}(A_0^\ast)$.
	\end{enumerate}
\end{remark}

The process $X$ is observed continuously in time on a bounded and connected (spatial) \emph{observation window} $\Gamma\subset \Lambda$ with positive Lebesgue-measure $\abs{\Gamma}>0$. The following assumption is the basis for all subsequent results.
\begin{assumption}[\textcolor{black}{Noise-scaling}]\label{assump:key}
	The noise level $\sigma=\sigma(\nu)\to 0$ depends on the diffusivity $0<\nu\le \bar{\nu}$ and tends to zero such that for absolute constants $0<\underbar{C}\le \bar{C}<\infty$, $0<\alpha<1$ and any $0\le \smash{\underbar{t}}<T$ we have
	\begin{equation}
		\sigma(\nu)^2\int_{0}^{t} \norm*{G_{\nu,\smash{\underbar{t}}+t,\smash{\underbar{t}}+s}(y,\MTemptyplaceholder)}_{\mathbb{H}}^2\,\D s\le \bar{C} t^{\alpha},\quad y\in\Lambda,\quad 0\le t\le T-\smash{\underbar{t}},\quad 0<\nu\le \smash{\bar{\nu}},\label{eq:Upperbound}
	\end{equation}
	as well as
	\begin{equation}
			\sigma(\nu)^2\int_{0}^{t}  \norm*{G_{\nu,\smash{\underbar{t}}+t,\smash{\underbar{t}}+s}(y,\MTemptyplaceholder)}_{\mathbb{H}}^2\,\D s\ge \underbar{C}\label{eq:Lowerbound} t^{\alpha},\quad y\in\Gamma,\quad 0\le t\le T-\smash{\underbar{t}},\quad 0<\nu\le \smash{\bar{\nu}}.
	\end{equation}
\end{assumption}
In the sequel, the dependency of the noise level on the diffusivity is sometimes omitted in the notation and we simply write $\sigma$ instead of $\sigma(\nu)$.
\begin{remark}[On Assumption \ref{assump:key} {\hyperref[assump:key]{(noise-scaling)}}]\label{rmk:on_Noisescaling}~
\begin{enumerate}[(a)]
	\item As for any time $0\le t\le T$ and location $y\in\Lambda$ we have
	\begin{equation*}
		\sigma^2\int_{0}^{t}  \norm*{G_{\nu,\smash{\underbar{t}}+t,\smash{\underbar{t}}+s}(y,\MTemptyplaceholder)}_{\mathbb{H}}^2\,\D s = \var\left(\sigma\int_{0}^{t}  \int_\Lambda G_{\nu,\smash{\underbar{t}}+t,\smash{\underbar{t}}+s}(y,\eta)\mathcal{W}(\D \eta,\D s)\right),
	\end{equation*}
	Assumption \ref{assump:key} \hyperref[assump:key]{(noise-scaling)} ensures that the noise level $\sigma$ and the diffusivity $\nu$ scale such that the variance of the Gaussian integral in \eqref{eq:RandomField1} does not degenerate as $\nu\to 0$ and thus $t^{-\alpha}\var(X_t(y))\sim 1$, $0<t\le t$.  Any other coupling of $\nu$ and $\sigma$ would imply either exploding or vanishing variance of $X_t(y)$. Consequently, this coupling between $\nu$ and $\sigma$ arises naturally when modelling random phenomena with SPDEs at low diffusivity.
	\item If $B$ has finite Hilbert-Schmidt norm, then \eqref{eq:Lowerbound} can only hold with $\sigma$ not tending to zero and Assumption \ref{assump:key} \hyperref[assump:key]{(noise-scaling)} is violated. Since the latter is required for first order SPDEs, this case is excluded.
	\item In the case of Dirichlet boundary conditions, the lower bounds in Assumption \ref{assump:key} \hyperref[assump:key]{(noise-scaling)} cannot hold uniformly in $y\in\Lambda$, which can be circumvented by the restriction to a subset $\Gamma$ with positive distance to $\partial\Lambda$.
	\item The parameter $0<\alpha<1$ in Assumption \ref{assump:key} \hyperref[assump:key]{(noise-scaling)} regulates the temporal behaviour of the density of $X_t(y)$, compare Proposition \ref{prop:DensityBounds} (below). The lower $\alpha$, the faster the density is smoothed out after the deterministic initial condition.
	\item  For any exponent $p\ge 1$, Lemma \ref{lem:UniformBoudnednessSolution} shows $\EV{\abs{X_t(y)}^p}<\infty$, uniformly in $0\le t\le T$, $y\in\Lambda$ and $0<\nu\le \bar{\nu}$.
\end{enumerate}
\end{remark}

The upper and lower bounds in Assumption \ref{assump:key} \hyperref[assump:key]{(noise-scaling)} form the basis for Corollary \ref{cor:DensityBounds} (below), which ensures that the (Lebesgue-) density $p_{\nu,\smash{\underbar{t}},t,y}$ of $X_{t}(y)$ under $\prob*[\smash{(\smash{\underbar{t}},\xi)}]{}$ exists for all $0<t\le T-{\smash{\underbar{t}}}$, $y\in\Lambda$, and can be bounded from below and above. For notational simplicity, we write $p_{\nu,t,y}\coloneqq p_{\nu,0,t,y}$. Note that $p_{\nu,\smash{\underbar{t}},t,y}$ also depends on the initial condition $\xi$, but the upper bounds on $p_{\nu,\smash{\underbar{t}},t,y}$ are uniform in $\xi$ and this dependency is suppressed for notational convenience.

\begin{example}\label{examp:main}~
\begin{enumerate}[(a)]
	\item Space-time white noise: The main example is the semi-linear stochastic heat equation with space-time white noise on an open and bounded interval $\Lambda\subset \mathbb{R}$, i.e.\ $d=1$, $A_t\equiv \Updelta$, $B=\identity$ and Dirichlet boundary conditions. Then Assumption \ref{assump:WellPosedness} \hyperref[assump:WellPosedness]{(well-posedness)} is satisfied and Lemma \ref{lem:Variance_Control_Mainexample*} yields Assumption \ref{assump:key} \hyperref[assump:key]{(noise-scaling)} with $\sigma^2=\nu^{1/2}$ and $\alpha=1/2$ if $\operatorname{dist}(\Gamma,\partial \Lambda)>0$. Typical realisations of $X$ for high and low diffusivity $\nu\in \Set{0.1,0.001}$ are displayed in Figure \ref{fig:Examp}.\label{num_Examp_a}
	\item Coloured noise I (Riesz kernel): Take $A_t\equiv\Updelta$ with Dirichlet boundary conditions on an open and bounded domain $\Lambda\subset\mathbb{R}^d$, $d\ge 1$, with Lipschitz-boundary and consider the covariance kernel $\chi(x)=\abs*{x}^{-\rho}$ for $1/2<\rho<1$. Define
	\begin{equation*}
		[Q\xi](y) \coloneqq \int_\Lambda \chi(y-\eta)\xi(\eta)\,\D \eta,\quad \xi\in \mathcal{H},\quad y\in\Lambda,
	\end{equation*}
	then Assumption \ref{assump:WellPosedness} \hyperref[assump:WellPosedness]{(well-posedness)} is satisfied and Lemma \ref{lem:VaryingCorrelationLength} yields Assumption \ref{assump:key} \hyperref[assump:key]{(noise-scaling)} with $\sigma^2=\nu^{\rho/2}$ and $\alpha = 1-\rho/2$, if $\Gamma$ is convex with $\operatorname{dist}(\Gamma,\partial \Lambda)>0$.
	\label{num_Examp_b}
	\item Coloured noise II (Spectral dispersion $B$): Fix $\rho_1>0$, $0\le \rho_2<2$ with ${\rho_1+2\rho_2\ge 1}$ and assume that $A_t\equiv A$ and $B$ have discrete spectra with the same eigenfunctions $(e_k)_{k\in\mathbb{N}}$ and eigenvalues of the order of $\lambda_k\sim -k^{\rho_1}$, $\sigma_k\sim k^{-\rho_2}$, respectively. If $A$ is the generator of a Markov process (e.g. $A= -(-\Updelta)^{\rho_1/2}$ with $1<\rho_1\le 2$, \citep{Lischke2020}), then the Green function $G$ for $A$ is given by the transition probabilities and Assumption \ref{assump:WellPosedness} \hyperref[assump:WellPosedness]{(well-posedness)} is satisfied. Lemma \ref{lem:Variance_Control_Spectral} yields Assumption \ref{assump:key} \hyperref[assump:key]{(noise-scaling)} with $\sigma^2=\nu^{(1-2\rho_2)/\rho_1}$ and $\alpha = 1+(2\rho_2-1)/\rho_1$, provided the mild Assumption  \ref{assump:B_Aux} (below) on the eigenfunctions $(e_k)_{k\in\mathbb{N}}$.\label{num_Examp_d}
		\item Lemma \ref{lem:Neumann} shows that the statements of (\ref{num_Examp_a}) and (\ref{num_Examp_b}) hold true with $\Gamma=\Lambda$, if we consider Neumann instead of Dirichlet boundary conditions.\label{num_Examp_c}
\end{enumerate}
\end{example}

\begin{figure}
\includegraphics[width = 0.32\textwidth]{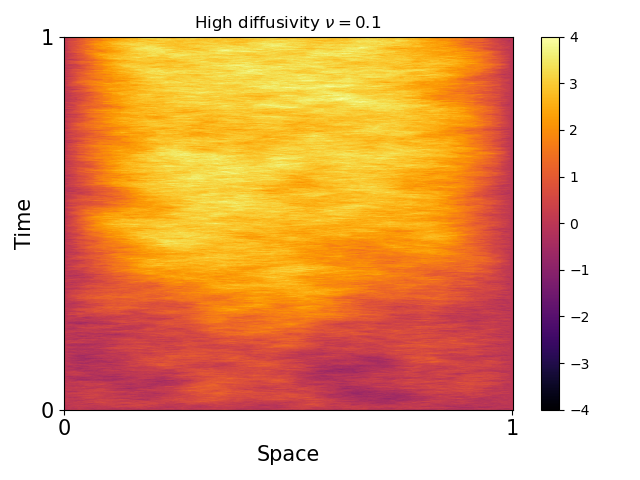}
\includegraphics[width = 0.32\textwidth]{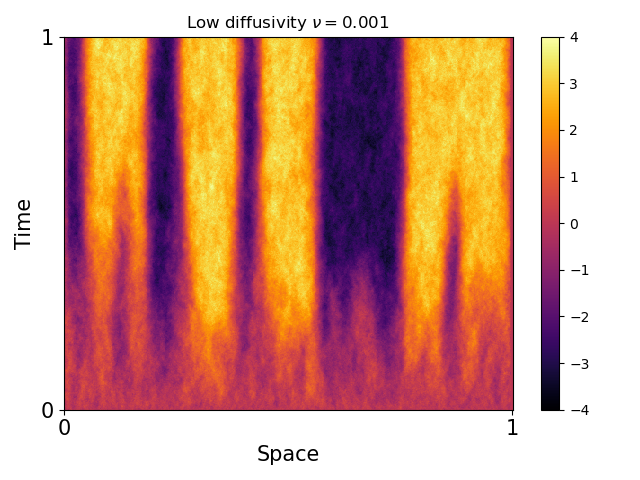}
\includegraphics[width = 0.32\textwidth]{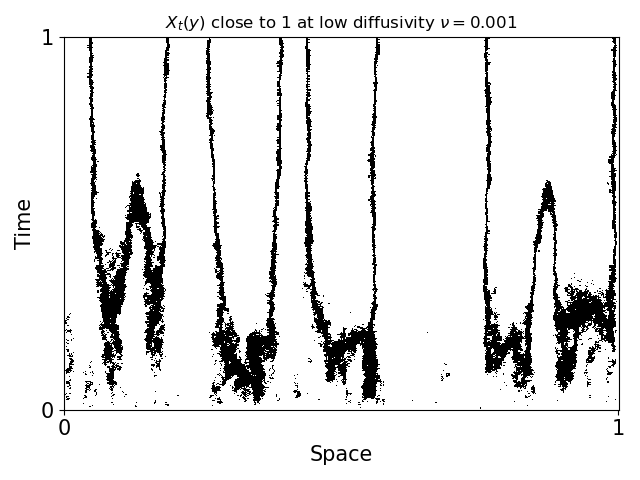}
\caption{Realisations of the semi-linear stochastic heat equation with space-time white noise (Example \ref{examp:main} (\ref{num_Examp_a})) and Allen-Cahn non-linearity $f$ with stable points $\pm 3$ given by \eqref{eq:Simulation_f} for $\nu=0.1$ (left) and $\nu=0.001$ (middle and right). Right: Those space-time points $(t,y)$, where $\abs{X_t(y)-1}\le 0.5$ are highlighted in black.
}\label{fig:Examp}
\end{figure}

\begin{remark}[On the connection to spatial ergodicity]\label{rmk:SpatialErgodicity}
	Further intuition for the coupling $\sigma=\sigma(\nu)\to 0$ can be obtained in the case of Example \ref{examp:main} (\ref{num_Examp_a}), the semi-linear stochastic heat equation driven by space-time white noise. Lemma \ref{lem:zoomingout} shows that $Y_t(y)\coloneqq X_t(\nu^{1/2}y)$ for $0\le t\le T$ and $y\in \nu^{-1/2}\Lambda$ solves
	\begin{equation*}
		\D Y_t = \Updelta Y_t\,\D t + F(Y_t)\,\D t + \nu^{-1/4}\sigma\,\D\bar{W}_t,\quad 0\le t\le T,\quad Y_0(y)= X_0(\nu^{1/2}y),
	\end{equation*}
	on $\nu^{-1/2}\Lambda$ with space-time white noise $\D\bar{W}_t/\D t$ on $\nu^{-1/2}\Lambda$. Note that $\sigma=\nu^{1/4}$ is such that the noise level of $Y_t$ is equal to one. The diffusivity level of $Y_t$ is one, whereas its domain $\nu^{-1/2}\Lambda$ grows. Thus, we can expect (spatial) ergodicity of $y\mapsto X_t(y)$ for fixed $0<t\le T$. See also Figure \ref{fig:Examp} (middle) for a visual indication of viewing $X_t$ as a spatially squeezed version of $Y_t$, which lives on a larger (spatial) domain.
\end{remark}

In the case of a general dispersion operator $B$ we prove consistency of the estimator for $f(x_0)$ and specify the convergence rate. The following assumption additionally allows for a central limit theorem (Corollary \ref{cor:CLT}), and a proof of the minimax-optimality of the convergence rate (Proposition \ref{prop:LowerBound}).

\begin{assumption}[\textcolor{black}{Noise covariance function}]\label{assump:CLT}
There exist constants ${0<\smash{\underbar{\Sigma}}\le\bar{\Sigma}<\infty}$ and a measurable function ${\Sigma\colon \mathbb{R}\to [\smash{\underbar{\Sigma}},\bar{\Sigma}]}$ such that  $[B \xi](y) = \Sigma(y)\xi(y)$ for all $\xi\in\mathcal{H}$ and $y\in\Lambda$.
\end{assumption}

Assumptions \ref{assump:WellPosedness} \hyperref[assump:WellPosedness]{(well-posedness)} and \ref{assump:key} \hyperref[assump:key]{(noise-scaling)} are required throughout this paper, whereas Assumption \ref{assump:CLT} \hyperref[assump:CLT]{(noise covariance function)} is only required for the central limit theorem and the (statistical) lower bounds.

	\section{Estimator and statistical main results}\label{sec:Balanced_MainResults}
	This section contains the statistical results for the $\nu\to 0$ and $\sigma=\sigma(\nu)\to 0$ asymptotics.
	
		\subsection{Definition of the estimator}\label{subsec:Balanced_estimator}
		
		Recall that a canonical choice for estimating $f(x_0)$ is the Nadaraya-Watson estimator
		\begin{equation}
			\hat{f}(x_0)_{h}^{\operatorname{NW}} = \frac{\int_0^T \iprod*{K_h(X_t)}{\D X_t - \nu A_t X_t\,\D t}}{\int_0^T \int_\Lambda K_h(X_t(y))\,\D y\D t},\label{eq:Estimator_Raw}
		\end{equation}
		which locally averages the information about $f$ around $x_0$ using weights given by the kernel $K_h$. There are two challenges:
		\begin{enumerate}[(a)]
			\item\label{num:challenge_a} If $(X_t)_{t\ge 0}$ was solving an SODE, the (ergodic) density $p_{\nu,t}$ of $X_t$ would be more regular than the drift $f$ and by choosing a kernel of sufficiently high order, the estimator from \eqref{eq:Estimator_Raw} benefits from arbitrary regularity of $f$ (see Proposition 1.65 of \citet{Kutoyants2013}). As mentioned in Section \ref{sec:Balances_Intro}, no results for SPDEs seem available in the literature so far, which ensure that $p_{\nu,t,y}$ has at least the regularity of $f$. Consequently, it is unclear how to reduce the bias of $\hat{f}(x_0)_{h}^{\operatorname{NW}}$. This challenge is tackled by combining two estimators of the type \eqref{eq:Estimator_Raw} with data-driven weights, similarly to local polynomial estimators. Since no precise control of the derivatives of $p_{\nu,t,y}$ is available, we require a positive kernel and restrict the regularity of $f$ to $1\le \beta\le 2$.
			\item\label{num:challenge_b} Since $X_t\not\in\operatorname{dom}(A_t)$, the term $\D X_t- \nu A_t X_t\,\D t$ is not well-defined in the (analytically) strong sense. We use an (analytically) weak formulation, as explained in Remark \ref{rmk:WellDefinednessEstimator} (below).
		\end{enumerate}
		To address (\ref{num:challenge_a}), we use a (random) kernel $K_{h,\sigma}$, which is a combination of two Nadaraya-Watson kernels $K_{-,h}$ and $K_{+,h}$ with data-driven weights. The weights ensure that
		\begin{enumerate}[(I)]
			\item\label{num_RmkIntegrate_a} $K_{h,\sigma}\ge 0$,
			\item $\int_0^T\int_{\Gamma} K_{h,\sigma}(X_t(y))\,\D y\D t= 1$ and\label{num_RmkIntegrate_b}
			\item $\int_0^T\int_{\Gamma} K_{h,\sigma}(X_t(y))(X_t(y)-x_0)\,\D y\D t =0.$\label{num_RmkIntegrate_c}
		\end{enumerate}
		The proof of Theorem \ref{thm:Nonparametric} (below) reveals that Property (\ref{num_RmkIntegrate_c}) reduces the bias of the estimator depending on the regularity of $f$, similarly to local polynomial estimators. To this end, let $K_-\colon \mathbb{R}\to\mathbb{R}_{\ge 0}$ and $K_+\colon \mathbb{R}\to\mathbb{R}_{\ge 0}$ be two functions satisfying the following mild condition.
		\begin{assumption}[\textcolor{black}{Kernel}]\label{assump:Kernel}
			Let $K_-\colon \mathbb{R}\to\mathbb{R}_{\ge 0}$ and $K_+\colon \mathbb{R}\to\mathbb{R}_{\ge 0}$ be (non-zero) Lipschitz-continuous functions, which have (almost) disjoint supports $\operatorname{supp}(K_-)\subset [-1,0]$ and $\operatorname{supp}(K_+)\subset [0,1]$.
		\end{assumption}
		For any function $g\colon \mathbb{R}\to\mathbb{R}$ and a bandwidth $h>0$ define the localisation
		\begin{equation}
			g_h(x) \coloneqq g\left(\frac{x-x_0}{h}\right),\quad x\in \mathbb{R},\label{eq:Rescaling}
		\end{equation}
		and write $K_{-,h}\coloneqq (K_-)_h$, $K_{+,h}\coloneqq (K_+)_h$.
		\begin{remark}[On the scaling \eqref{eq:Rescaling}]\label{rmk:Scaling}
			The scaling ensures that
			\begin{align*}
				\norm{g_h}_{L^1(\mathbb{R})}&=h\norm{g}_{L^1(\mathbb{R})},\quad 
				\norm{g_h^2}_{L^1(\mathbb{R})}=h\norm{g^2}_{L^1(\mathbb{R})}=h\norm{g}_{L^2(\mathbb{R})}^2,\\
				\norm{g_h'}_{L^1(\mathbb{R})}&\coloneqq \norm{(g_h)'}_{L^1(\mathbb{R})} = \norm{g'}_{L^1(\mathbb{R})},\quad
				\norm{(g_h^2)'}_{L^1(\mathbb{R})}=\norm{(g^2)'}_{L^1(\mathbb{R})},
			\end{align*}
			whenever these norms are finite. Estimators of the type \eqref{eq:Estimator_Raw} are invariant under rescaling the kernel, for example normalising in $L^1(\mathbb{R})$, and the choice in \eqref{eq:Rescaling} is merely for notational convenience.
		\end{remark}
		
		Fix a bandwidth $h>0$ and introduce the non-negative random quantities
		\begin{align*}
			\mathcal{T}_{h,\sigma}^{\smash{\pm,1}}&\coloneqq \int_0^T \int_{\Gamma} K_{\pm,h}(X_t(y))\,\D y\D t,\quad \mathcal{T}_{h,\sigma}^{\smash{\pm,2}}\coloneqq \pm\int_0^T \int_{\Gamma} K_{\pm,h}(X_t(y))(X_t(y)-x_0)\,\D y\D t,\\
				\mathcal{J}_{h,\sigma}&\coloneqq \mathcal{T}_{h,\sigma}^{\smash{-,1}}\mathcal{T}_{h,\sigma}^{\smash{+,2}}+\mathcal{T}_{h,\sigma}^{\smash{+,1}}\mathcal{T}_{h,\sigma}^{\smash{-,2}}
		\end{align*}
		and define the random kernel
		\begin{equation}
			K_{h,\sigma}(x)\coloneqq \frac{\mathcal{T}_{h,\sigma}^{\smash{+,2}}K_{-,h}(x) + \mathcal{T}_{h,\sigma}^{\smash{-,2}}K_{+,h}(x)}{\mathcal{J}_{h,\sigma}},\quad x\in\mathbb{R},\quad h>0.\label{eq:Kernel}
		\end{equation}
		Note that $\operatorname{supp}(K_{h,\sigma})\subset [x_0-h,x_0+h]$.
		\begin{remark}[On the dependence on $\nu$ and $\sigma$]
			Note that $\mathcal{T}_{h,\sigma}^{\pm,1}$, $\mathcal{T}_{h,\sigma}^{\pm,2}$, $\mathcal{J}_{h,\sigma}$ and $K_{h,\sigma}$ depend on the diffusivity $\nu$ through the law of $X$. We will see in Corollary \ref{cor:OptimalRate} that by indexing the statistical quantities by the noise level $\sigma=\sigma(\nu)$ instead of the diffusivity level $\nu$, we obtain the classical non-parametric rates.
		\end{remark}
		\begin{definition}
		With $K_{h,\sigma}$ from \eqref{eq:Kernel} the estimator $\hat{f}(x_0)_{h,\sigma}$ for $f(x_0)$ is defined as
		\begin{equation}
			\hat{f}(x_0)_{h,\sigma}\coloneqq \int_0^T \iprod*{\indicator_{\Gamma} K_{h,\sigma}(X_t)}{\D X_t-\nu A_t X_t\,\D t},\quad h>0.\label{eq:Estimator}
		\end{equation}
		\end{definition}
		
		\begin{remark}[On the definition of the estimator \eqref{eq:Estimator}]\label{rmk:WellDefinednessEstimator}
			To define ${\D X_t- \nu A_t X_t\,\D t}$, consider any complete orthonormal system $(e_k)_{k\in\mathbb{N}}\subset \operatorname{dom}(A_0^\ast)$ of $L^2(\Gamma)\subset \mathcal{H}$, extended to functions on $\Lambda$ by zero. Let
		\begin{align*}
			\int_0^T \iprod*{\indicator_{\Gamma}K_{h,\sigma}(X_t)}{\D X_t-\nu A_t X_t\,\D t}\hspace{-10em}&\\
			&\coloneqq\sum_{k\in\mathbb{N}}\int_0^T\iprod*{\indicator_{\Gamma}K_{h,\sigma}(X_t)}{e_k}\left(\iprod*{e_k}{\D X_t} - \iprod{\nu A_t^\ast e_k}{X_t}\,\D t\right),
			\end{align*}
			which is well-defined by \eqref{eq:WeakSolution} and satisfies under the data-generating law $\prob{}$ the representation
			\begin{align}
			\begin{split}
			\int_0^T \iprod*{\indicator_{\Gamma}K_{h,\sigma}(X_t)}{\D X_t-\nu A_t X_t\,\D t}\hspace{-10em}&\\
			&\overset{\prob*{}}{=}\int_0^T \iprod*{\indicator_{\Gamma}K_{h,\sigma}(X_t)}{F(X_t)}\,\D t +
			 \sigma\int_0^T \iprod*{\indicator_{\Gamma}K_{h,\sigma}(X_t)}{\D W_t}.
			 \end{split}\label{eq:Dynamics_pluggedin}
		\end{align}
	\end{remark}
	\begin{remark}[On the decomposition of the estimation error]\label{rmk:ErrorDecomposition}
		In view of \eqref{eq:Dynamics_pluggedin}, it is natural to introduce the (stochastic) approximation error
		\begin{equation*}
B_{h,\sigma}\coloneqq \int_0^T\iprod*{\indicator_{\Gamma} K_{h,\sigma}(X_t)}{F(X_t)}\,\D t -f(x_0),
\end{equation*}
and the time martingales $\mathcal{M}_{h,\sigma}^{\smash{\pm}}$ with quadratic variations $\mathcal{I}_{h,\sigma}^{\smash{\pm}}$ by
\begin{equation*}
	\mathcal{M}_{h,\sigma}^{\smash{\pm}}\coloneqq \int_0^T\iprod*{\indicator_{\Gamma}K_{\pm,h}(X_t)}{\D W_t},\quad  \mathcal{I}_{h,\sigma}^{\smash{\pm}}\coloneqq \int_0^T\norm*{\indicator_{\Gamma}K_{\pm,h}(X_t)}_{\mathbb{H}}^2\,\D t.
	\end{equation*}
 With the short-hand notations
		\begin{equation*}
	A_{h,\sigma}^{\smash{\pm}}\coloneqq \frac{\mathcal{T}_{h,\sigma}^{\smash{\mp,2}}\EV{\mathcal{I}_{h,\sigma}^{\smash{\pm}}}^{1/2}}{\mathcal{J}_{h,\sigma}},\quad 
	Z_{h,\sigma}^{\smash{\pm}}\coloneqq \frac{\mathcal{M}_{h,\sigma}^{\smash{\pm}}}{\EV{\mathcal{I}_{h,\sigma}^{\smash{\pm}}}^{1/2}},
\end{equation*}
		the identity \eqref{eq:Dynamics_pluggedin} yields under $\prob{}$ the following decomposition of the estimation error:
		\begin{equation}
			\hat{f}(x_0)_{h,\sigma}-f(x_0) \overset{\prob{}}{=} B_{h,\sigma} + \sigma A_{h,\sigma}^{\smash{-}}Z_{h,\sigma}^{\smash{-}} + \sigma A_{h,\sigma}^{\smash{+}}Z_{h,\sigma}^{\smash{+}}.\label{eq:ErrorDecomposition}
\end{equation}
We refer to $B_{h,\sigma}$ as the \emph{approximation error} and to $\sigma A_{h,\sigma}^{\smash{-}}Z_{h,\sigma}^{\smash{-}} + \sigma A_{h,\sigma}^{\smash{+}}Z_{h,\sigma}^{\smash{+}}$ as the \emph{stochastic error}.
		\end{remark}

\pagebreak

		\begin{remark}[On knowing the diffusivity $\nu$]
			The estimator $\hat{f}(x_0)_{h,\sigma}$ relies on the knowledge of the diffusivity $\nu$. This can be justified since $\nu$ is known exactly given continuous observations \citep{Huebner1995,Altmeyer2020}. Consequently, given discrete observations at a sufficiently dense space-time grid, a plug-in estimator seems promising. A precise error analysis is non-trivial due to the loss of spatial smoothness as $\nu\to 0$ (compare Proposition 3.14 of \citet{gaudlitzEstimationReactionTerm2023} and Theorem 4.2 of \citet{Bibinger2020}), and is beyond the scope of this thesis.
		\end{remark}
		
			\begin{remark}[Implementation of $\hat{f}(x_0)_{h,\sigma}$ via (weighted) least squares]\label{rmk:PracticalImplementation}
			Given discrete data points $(X_{t_i}(y_k)| i=0,\dots,N, {k=0,\dots, M})$, the estimator $\hat{f}(x_0)_{h,\sigma}$ can be implemented via a (weighted) least squares approach: For bandwidth $h\sim \sigma^{2/3}$ chosen according to Corollary \ref{cor:OptimalRate} (below), define $\mathfrak{I}\coloneqq \Set{(i,k)\given X_{t_i}(y_k)\in \operatorname{supp}(K_{h,\sigma})}$. Estimate the weights $\mathcal{T}_{h,\sigma}^{\pm,\smash{\Set{1,2}}}$ and $\mathcal{J}_{h,\sigma}$ by Riemann sums to obtain an estimated kernel function $\hat{K}_{h,\sigma}$ from \eqref{eq:Kernel}. For every $(i,k)\in\mathfrak{I}$ compute an estimate $\widehat{A_{t_i} X_{t_i}(y_k)}$ of $A_{t_i} X_{t_i}(y_k)$, for example using finite differences, and let $Y_{t_i}(y_k)\coloneqq (X_{t_{i+1}}(y_k)-X_{t_{i}}(y_k)) / (t_{i+1}-t_i)$. The estimator $\hat{f}(x_0)_{h,\sigma}$ is approximated by the solution to the weighted least squares problem
			\begin{equation}
				\argmin_{\zeta\in\mathbb{R}}\sum_{(i,k)\in\mathfrak{I}}\hat{K}_{h,\sigma}(X_{t_i}(y_k))[Y_{t_i}(y_k) - \nu \widehat{A_{t_i}X_{t_i}(y_k)} - \zeta]^2,\label{eq:OLS}
			\end{equation}
			i.e.\
			\begin{equation*}
				\hat{f}(x_0)_{h,\sigma} =\frac{\sum_{(i,k)\in\mathfrak{I}}\hat{K}_{h,\sigma}(X_{t_i}(y_k))[Y_{t_i}(y_k) - \nu \widehat{A_{t_i}X_{t_i}(y_k)}]}{\sum_{(i,k)\in\mathfrak{I}}\hat{K}_{h,\sigma}(X_{t_i}(y_k))}.
			\end{equation*}
			Joint estimation of $\nu$ and $f(x_0)$ can be performed by simultaneously minimizing over $\nu$ in \eqref{eq:OLS} and choosing $h$ by cross-validation. A precise control of the discretisation error is left for future research.
		\end{remark}

	\subsection{Main results}
		
	This subsection contains the consistency result (Theorem \ref{thm:Nonparametric}), the optimal convergence rate (Corollary \ref{cor:OptimalRate} and Proposition \ref{prop:LowerBound}), and the central limit theorem (Corollary \ref{cor:CLT}) for the estimator $\hat{f}(x_0)_{h,\sigma}$ of $f(x_0)$ from \eqref{eq:Estimator}. Recall that we use $\sigma$ instead of $\nu$ to index the statistical quantities and that $\sigma=\sigma(\nu)\to 0$ as specified by Assumption \ref{assump:key} \hyperref[assump:key]{(noise-scaling)}.

		\begin{theorem}\label{thm:Nonparametric}
			Fix $1\le \beta\le 2$ and $L>0$. Grant Assumptions \ref{assump:WellPosedness} \hyperref[assump:WellPosedness]{(well-posedness)}, \ref{assump:key} \hyperref[assump:key]{(noise-scaling)}, \ref{assump:Kernel} \hyperref[assump:Kernel]{(kernel)} and assume $\sigma=\smallo(h)$, then the estimation error of $\hat{f}(x_0)_{h,\sigma}$ from \eqref{eq:Estimator} satisfies
			\begin{equation*}
				\hat{f}(x_0)_{h,\sigma}-f(x_0) = \mathcal{O}(h^{\beta}) + \mathcal{O}_{\prob{}}(\sigma h^{-1/2}),
			\end{equation*}
			uniformly in $f\in \Sigma(\beta,L)$, as $\nu\to 0$ and $\sigma=\sigma(\nu)\to 0$. More precisely, we can decompose the estimation error as
			\begin{equation}
				\hat{f}(x_0)_{h,\sigma}-f(x_0) = B_{h,\sigma} + \sigma A_{h,\sigma}^{\smash{-}}Z_{h,\sigma}^{\smash{-}} + \sigma A_{h,\sigma}^{\smash{+}}Z_{h,\sigma}^{\smash{+}},\label{eq:ErrorDecomposition2}
			\end{equation}
			where
			\begin{enumerate}[(a)]
				\item $B_{h,\sigma}= \mathcal{O}(h^{\beta})$,\label{num:MainThm_f}
				\item $Z_{h,\sigma}^{\smash{-}}= \mathcal{O}_{\prob{}}(1)$ and $Z_{h,\sigma}^{\smash{+}}= \mathcal{O}_{\prob{}}(1)$,\label{num:MainThm_b}
				\item $A_{h,\sigma}^{\smash{-}} = \mathcal{O}_{\prob{}}(h^{-1/2})$ and $A_{h,\sigma}^{\smash{+}} = \mathcal{O}_{\prob{}}(h^{-1/2})$.\label{num:MainThm_c}
			\end{enumerate}
			If, additionally, Assumption \ref{assump:CLT} \hyperref[assump:CLT]{(noise covariance function)} is satisfied, then
			\begin{enumerate}[(a)]\setcounter{enumi}{3}
				\item $A_{h,\sigma}^{\smash{-}} \sim h^{-1/2} + \smallo_{\prob{}}(h^{-1/2})$ and $A_{h,\sigma}^{\smash{+}} \sim h^{-1/2} + \smallo_{\prob{}}(h^{-1/2})$,\label{num:MainThm_e}
				\item $Z_{h,\sigma}^{\smash{-}}$ and $Z_{h,\sigma}^{\smash{+}}$ are uncorrelated random variables,\label{num:MainThm_a}
				\item $(Z_{h,\sigma}^{\smash{-}},Z_{h,\sigma}^{\smash{+}})\transpose\xrightarrow{d} N(0,\operatorname{Id}_{2\times 2})$, where $\operatorname{Id}_{2\times 2}$ is the identity matrix in $\mathbb{R}^{2\times 2}$.\label{num:MainThm_d}
			\end{enumerate}
		\end{theorem}
		Before proving Theorem \ref{thm:Nonparametric}, we show the consequences of the spatial ergodicity at the concrete example of $\mathcal{T}_{h,\sigma}^{\smash{+,1}}=\int_0^T\int_\Gamma K_{+,h}(X_t(y))\,\D y\D t$, a key ingredient of the estimator $\hat{f}(x_0)_{h,\sigma}$ from \eqref{eq:Estimator}. In Proposition \ref{prop:DensityBounds} (below), we deduce upper and lower (Gaussian) bounds for the (Lebesgue-) density $p_{\nu,t,y}$ of $X_t(y)$, that are uniform in $0<\nu\le \bar{\nu}$. In the form provided by Corollary \ref{cor:DensityBounds} (below), these bounds imply
		\begin{equation*}
			\EV{\mathcal{T}_{h,\sigma}^{\smash{+,1}}} = \int_0^T\int_\Gamma \int_{\mathbb{R}}K_{+,h}(x)p_{\nu,t,y}(x)\,\D x\D y\D t
			\begin{cases}
			\lesssim \int_0^T \int_\Gamma \int_{\mathbb{R}}K_{+,h}(x)t^{-\alpha}\,\D x\D y\D t\lesssim \norm{K_{+,h}}_{L^1(\mathbb{R})},\\
			\gtrsim \int_\Delta^{t_0} \int_\Gamma \int_{\mathbb{R}}K_{+,h}(x)\,\D x\D y\D t\gtrsim \norm{K_{+,h}}_{L^1(\mathbb{R})},			
			\end{cases}
		\end{equation*}
		for $0<\Delta<t_0\le T$ as in Corollary \ref{cor:DensityBounds}. Consequently, we find 
		\begin{equation}
			\EV{\mathcal{T}_{h,\sigma}^{\smash{+,1}}}\sim \norm{K_{+,h}}_{L^1(\mathbb{R})} \sim h,\quad h>0,\label{eq:exampleExpectation}
		\end{equation}
		uniformly in $0<\nu\le\bar{\nu}$. Importantly, we will see in Proposition \ref{prop:VarianceBound} (below) that we can bound the variance of the spatial average
		\begin{equation}
			\var\bigg(\int_\Gamma K_{+,h}(X_t(y))\,\D y\bigg)\lesssim \sigma^2 \norm{K_{+,h}'}_{L^1(\mathbb{R})}^2\lesssim \sigma^2,\label{eq:exampleVariance}
		\end{equation}
		uniformly in $0\le t\le T$ and $h>0$. This concentration is due to the (asymptotic) spatial ergodicity of the process $y\to X_t(y)$ for $0<t\le T$ fixed. Since the bound \eqref{eq:exampleVariance} is uniform in $0\le t\le T$, it persists after integrating in time and carries over to $\mathcal{T}_{h,\sigma}^{\smash{+,1}}$ such that
		\begin{equation}
		\var(\mathcal{T}_{h,\sigma}^{\smash{+,1}})\lesssim  \sigma^2.\label{eq:exampleVariance_2}
		\end{equation}
		By combining \eqref{eq:exampleExpectation} with \eqref{eq:exampleVariance} and Chebychev's inequality, we obtain the convergence
		\begin{equation*}
			\frac{\mathcal{T}_{h,\sigma}^{\smash{+,1}}}{\EV{\mathcal{T}_{h,\sigma}^{\smash{+,1}}}}\xrightarrow{\prob{}}1,
		\end{equation*}
		as $\nu\to 0$, provided $\sigma= \sigma(\nu) = \smallo(h)$. This procedure is applied to the other weights $\mathcal{T}_{h,\sigma}^{\smash{+,2}}$, $\mathcal{T}_{h,\sigma}^{\smash{-,1}}$, $\mathcal{T}_{h,\sigma}^{\smash{-,2}}$, as well as to the quadratic variations $\mathcal{I}_{h,\sigma}^+$ and $\mathcal{I}_{h,\sigma}^-$ in Lemma \ref{lem:DetailsStochasticError} (below).
		\begin{proof}[Proof of Theorem \ref{thm:Nonparametric}]~\\%
	Recall that the decomposition of the estimation error from \eqref{eq:ErrorDecomposition} yields \eqref{eq:ErrorDecomposition2}.\\
	\textbf{Step 1} (Controlling the approximation error $B_{h,\sigma}$ and proving (\ref{num:MainThm_f})).\\
Fix some $h>0$. Using the properties of the kernel $K_{h,\sigma}$ from (\ref{num_RmkIntegrate_a})-(\ref{num_RmkIntegrate_c}), the approximation error can be controlled by standard arguments. An application of Property (\ref{num_RmkIntegrate_b}) shows
		\begin{equation*}
			\abs{B_{h,\sigma}}= \abs[\bigg]{\int_0^T \int_{\Gamma} K_{h,\sigma}(X_t(y))[f(X_t(y)) - f(x_0)]\,\D y\D t}.
		\end{equation*}
		In the case of $\beta=1$, note that $\operatorname{supp}(K_{h,\sigma})\subset [x_0-h,x_0+h]$. By Properties (\ref{num_RmkIntegrate_a}) and (\ref{num_RmkIntegrate_b}), we obtain the bound $\abs{B_{h,\sigma}}\le L h$. For $1<\beta\le 2$, the Property (\ref{num_RmkIntegrate_c}) yields with some $\tilde{x}_{t,y}$ between $X_t(y)$ and $x_0$ the bound
		\begin{equation*}
		\abs{B_{h,\sigma}}= \abs[\bigg]{\int_0^T \int_{\Gamma} K_{h,\sigma}(X_t(y))[f'(\tilde{x}_{t,y}) - f'(x_0)](X_t(y)-x_0)\,\D y\D t}
			\le  L h^{\beta}.
		\end{equation*}
	\textbf{Step 2} (Controlling the stochastic error and proving (\ref{num:MainThm_b})-(\ref{num:MainThm_a})).\\
The property (\ref{num:MainThm_b}) follows from Markov's inequality. Lemma \ref{lem:DetailsStochasticError} (\ref{num:AuxLemma_a})-(\ref{num:AuxLemma_c}) show that the property (\ref{num:MainThm_c}) is a consequence of the positivity of $K_{h,\sigma}$, the bounds on the marginal densities of $X$ and the concentration of $\mathcal{J}_{h,\sigma}$ as $\nu\to 0$. Lemma \ref{lem:DetailsStochasticError} (\ref{num:AuxLemma_a}), (\ref{num:AuxLemma_b}) and (\ref{num:AuxLemma_d}) imply that the property (\ref{num:MainThm_e}) is a consequence of the concentration due to (asymptotic) spatial ergodicity as $\nu \to 0$. The property (\ref{num:MainThm_a}) follows, provided Assumption \ref{assump:CLT} \hyperref[assump:CLT]{(noise covariance function)}, by the (almost) disjoint supports of the kernels $K_-$ and $K_+$.\\
	\textbf{Step 3} (Controlling the stochastic error and proving (\ref{num:MainThm_d})).
It is left to show (\ref{num:MainThm_d}) under Assumption \ref{assump:CLT} \hyperref[assump:CLT]{(noise covariance function)}. To this end, fix $a,b\in\mathbb{R}$ and write
\begin{equation*}
	aZ_{h,\sigma}^{\smash{-}} + bZ_{h,\sigma}^{\smash{+}}=\frac{\int_0^T\iprod{a\EV{\mathcal{I}_{h,\sigma}^{\smash{+}}}^{1/2}K_{-,h}(X_t) + b\EV{\mathcal{I}_{h,\sigma}^{\smash{-}}}^{1/2}K_{+,h}(X_t)}{\indicator_{\Gamma}\D W_t}}{\EV{\mathcal{I}_{h,\sigma}^{\smash{-}}}^{1/2}\EV{\mathcal{I}_{h,\sigma}^{\smash{+}}}^{1/2}}.
\end{equation*}
Due to the (almost) disjoint supports of $K_-$ and $K_+$, the quadratic variation of the martingale in the numerator is given by
\begin{align*}
	\int_0^T\norm[\big]{\Sigma \indicator_{\Gamma}\big( a\EV{\mathcal{I}_{h,\sigma}^{\smash{+}}}^{1/2}K_{-,h}(X_t) + b\EV{\mathcal{I}_{h,\sigma}^{\smash{-}}}^{1/2}K_{+,h}(X_t)\big)}^2\,\D t \hspace{-13em}&\\
		&= \int_0^T\norm[\big]{\Sigma \indicator_{\Gamma}a\EV{\mathcal{I}_{h,\sigma}^{\smash{+}}}^{1/2}K_{-,h}(X_t)}^2\,\D t + \int_0^T\norm[\big]{\Sigma \indicator_{\Gamma}b\EV{\mathcal{I}_{h,\sigma}^{\smash{-}}}^{1/2}K_{+,h}(X_t))}^2\,\D t\\
		&= \EV{\mathcal{I}_{h,\sigma}^{\smash{+}}}\mathcal{I}_{h,\sigma}^{\smash{-}}a^2  + \EV{\mathcal{I}_{h,\sigma}^{\smash{-}}}\mathcal{I}_{h,\sigma}^{\smash{+}} b^2.
\end{align*}
By Lemma \ref{lem:DetailsStochasticError} (\ref{num:AuxLemma_d}), the quadratic variation satisfies
\begin{equation*}
	\frac{\EV{\mathcal{I}_{h,\sigma}^{\smash{+}}}\mathcal{I}_{h,\sigma}^{\smash{-}}a^2  + \EV{\mathcal{I}_{h,\sigma}^{\smash{-}}}\mathcal{I}_{h,\sigma}^{\smash{+}} b^2}{\EV{\mathcal{I}_{h,\sigma}^{\smash{+}}}\EV{\mathcal{I}_{h,\sigma}^{\smash{-}}}(a^2+b^2)}\xrightarrow{\prob{}}1
\end{equation*}
as $\nu\to 0$ due to the asymptotic spatial ergodicity. An application of a martingale central limit theorem (Theorem 5.5.4 of \citet{Liptser1989}) yields $aZ_{h,\sigma}^{\smash{-}} + bZ_{h,\sigma}^{\smash{+}}\xrightarrow{d} N(0,a^2+b^2)$ and concludes the proof of (\ref{num:MainThm_d}).
		\end{proof}

		\begin{corollary}\label{cor:OptimalRate}
			Grant Assumptions \ref{assump:WellPosedness} \hyperref[assump:WellPosedness]{(well-posedness)}, \ref{assump:key} \hyperref[assump:key]{(noise-scaling)}, \ref{assump:Kernel} \hyperref[assump:Kernel]{(kernel)}, fix $1\le\beta\le 2$ and $L>0$. Assume that $h\sim \sigma^{2/(1+2\beta)}$, then
			\begin{equation*}
				\hat{f}(x_0)_{h,\sigma}-f(x_0) = \mathcal{O}_{\prob{}}((\sigma^2)^{\smash{\beta/(1+2\beta)}})
			\end{equation*}
			uniformly in $f\in \Sigma(\beta,L)$ as $\nu\to 0$.
		\end{corollary}
		\begin{proof}
			The claim follows from Theorem \ref{thm:Nonparametric}.
		\end{proof}

		\begin{example}[Continued]
			For Examples \ref{examp:main} (\ref{num_Examp_a})-(\ref{num_Examp_c}), the RMSE-optimised rates are given as follows.
				\begin{enumerate}[(a)]
						\item Stochastic heat equation with space-time white noise: $\sigma^{\smash{2\beta/(1+2\beta)}} = \nu^{\smash{\beta/(2+4\beta)}}$,
						\item Riesz kernel (coloured noise I): $\sigma^{\smash{2\beta/(1+2\beta)}} = \nu^{\smash{\rho\beta/(2+4\beta)}}$ and
						\item Spectral dispersion (coloured noise II): $\sigma^{\smash{2\beta/(1+2\beta)}} = \nu^{\smash{\beta(1-2\rho_2)/(\rho_1(1+2\beta))}}$.
						\item The rates in (a) and (b) are the same for Neumann instead of Dirichlet boundary conditions.
					\end{enumerate}
				\end{example}		
		
		\begin{corollary}\label{cor:CLT}
			Grant Assumptions \ref{assump:WellPosedness} \hyperref[assump:WellPosedness]{(well-posedness)}, \ref{assump:key} \hyperref[assump:key]{(noise-scaling)}, \ref{assump:CLT} \hyperref[assump:CLT]{(noise covariance function)}, \ref{assump:Kernel} \hyperref[assump:Kernel]{(kernel)}, fix $1\le\beta\le 2$ and $L>0$. Assume that $\sigma=\smallo(h)$ and $h=\smallo(\sigma^{2/(1+2\beta)})$. Then, for all $f\in \Sigma(\beta,L)$, we find
			\begin{equation*}
				\frac{\mathcal{J}_{h,\sigma}}{\sigma\sqrt{(\mathcal{T}_{h,\sigma}^{\smash{+,2}})^2\mathcal{I}_{h,\sigma}^{\smash{-}}+(\mathcal{T}_{h,\sigma}^{\smash{-,2}})^2\mathcal{I}_{h,\sigma}^{\smash{+}}}}(\hat{f}(x_0)_{h,\sigma}-f(x_0)) \xrightarrow{d} N(0,1),
			\end{equation*}
			 as $\nu\to 0$.
		\end{corollary}
		\begin{proof}
			See Subsection \ref{subsec:Proofs_Main}.
		\end{proof}
			\begin{corollary}\label{cor:Test}
				Grant Assumptions \ref{assump:WellPosedness} \hyperref[assump:WellPosedness]{(well-posedness)}, \ref{assump:key} \hyperref[assump:key]{(noise-scaling)}, \ref{assump:CLT} \hyperref[assump:CLT]{(noise covariance function)}, \ref{assump:Kernel} \hyperref[assump:Kernel]{(kernel)}, fix $1\le\beta\le 2$ and $L>0$. Assume that $\sigma=\smallo(h)$ and $h=\smallo(\sigma^{2/(1+2\beta)})$. Let $f\in\Sigma(\beta,L)$, $0<\bar{\alpha}<1$ and denote by $q_{1-\bar{\alpha}/2}$ the $(1-\bar{\alpha}/2)$-standard normal quantile.
					\begin{enumerate}[(a)]
					\item Then 
					\begin{align*}
						\mathcal{A}_{h,\sigma}^{\bar{\alpha}}\coloneqq \bigg[\hat{f}(x_0)_{h,\sigma}&-\sigma \sqrt{(\mathcal{T}_{h,\sigma}^{\smash{+,2}})^2\mathcal{I}_{h,\sigma}^{\smash{-}}+(\mathcal{T}_{h,\sigma}^{\smash{-,2}})^2\mathcal{I}_{h,\sigma}^{\smash{+}}}q_{1-\bar{\alpha}/2}/\mathcal{J}_{h,\sigma},\\
						\hat{f}(x_0)_{h,\sigma}&+\sigma \sqrt{(\mathcal{T}_{h,\sigma}^{\smash{+,2}})^2\mathcal{I}_{h,\sigma}^{\smash{-}}+(\mathcal{T}_{h,\sigma}^{\smash{-,2}})^2\mathcal{I}_{h,\sigma}^{\smash{+}}}q_{1-\bar{\alpha}/2}/\mathcal{J}_{h,\sigma}\bigg]
					\end{align*}
					are asymptotic $(1-\bar{\alpha})$-confidence intervals.	
					\item For fixed $\zeta\in\mathbb{R}$ introduce the test statistic
					\begin{equation*}
						\mathcal{T}_{h,\sigma}\coloneqq \frac{\mathcal{J}_{h,\sigma}}{\sigma\sqrt{(\mathcal{T}_{h,\sigma}^{\smash{+,2}})^2\mathcal{I}_{h,\sigma}^{\smash{-}}+(\mathcal{T}_{h,\sigma}^{\smash{-,2}})^2\mathcal{I}_{h,\sigma}^{\smash{+}}}}(\hat{f}(x_0)_{h,\sigma}-\zeta).
					\end{equation*}
					Then the test of $H_0\colon f(x_0) = \zeta$ versus $H_1\colon f(x_0)\neq \zeta$ defined by
					\begin{equation*}
						\Psi\coloneqq \indicator(\mathcal{T}_{h,\sigma}\not\in  [-q_{1-\bar{\alpha}/2},q_{1-\bar{\alpha}/2}])
					\end{equation*}
					is asymptotically of level $\bar{\alpha}$.
			\end{enumerate}					 
			\end{corollary}
				\begin{proof}
					The two claims follow from Corollary \ref{cor:CLT}.
				\end{proof}

		We prove (statistical) lower bounds in the case of Assumption \ref{assump:CLT} \hyperref[assump:CLT]{(noise covariance function)} and if $\Gamma=\Lambda$.

			\begin{proposition}\label{prop:LowerBound}
			Grant Assumptions \ref{assump:WellPosedness} \hyperref[assump:WellPosedness]{(well-posedness)}, \ref{assump:key} \hyperref[assump:key]{(noise-scaling)} with $\Gamma=\Lambda$,  Assumption \ref{assump:CLT} \hyperref[assump:CLT]{(noise covariance function)} and fix $\beta\ge 1$. Then
			\begin{equation*}
				\liminf_{\sigma\to 0}\inf_{T_\sigma(x_0)}\sup_{f\in \Sigma(\beta,L)} \prob*[f]{\abs{T_\sigma(x_0)-f(x_0)}\ge (\sigma^2)^{\beta/(1+2\beta)}}\ge C,
			\end{equation*}
			where $\inf_{T_\sigma(x_0)}$ denotes the infimum over all estimators $T_\sigma(x_0)$ for $f(x_0)$ based on observing $(X_t(y)|t\in[0,T],y\in\Lambda)$ solving the semi-linear SPDE \eqref{eq:SPDE1} with diffusivity $\nu>0$ and noise level $\sigma=\sigma(\nu)$. The constant $C>0$ depends on $L$, $\smash{\underbar{\Sigma}}$, $T$, $\abs{\Lambda}$ and $p_{\max}$ from Lemma \ref{lem:BoundsforExpectations}.
			\end{proposition}
			\begin{proof}
				See Subsection \ref{subsec:Proofs_Main}.
			\end{proof}
			
			\subsection{A numerical example}\label{subsec:Balanced_Numerics}
			
			We simulate the semi-linear stochastic heat equation with space-time white noise (Example \ref{examp:main} (\ref{num_Examp_a})), i.e.\
			\begin{equation*}
				\D X_t = \nu \Updelta X_t\,\D t  + F(X_t)\,\D t + \nu^{1/4}\,\D W_t,\quad 0\le t\le T,\quad X_0\equiv 0,
			\end{equation*}
			on $\Lambda=(0,1)$ with Dirichlet boundary conditions and final time $T=1$. We choose the well-known Allen-Cahn phase field model with stable points $\pm 3$. Since $f$ is required to have globally bounded Lipschitz norm, we set 
			\begin{equation}
			f(x) = \begin{cases}
			-x+10^3-10^2,& x\le -10,\\
			-(x^3-9x),& -10<x<10,\\
			-x-10^3+10^2,&x\ge 10.
			\end{cases}\label{eq:Simulation_f}
				\end{equation}
				The simulation is performed using a semi-implicit Euler scheme with a finite-difference approximation of $\Updelta$ according to Algorithm 10.8 of \citet{Lord2014}. The SPDE is discretised to a space-time mesh with $500$ spatial and $500^2$ temporal points satisfying the Courant–Friedrichs–Lewy (CFL) condition \citep{Lord2014}. Typical realisations of $X$ for high and low diffusivity $\nu\in\Set{0.1,0.001}$ are displayed in Figure \ref{fig:Examp}. At high diffusivity $\nu=0.1$ only one stable point at $+3$ is visible, whereas at low diffusivity both stable points $\pm 3$ appear. The estimator $\hat{f}(x_0)_{h,\sigma}$ uses local information around $x_0$. Figure \ref{fig:Examp} (right) shows those space-time points, which are within the range of $\pm 0.5$ of $x_0=1$.
			
			We implement the estimator $\hat{f}(x_0)_{h,\sigma}$ according to Remark \ref{rmk:PracticalImplementation} with known diffusivity $\nu>0$. Theorem \ref{thm:Nonparametric} ensures that $\hat{f}(x_0)_{h,\sigma}$ is a consistent estimator  of $f(x_0)$ at every $x_0\in\mathbb{R}$. To better understand the performance of $\hat{f}(x_0)_{h,\sigma}$ depending on $x_0$, we perform a Monte-Carlo simulation with $10^4$ runs. The results are displayed in Figure \ref{fig:Performance}. The figure on the left shows that, as expected, the estimator concentrates around the true function $f$. Interestingly, the figure on the right reveals that the spread of $\hat{f}(x_0)_{h,\sigma}$ is the smallest if $x_0$ is a stationary point of the SPDE ($\pm 3,0$). The interpretation for the comparatively small variance at $\pm3$ is clear, since these are the stable points. The small spread at $x_0=0$ is due to the initial condition $X_0\equiv 0$. This behaviour can be rigorously understood by further exploring the dependency of the density $p_{\nu,t,y}(x)$ on $x$ and on the initial condition beyond Proposition \ref{prop:DensityBounds}.
			
			\begin{figure}
			\includegraphics[width =.49 \textwidth]{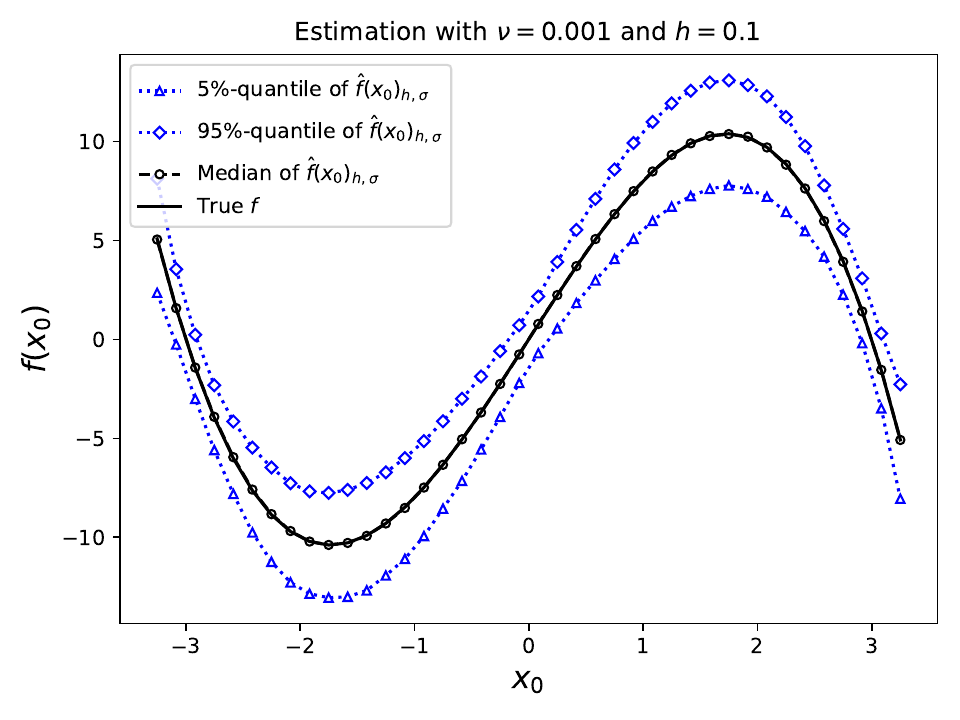}
			\includegraphics[width =.49 \textwidth]{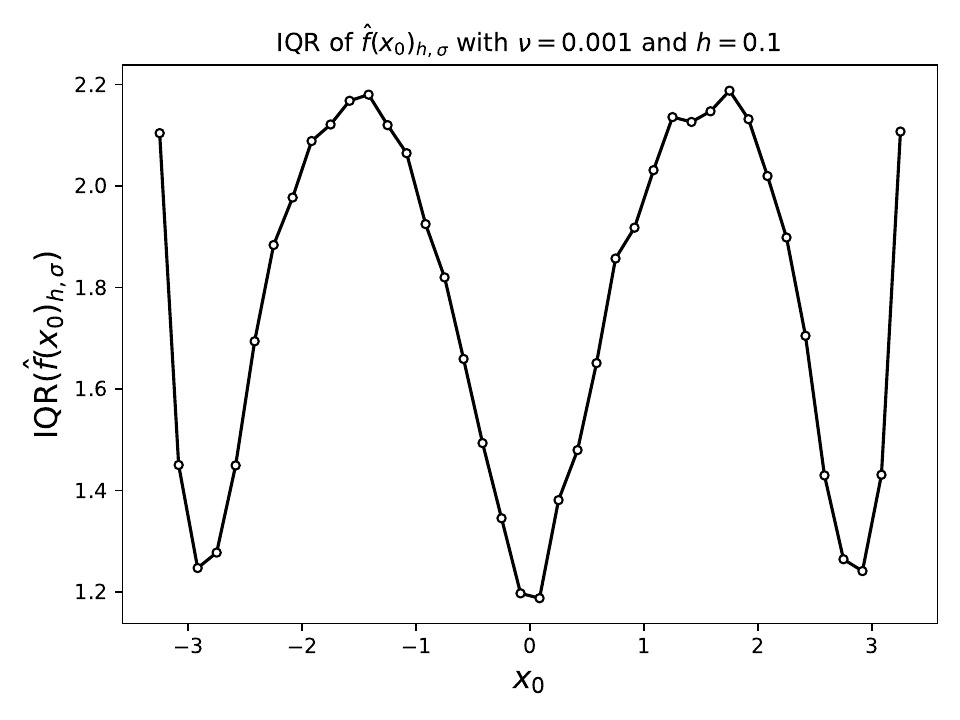}
			\caption{Monte Carlo simulation with $10^4$ runs of the performance of the estimator $\hat{f}(x_0)_{h,\sigma}$ at $\nu=0.001$ and $h=0.1$ with reaction function $f$ from \eqref{eq:Simulation_f}. Left: Median, $5\%$- and $95\%$-quantiles of $\hat{f}(x_0)_{h,\sigma}$. Right: Interquartile range (IQR, difference between $75\%$- and $25\%$-quantiles) of $\hat{f}(x_0)_{h,\sigma}$.}
			\label{fig:Performance}
			\end{figure}
		
		\section{Controlling the marginal densities and spatial ergodicity}\label{sec:Density}
		
		The reason for the convergence of the estimator $\hat{f}(x_0)_{h,\sigma}$ from \eqref{eq:Estimator} to $f(x_0)$ is the concentration of functionals $\mathcal{G}_{g,t,\sigma}\coloneqq \int_\Lambda g(X_t(y))\,\D y$ around their expectations due to the (asymptotic) spatial ergodicity of the non-Markovian process $y\mapsto X_t(y)$ for fixed time $0\le t\le T$. In this section, we introduce a novel methodology to prove this concentration, which builds upon exploiting the conditioning in the Clark-Ocone formula. The first step is to bound the (marginal) densities of $X_t(y)$ and to note that the upper bound does not depend on the (deterministic) initial condition.
	
	\subsection{The marginal densities}\label{subsec:Density}
		
		This subsection contains the upper and lower bounds on the marginal density $p_{\nu,\smash{\underbar{t}},t,y}$ (Corollary \ref{cor:DensityBounds}). We apply the methodology developed by \citet{nourdinDensityFormulaConcentration2009} to the setting of Assumptions \ref{assump:WellPosedness} \hyperref[assump:WellPosedness]{(well-posedness)} and \ref{assump:key} \hyperref[assump:key]{(noise-scaling)} and extend the results of \citet{nualartGaussianDensityEstimates2009a, nualartOptimalGaussianDensity2011} beyond the Laplacian case, to more general domains and spatial covariance structures of the noise process $(W_t)_{t\in[0,T]}$. Recall that for a starting time $0\le \smash{\underbar{t}} <T$ and a deterministic initial condition $\xi\in C(\Lambda)$ we denote by $\prob[\smash{(\smash{\underbar{t}},\xi)}]{}$ the law of $(X_t)_{t\in [0,T-\smash{\underbar{t}}]}$ started at $(\smash{\underbar{t}},\xi)$.

		\begin{proposition}\label{prop:DensityBounds}
			Grant Assumptions \ref{assump:WellPosedness} \hyperref[assump:WellPosedness]{(well-posedness)} and \ref{assump:key} \hyperref[assump:key]{(noise-scaling)}. Then there exist a time $0<t_0\le T$ defined in \eqref{eq:lower_t_0} and constants $0<c_1\le C_1<\infty$, $0<c_2\le C_2<\infty$, depending only on $\underbar{C}$, $\bar{C}$, $C_0$, $\norm{f'}_\infty$, $\alpha$ and $T$, such that for all diffusivity levels $0<\nu\le\bar{\nu}$, starting times $0\le \smash{\underbar{t}}<T$, deterministic initial conditions $\xi\in C(\Lambda)$, locations $y\in\Gamma$ and time points $0<t\le T-\smash{\underbar{t}}$ the Lebesgue-density $p_{\nu,\smash{\underbar{t}},t,y}$ of $X_{t}(y)$ under $\prob[\smash{(\smash{\underbar{t}},\xi)}]{}$ exists and satisfies the bound
			\begin{equation*}
				p_{\nu,\smash{\underbar{t}},t,y}\left(x-\EV[\smash{(\smash{\underbar{t}},\xi)}]{X_{t}(y)}\right)\le C_1 t^{-\alpha/2}\exp\left(-\frac{x^2}{2C_2t^{\alpha}}\right),\quad x\in\mathbb{R}.
			\end{equation*}
			If, moreover, $0<t\le t_0$, then
			\begin{equation*}
				p_{\nu,\smash{\underbar{t}},t,y}\left(x-\EV[\smash{(\smash{\underbar{t}},\xi)}]{X_{t}(y)}\right)\ge c_1 t^{-\alpha/2}\exp\left(-\frac{x^2}{2c_2t^{\alpha}}\right),\quad x\in\mathbb{R}.
			\end{equation*}
		\end{proposition}
		\begin{proof}
		See Subsection \ref{subsec:Proofs_Density_Appendix}.
	\end{proof}

		\begin{example}[Continued]
			Proposition \ref{prop:DensityBounds} implies the following Gaussian bounds for $x\in\mathbb{R}$, $0\le \smash{\underbar{t}}<T$, $\xi\in C(\Lambda)$, $y\in\Gamma$ and $0<t\le \min(t_0,T-\smash{\underbar{t}})$ with $t_0$ from \eqref{eq:lower_t_0} for Examples \ref{examp:main} (\ref{num_Examp_a})-(\ref{num_Examp_c}).
			\begin{enumerate}[(a)]
				\item In the case of the semi-linear stochastic heat equation driven by space-time white noise, Proposition \ref{prop:DensityBounds} implies the following  specification of Theorem 3.1 of \citet{nualartGaussianDensityEstimates2009a}:
				\begin{equation}
				t^{-1/4}\exp\left(-\frac{x^2}{2c_1t^{1/2}}\right) \lesssim p_{\nu,\smash{\underbar{t}},t,y}\left(x-\EV[\smash{(\smash{\underbar{t}},\xi)}]{X_{t}(y)}\right)\lesssim t^{-1/4}\exp\left(-\frac{x^2}{2C_1t^{1/2}}\right).\label{eq:Density_bounds_a}
			\end{equation}
				\item In the case of the Riesz kernel (coloured noise I), we obtain
				\begin{equation}
					t^{-1/2+\rho/4}\exp\left(-\frac{x^2}{2c_1t^{1-\rho/2}}\right) \lesssim p_{\nu,\smash{\underbar{t}},t,y}\left(x-\EV[\smash{(\smash{\underbar{t}},\xi)}]{X_{t}(y)}\right)\lesssim t^{-1/2+\rho/4}\exp\left(-\frac{x^2}{2C_1t^{1-\rho/2}}\right).\label{eq:Density_bounds_b}
				\end{equation}		
				\item In the case of spectral dispersion $B$ (coloured noise II), we obtain
				\begin{align*}
					p_{\nu,\smash{\underbar{t}},t,y}\left(x-\EV[\smash{(\smash{\underbar{t}},\xi)}]{X_{t}(y)}\right)&\gtrsim t^{-1/2-(2\rho_2-1)/(2\rho_1)}\exp\left(-\frac{x^2}{2c_1t^{1+(2\rho_2-1)/\rho_1}}\right),\\
					p_{\nu,\smash{\underbar{t}},t,y}\left(x-\EV[\smash{(\smash{\underbar{t}},\xi)}]{X_{t}(y)}\right)&\lesssim t^{-1/2-(2\rho_2-1)/(2\rho_1)}\exp\left(-\frac{x^2}{2C_1t^{1+(2\rho_2-1)/\rho_1}}\right).
				\end{align*}
				\item The density bounds \eqref{eq:Density_bounds_a} and \eqref{eq:Density_bounds_b} hold uniformly in $y\in\Lambda$ for Neumann boundary conditions.
			\end{enumerate}
		\end{example}
		\begin{corollary}\label{cor:DensityBounds}
			Grant Assumptions \ref{assump:WellPosedness} \hyperref[assump:WellPosedness]{(well-posedness)} and \ref{assump:key} \hyperref[assump:key]{(noise-scaling)}.
			\begin{enumerate}[(a)]
				\item\label{num:cordensitybounds_a} There exists a constant $0<p_{\max}<\infty$, depending only on $\underbar{C}$, $\bar{C}$, $C_0$, $\norm{f'}_\infty$, $\alpha$ and $T$, such that for all diffusivity levels $0<\nu\le\bar{\nu}$, starting times $0\le \smash{\underbar{t}} <T$, deterministic initial conditions $\xi\in C(\Lambda)$, locations $y\in\Gamma$ and time points $0<t \le T-\smash{\underbar{t}}$ the Lebesgue-density $p_{\nu,\smash{\underbar{t}},t,y}$ of $X_{t}(y)$ under $\prob[\smash{(\underbar{t},\xi)}]{}$ exists and satisfies
			\begin{equation*}
				p_{\nu,\smash{\underbar{t}},t,y}(x)\le p_{\max}t^{-\alpha/2}<\infty,\quad x\in\mathbb{R}.
			\end{equation*}
			\item\label{num:cordensitybounds_b} Consider the starting configuration $\smash{\underbar{t}}=0$ and $\xi\in C(\Lambda)$ deterministic. For every fixed bounded subset $\mathcal{N}\subset\mathbb{R}$ and $0<\Delta<t_0$ with $0<t_0\le T$ from \eqref{eq:lower_t_0} there exists a constant $0<p_{\min,\norm{\xi}_\infty,\mathcal{N},\Delta}<\infty$, depending on $\norm{\xi}_\infty$, $\mathcal{N}$, $\Delta$, $\underbar{C}$, $\bar{C}$, $C_0$, $\norm{f'}_\infty$, $\alpha$ and $T$, such that
			\begin{equation*}
				p_{\nu,t,y}(x) \ge p_{\min,\norm{\xi}_\infty,\mathcal{N}, \Delta}>0,\quad x\in\mathcal{N},
			\end{equation*}
			for all $0<\nu \le \smash{\bar{\nu}}$, $y\in\Gamma$ and $\Delta \le t\le t_0$.
			\end{enumerate}
		\end{corollary}
		\begin{proof}
			The upper bound (\ref{num:cordensitybounds_a}) follows directly from Proposition \ref{prop:DensityBounds}. For the lower bound (\ref{num:cordensitybounds_b}), it suffices to make sure that $\EV[\smash{(\underbar{t},\xi)}]{X_{t}(y)}$ remains bounded as $\nu\to 0$. This is achieved by applying Lemma \ref{lem:UniformBoudnednessSolution}.
		\end{proof}
		
\subsection{Concentration due to spatial ergodicity}\label{subsec:Concentration}

	The purpose of this subsection is to show that Assumptions \ref{assump:WellPosedness} \hyperref[assump:WellPosedness]{(well-posedness)} and \ref{assump:key} \hyperref[assump:key]{(noise-scaling)} ensure the concentration of functionals
	\begin{equation*}
		\mathcal{G}_{g,t,\sigma}\coloneqq \int_{\Gamma} g(X_t(y))\,\D y,\quad 0\le t\le T,
	\end{equation*}
	for suitable functions $g\colon\mathbb{R}\to\mathbb{R}$ around their respective means as $\nu\to 0$ (and thus $\sigma\to 0$), even if $\norm{g'}_\infty\to\infty$. The first building block is a control of the variance.
	\begin{proposition}\label{prop:VarianceBound}
	Grant Assumptions \ref{assump:WellPosedness} \hyperref[assump:WellPosedness]{(well-posedness)}, \ref{assump:key} \hyperref[assump:key]{(noise-scaling)} and fix a (globally) Lipschitz-continuous function $g\colon \mathbb{R}\to \mathbb{R}$ with almost everywhere existing derivative $g'\in L^\infty(\mathbb{R})$. Then
	\begin{equation*}
	\operatorname{Var}\left(\mathcal{G}_{g,t,\sigma}\right) = \sigma^2\mathcal{O}\left(\norm{g'}_{L^1(\mathbb{R})}^2\wedge\norm{g'}_{L^2(\mathbb{R})}^2\wedge\norm{g'}_\infty^2\right),\quad 0\le t\le T,\quad 0<\nu\le\bar{\nu}.
	\end{equation*}
	More precisely, with $0<p_{\max}<\infty$ from Corollary \ref{cor:DensityBounds}, the variance $\operatorname{Var}\left(\mathcal{G}_{g,t,\sigma}\right)$ is bounded by
	\begin{equation*}
		\sigma^2\abs{\Gamma}\norm{B}^2C_0^2t^{1-\alpha} \begin{cases}
			\norm{g'}_{L^1(\mathbb{R})}^2 p_{\max}^2(1-\alpha)^{-1}2(1+ e^{2\norm{f'}_\infty t}\norm{f'}_\infty^2 t^2) ,& g'\in L^1(\mathbb{R})\cap L^\infty(\mathbb{R}),\\
			\norm{g'}_{L^2(\mathbb{R})}^2t^{\alpha/2}  p_{\max}e^{2C_0\norm{f'}_\infty t} ,& g'\in L^2(\mathbb{R})\cap L^\infty(\mathbb{R}),\\
			\norm{g'}_{\infty}^2 t^\alpha  2(1+ e^{2\norm{f'}_\infty t}\norm{f'}_\infty^2 t^2) ,& g'\in L^\infty(\mathbb{R})
		\end{cases}
	\end{equation*}
	for all $0\le t\le T$ and $0<\nu\le\bar{\nu}$.
	\end{proposition}
	\begin{proof}
		See Subsection \ref{subsec:ProofsDensity}.
	\end{proof}
	
		\begin{remark}[On the connection to \citet{chenSpatialErgodicitySPDEs2021} and \citet{gaudlitzEstimationReactionTerm2023}]\label{rmk:Connection}
		 The novelty of Proposition \ref{prop:VarianceBound} is the $\mathcal{O}(\sigma^2 \norm{g'}_{L^1(\mathbb{R})}^2)$-bound for the variance. The bound of the order of $\mathcal{O}(\sigma^2 \norm{g'}_{L^2(\mathbb{R})}^2)$ follows by combining the results from Section 2.4.1 of \citet{gaudlitzEstimationReactionTerm2023} with the density estimates from Corollary \ref{cor:DensityBounds}. Bounding the variance of the order of $\mathcal{O}(\sigma^2 \norm{g'}_\infty^2)$ can be achieved similarly to Theorems 1.6 and 1.7 of \citet{chenSpatialErgodicitySPDEs2021}. As we consider localised functions $g_h$ with $\norm{g_h'}_{L^1(\mathbb{R})}\sim 1$, $\norm{g_h'}_{L^2(\mathbb{R})}^2\sim h^{-1}$ and $\norm{g_h'}_\infty^2\sim h^{-2}$, Proposition \ref{prop:VarianceBound} yields a bound of the order of $\mathcal{O}(\sigma^2)$ and is significantly sharper than previous results. We conjecture that control on the derivatives of $p_{\nu,t,y}$ would allow for a bound of the order of $\mathcal{O}(h^2\sigma^2)$. 
	\end{remark}
	A key ingredient for the proof of the $\mathcal{O}(\sigma^2 \norm{g'}_{L^1(\mathbb{R})}^2)$-bound for the variance in Proposition \ref{prop:VarianceBound} is the well-known Clark-Ocone formula (Proposition 6.3 of \citet{chenSpatialErgodicitySPDEs2021}), which implies
	\begin{equation}
			\mathcal{G}_{g,t,\sigma} = \EV*{\mathcal{G}_{g,t,\sigma}} + \int_0^t \int_\Lambda \EV*{\mathcal{D}_{\tau,z}\mathcal{G}_{g,t,\sigma}\given \mathcal{F}_{\tau}}\,\mathcal{W}(\D z,\D \tau),\quad 0\le t\le T,\label{eq:Clark_ocone}
		\end{equation}
		almost surely. In contrast to the approaches of \citet{chenSpatialErgodicitySPDEs2021} and \citet{gaudlitzEstimationReactionTerm2023}, which are based on the Poincaré inequality (compare Equation (1.15) of \citet{chenSpatialErgodicitySPDEs2021}), we make use of the averaging effect of the conditional expectation in \eqref{eq:Clark_ocone}. 
		The following lemma shows that a bound for the conditional expectation $\abs{\EV{\phi(X_t(y))\given \mathcal{F}_\tau}}$, uniformly in $y\in\Gamma$, suffices to control the conditional expectation in the Clark-Ocone formula \eqref{eq:Clark_ocone}.
	\begin{lemma}\label{lem:CondExpectation_MalliavinIntegration}
		Grant Assumptions \ref{assump:WellPosedness} \hyperref[assump:WellPosedness]{(well-posedness)} and \ref{assump:key} \hyperref[assump:key]{(noise-scaling)}. Fix $0<\nu\le\bar{\nu}$, $0<t\le T$ and take any $\phi\in L^\infty(\mathbb{R})$. Assume that there exists a family of random variables $(\kappa(t,\tau))_{\tau\in [0,t)}\subset \mathbb{R}_{\ge0}$, jointly measurable as a function $(\omega,\tau)\mapsto \kappa(t,\tau)(\omega)$ and potentially depending on $\nu$, $t$ and $\phi$, such that
		\begin{equation*}
			\sup_{y\in\Gamma}\abs{\EV{\phi(X_t(y))\given \mathcal{F}_\tau}}\le \kappa(t,\tau),\quad  0\le \tau<t,
		\end{equation*}
		almost surely. Then
		\begin{align*}
		\EV[][][\bigg]{\int_0^t \norm[\bigg]{\EV[][][\bigg]{\int_{\Gamma} \phi(X_t(y))\mathcal{D}_{\tau}X_t(y)\,\D y\given \mathcal{F}_\tau}}_{\mathbb{H}}^2\,\D\tau }\hspace{-5em}&\\
		&\le  2 \sigma^2 C_0^2 \norm{B}^2\abs{\Gamma}\int_0^t \EV{\kappa(t,\tau)^2}\,\D\tau \left(1+e^{2\norm{f'}_\infty t}\norm{f'}_\infty^2 t^2\right).
		\end{align*}
	\end{lemma}
	\begin{remark}
		If the map $(\omega,\tau)\mapsto \sup_{y\in\Gamma}\abs{\EV{\phi(X_t(y))\given \mathcal{F}_\tau}}$ is jointly measurable, for example if $\phi$ is continuous, then $\kappa$ could be chosen as $\kappa(t,\tau)\coloneq \sup_{y\in\Gamma}\abs{\EV{\phi(X_t(y))\given \mathcal{F}_\tau}}$ in Lemma \ref{lem:CondExpectation_MalliavinIntegration}.
	\end{remark}
	\begin{proof}
		See Subsection \ref{subsec:Proofs_Concentration}.
	\end{proof}
	The next Lemma shows that by combining the upper density bound from Corollary \ref{cor:DensityBounds} with the Markovianity of $X$ we obtain a tight upper bound $\kappa(t,\tau)$ for Lemma \ref{lem:CondExpectation_MalliavinIntegration}, which does not depend on the diffusivity $0<\nu\le\bar{\nu}$.
	\pagebreak
	\begin{lemma}\label{lem:BoundConditionalExpectation}
		Grant Assumptions \ref{assump:WellPosedness} \hyperref[assump:WellPosedness]{(well-posedness)}, \ref{assump:key} \hyperref[assump:key]{(noise-scaling)} and consider the constant \linebreak$0<p_{\max}<\infty$, depending only on $\underbar{C}$, $\bar{C}$, $C_0$, $\norm{f'}_\infty$, $\alpha$ and $T$, from Corollary \ref{cor:DensityBounds}. For all $0<\nu\le\bar{\nu}$, $0<t\le T$ and $\phi\in L^1(\mathbb{R})\cup L^\infty(\mathbb{R})$ we have
		\begin{equation*}
			\sup_{y\in\Gamma}\EV*{\abs*{\phi(X_t(y))}\given \mathcal{F}_\tau}\le \norm{\phi}_{L^1(\mathbb{R})} p_{\max}(t-\tau)^{-\alpha/2} \wedge \norm{\phi}_\infty,\quad 0\le\tau <t,
		\end{equation*}
		almost surely.
	\end{lemma}
		\begin{proof}
		See Subsection \ref{subsec:ProofsDensity}.
	\end{proof}
	The second building block for the concentration consists of controlling the expectation of the functional $\mathcal{G}_{h,t,\sigma}$ using the density bounds from Corollary \ref{cor:DensityBounds}.
		
		\begin{lemma}\label{lem:BoundsforExpectations}
				Grant Assumptions \ref{assump:WellPosedness} \hyperref[assump:WellPosedness]{(well-posedness)} and \ref{assump:key} \hyperref[assump:key]{(noise-scaling)}. Let $\Xi\subset\mathbb{R}$ be compact and take the constant $0<p_{\max}<\infty$ from Corollary \ref{cor:DensityBounds}. For $0<t_0\le T$ from \eqref{eq:lower_t_0} and any $0<\Delta<t_0$ take the constant $0<p_{\min,\norm{X_0}_\infty,\Xi,\Delta}<\infty$ from Corollary \ref{cor:DensityBounds}. For all $0<\nu\le\bar{\nu}$, $0<t\le T$ and $g\colon\mathbb{R}\to\mathbb{R}_{\ge 0}$ with $g\in L^1(\mathbb{R})$ we can bound
				\begin{equation*}
					\EV*{\mathcal{G}_{g,t,\sigma}} \le \abs{\Gamma} p_{\max}t^{-\alpha/2}\norm{g}_{L^1(\mathbb{R})}.
				\end{equation*}
				If, additionally, $g$ has compact support $\operatorname{supp}(g)\subset \Xi$, then
				\begin{equation*}
					\EV*{\mathcal{G}_{g,t,\sigma}} \ge  \abs{\Gamma} p_{\min,\norm{X_0}_\infty,\Xi, \Delta}\norm{g}_{L^1(\mathbb{R})},\quad 0<\Delta\le t\le t_0\le T,\quad 0<\nu\le\bar{\nu}.
				\end{equation*}
			\end{lemma}
				\begin{proof}
		See Subsection \ref{subsec:ProofsDensity}.
	\end{proof}
		We can state the concentration result for functionals of the type $\mathcal{G}_{g,t,h,\sigma}\coloneqq \int_\Gamma g_h(X_t(y))\,\D y$, $0\le t\le t_0$.
		\begin{proposition}\label{prop:SpatailErgodicity}
			Grant Assumptions \ref{assump:WellPosedness} \hyperref[assump:WellPosedness]{(well-posedness)} and \ref{assump:key} \hyperref[assump:key]{(noise-scaling)}. Take $0\le t\le t_0$ with $0<t_0\le T$ from \eqref{eq:lower_t_0} and a locally Lipschitz-continuous function $g\colon \mathbb{R}\to\mathbb{R}_{\ge 0}$ with compact support $\operatorname{supp}(g)$. Then
			\begin{equation*}
				\frac{\mathcal{G}_{g,t,h,\sigma}}{\EV*{\mathcal{G}_{g,t,h,\sigma}}}\xrightarrow{\prob{}} 1
			\end{equation*}
			 as $\nu\to 0$ for all bounded sequences of $h\le 1$, provided $\sigma=\smallo(h)$.
		\end{proposition}
		\begin{proof}
			Recall the scaling $\norm{g_h}_{L^1(\mathbb{R})}=h\norm{g}_{L^1(\mathbb{R})}$ and $\norm{g_h'}_{L^1(\mathbb{R})}=\norm{g'}_{L^1(\mathbb{R})}$. Proposition \ref{prop:VarianceBound} shows that $\operatorname{Var}(\mathcal{G}_{g,t,h,\sigma})\lesssim \sigma^2 \norm{g_h'}_{L^1(\mathbb{R})}^2\lesssim\sigma^2$. Since $\operatorname{supp}(g_h)\subset \operatorname{supp}(g)$ for $h\le 1$, Lemma \ref{lem:BoundsforExpectations} implies the lower bound $\EV{\mathcal{G}_{g,t,h,\sigma}}\gtrsim h$. An application of Chebychev's inequality yields $\mathcal{G}_{g,t,h,\sigma} / \EV*{\mathcal{G}_{g,t,h,\sigma}}=1+\mathcal{O}_{\prob{}}(\sigma h^{-1})$ and completes the proof.
		\end{proof}
		\begin{remark}[On the role of the observation window $\Gamma$ in Proposition \ref{prop:VarianceBound} and Lemma \ref{lem:BoundsforExpectations}]
		Note that the bounds in Proposition \ref{prop:VarianceBound} and Lemma \ref{lem:BoundsforExpectations} only depend on the size of the spatial observation window $\Gamma$ and not on the underlying domain $\Lambda$. This is crucial for Section \ref{sec:GrowingObs}, where we consider $\Lambda=\mathbb{R}$ and consider $\abs{\Gamma}\to\infty$.
	\end{remark}

	\subsection*{Application to the occupation time} Let $A\subset\mathbb{R}$ be a bounded interval and fix some time $0\le t\le T$. The occupation time $M(A)$ and the occupation measure $\mu(A)$ of the process $(X_t(y))_{y\in \Gamma}$ at $A$ are defined as
	\begin{equation*}
		M(A)\coloneqq \int_\Gamma \indicator_A(X_t(y))\,\D y,\quad \mu(A)\coloneqq \EV{M(A)} = \int_\Gamma \prob{X_t(y)\in A}\,\D y.
	\end{equation*}
	The following lemma shows that the concentration result of Proposition \ref{prop:SpatailErgodicity} extends to the occupation time.
	\begin{lemma}\label{lem:Occupationtime}
		Grant Assumptions \ref{assump:WellPosedness} \hyperref[assump:WellPosedness]{(well-posedness)} and \ref{assump:key} \hyperref[assump:key]{(noise-scaling)}. Fix $0<t\le T$ and a bounded non-empty interval $A\subset \mathbb{R}$. Then $\mu(A)\sim \abs{A}$ and
		\begin{equation*}
			\frac{M(A)}{\mu(A)}\xrightarrow{\prob{}}1,
		\end{equation*}
		as $\nu\to 0$.
	\end{lemma}
	\begin{proof}
		See Subsection \ref{subsec:Proofs_Concentration}.
	\end{proof}

	\subsection{Proofs for Proposition \ref{prop:VarianceBound}, Lemmas \ref{lem:BoundConditionalExpectation} and \ref{lem:BoundsforExpectations}}\label{subsec:ProofsDensity}
	
	This subsection contains the proofs of the key results in Subsection \ref{subsec:Concentration}.

	\begin{proof}[Proof of Proposition \ref{prop:VarianceBound}]
		Since $X_0$ is deterministic, the claimed bound for $t=0$ is trivial and we proceed with $0<t\le T$ arbitrary but fixed.\\
		\textbf{Step 1} (The $\norm{g'}_{L^1(\mathbb{R})}$- and the $\norm{g'}_{\infty}$-bounds). 
		The Clark-Ocone formula \eqref{eq:Clark_ocone} implies the representation
	\begin{equation*}
		\mathcal{G}_{g,t,\sigma}-\EV*{\mathcal{G}_{g,t,\sigma}} = \int_0^t \int_\Lambda \EV*{\mathcal{D}_{\tau,z}\mathcal{G}_{g,t,\sigma}\given\mathcal{F}_\tau}\,\mathcal{W}(\D \tau,\D z).
	\end{equation*}
	By It\^o's isometry, it follows that
	\begin{equation*}
		\var\left(\mathcal{G}_{g,t,\sigma}\right) =\EV[][][\bigg]{ \int_0^t\norm*{\EV*{\mathcal{D}\mathcal{G}_{g,t,\sigma}\given\mathcal{F}_\tau}}_{\mathbb{H}}^2\,\D\tau}= \EV[][][\bigg]{\int_0^t\norm[\bigg]{\EV[][][\bigg]{\int_{\Gamma}g'(X_t(y))\mathcal{D}X_t(y)\,\D y\given\mathcal{F}_\tau}}_{\mathbb{H}}^2\,\D\tau }.
	\end{equation*}
	We aim to apply Lemma \ref{lem:CondExpectation_MalliavinIntegration} to control the variance. To this end, we apply Lemma \ref{lem:BoundConditionalExpectation} with $\phi=g'$ to obtain for $0\le \tau<t$ the upper bound
	\begin{equation*}
		\EV*{\abs*{g'(X_t(y))}\given \mathcal{F}_\tau}\le \norm{g'}_{L^1(\mathbb{R})} p_{\max}(t-\tau)^{-\alpha/2} \wedge \norm{g'}_\infty=\colon \kappa(t,\tau),\quad 0<\nu \le \smash{\bar{\nu}},\quad  y\in\Gamma.
	\end{equation*}
		We can apply Lemma \ref{lem:CondExpectation_MalliavinIntegration} with this choice of the upper bound $\kappa$ and $\phi=g'$. Using $\alpha<1$, we obtain
		\begin{align*}
			\var\left(\mathcal{G}_{g,t,\sigma}\right)&\le 2 \sigma^2 C_0^2 \norm{B}^2\abs{\Gamma}\int_0^t \EV{\kappa(t,\tau)^2}\,\D\tau \left(1+e^{2\norm{f'}_\infty t}\norm{f'}_\infty^2 t^2\right)\\
			&\le 2 \sigma^2 C_0^2 \norm{B}^2\abs{\Gamma}\bigg(\frac{\norm{g'}_{L^1(\mathbb{R})}^2 p_{\max}^2t^{1-\alpha}}{1-\alpha} \wedge t\norm{g'}_\infty^2 \bigg) \left(1+e^{2\norm{f'}_\infty t}\norm{f'}_\infty^2 t^2\right).
		\end{align*}

		\noindent\textbf{Step 2} (The $\norm{g'}_{L^2(\mathbb{R})}$-bound). For the bound in terms of $\norm{g'}_{L^2(\mathbb{R})}$ we use the Poincaré inequality (Proposition 3.1 of \citet{Nourdin2009}) to bound
		\begin{equation*}
			\var\left(\mathcal{G}_{g,t,\sigma}\right)\le \EV[][][\bigg]{\int_0^t \norm*{\mathcal{D}_\tau \mathcal{G}_{g,t,\sigma}}_{\mathbb{H}}^2\,\D\tau}=\EV[][][\bigg]{\int_0^t \norm[\bigg]{\int_\Gamma g'(X_t(y))\mathcal{D}_\tau X_t(y)\,\D y}_{\mathbb{H}}^2\,\D\tau }.
			\end{equation*}
		Applying first Lemma 2.11 of \citet{gaudlitzEstimationReactionTerm2023} and then Lemma \ref{lem:BoundsforExpectations} yields
		\begin{align*}
			\var\left(\mathcal{G}_{g,t,\sigma}\right)&\le \sigma^2 \norm{B}^2C_0^2\EV[][][\bigg]{\int_0^t e^{2C_0\norm{f'}_\infty t} \norm{\indicator_\Gamma g'(X_t)}^2\,\D\tau}\\
			&= \sigma^2 \norm{B}^2 C_0^2 e^{2C_0\norm{f'}_\infty t}t \EV[][][\bigg]{\int_\Gamma g'(X_t(y))^2\,\D y}\\
			&\le \sigma^2 \norm{B}^2C_0^2  e^{2C_0\norm{f'}_\infty t}t^{1-\alpha/2}p_{\max} \abs{\Gamma} \norm{g'}_{L^2(\mathbb{R})}^2.\qedhere
		\end{align*}
	\end{proof}
		\begin{proof}[Proof of Lemma \ref{lem:BoundConditionalExpectation}]
		The bound by $\norm{\phi}_\infty$ is clear and we proceed to the $\norm{\phi}_{L^1(\mathbb{R})}$-bound. Using the Markov property from \eqref{eq:RandomField1} we can deduce the equality
		\begin{equation*}
			\EV*{\abs*{\phi(X_t(y))}\given \mathcal{F}_\tau} = \EV*{\abs*{\phi(X_t(y))}\given X_\tau}=\left.\EV*[\smash{(\tau,\xi)}]{\abs*{\phi(X_{t-\tau}(y))}}\right|_{\xi=X_\tau}
		\end{equation*}
		almost surely. Since $X_\tau\in C(\Lambda)$ almost surely, the existence result and the upper bound for the density $p_{\nu,\tau,t-\tau,y}$ of $X_{t-\tau}(y)$ under $\prob[\smash{(\tau,\xi)}]{}$ with deterministic initial condition $\xi=X_\tau$ from Corollary \ref{cor:DensityBounds} can be applied. This shows the bound
		\begin{equation*}
			\EV*{\abs*{\phi(X_t(y))}\given \mathcal{F}_\tau}= \left.\int_\mathbb{R}\abs*{\phi(x)}p_{\nu,\tau,t-\tau,y}(x)\,\D x\right|_{\xi = X_\tau}\le \norm*{\phi}_{L^1(\mathbb{R})}p_{\max}(t-\tau)^{-\alpha/2}
		\end{equation*}
		and completes the proof.
	\end{proof}
	
			\begin{proof}[Proof of Lemma \ref{lem:BoundsforExpectations}]
				Since the initial condition $X_0\in C(\Lambda)$ is deterministic, an application of the upper bound from Corollary \ref{cor:DensityBounds} yields the bound
				\begin{equation*}
					\int_\Gamma \EV*{g(X_t(y))}\,\D y  = \int_\Gamma \int_{\mathbb{R}}g(x)p_{\nu,t,y}(x)\,\D x\D y\le p_{\max}t^{-\alpha/2}\abs{\Gamma}\norm{g}_{L^1(\mathbb{R})}.
				\end{equation*}
				
				For the claimed lower bound, an application of the lower bound from Corollary \ref{cor:DensityBounds} yields for any $0<\Delta \le t\le t_0$ with $t_0$ from \eqref{eq:lower_t_0} the inequality
				\begin{equation*}
					\int_\Gamma \EV{g(X_t(y))}\,\D y = \int_\Gamma \int_{\mathbb{R}}g(x)p_{\nu,t,y}(x)\,\D x\D y\ge p_{\min,\norm{X_0}_\infty,\Xi, \Delta}\abs{\Gamma}\norm{g}_{L^1(\mathbb{R})}.\qedhere
				\end{equation*}
			\end{proof}
						
\section{Extension: Growing observation window}\label{sec:GrowingObs}

In this section we show that the methodology developed in Sections \ref{sec:Balanced_MainResults} and \ref{sec:Density} is not specific to the small diffusivity regime. As a concrete example, we consider the semi-linear stochastic heat equation on $\Lambda = \mathbb{R}$ with space-time white noise. In contrast to the rest of the paper, the diffusivity $\nu=1$ and noise level $\sigma=1$ are constant and the observation window $\Gamma$ is assumed to grow such that $\gamma\coloneqq \abs{\Gamma}\to\infty$. We will see in Theorem \ref{thm:Nonparametric_GrowingDomains} that this regime allows to recover $f(x_0)$ consistently.

For $t>0$ and $y\in\mathbb{R}$ denote by $\phi_t(y)\coloneqq (4\pi t)^{-1/2}e^{-y^2/(4t)}$ the heat kernel on $\mathbb{R}$. Note that we have
\begin{equation*}
	\int_0^t \norm*{\phi_s(y-\MTemptyplaceholder)}^2\,\D s = \frac{\sqrt{2}}{4\sqrt{\pi}}t^{1/2},\quad y\in\mathbb{R},\quad t>0,
\end{equation*}
which allows for similar density estimates as for Example \ref{examp:main} (\ref{num_Examp_a}), see Lemma \ref{lem:DensityBounds_GrowingDomains}. Young's convolution inequality (Theorem 3.9.4 of \citet{bogachevMeasureTheory2007}) implies that the heat semi-group is a contraction on $L^2(\Lambda)=L^2(\mathbb{R})$ and thus $C_0$ from Assumption \ref{assump:WellPosedness} \hyperref[assump:WellPosedness]{(well-posedness)} is equal to one. Let $f\colon\mathbb{R}\to\mathbb{R}$ be globally Lipschitz-continuous with Lipschitz-constant $\norm{f'}_\infty<\infty$. It is a classical result (e.g.\ Section 4 of \citet{nualartGaussianDensityEstimates2009a} or Lemma \ref{lem:aux_WellPosedness} (\ref{num:aux_wellposed_a})), that the (unique) random field solution $Z_t(y)$ to the SPDE
\begin{equation*}
	\D Z_t = \Updelta Z_t\,\D t + F(Z_t)\,\D t + \D W_t,\quad Z_0\equiv 0,\quad 0\le t\le T,
\end{equation*}
on $\mathbb{R}$ is given by
\begin{equation*}
	Z_t(y) = (\phi_t\star Z_0)(y) + \int_0^t (\phi_{t-s}\star F(Z_s))(y)\,\D s + \int_0^t\int_{\mathbb{R}} \phi_{t-s}(y,\eta)\mathcal{W}(\D\eta,\D s),
\end{equation*}
where $\star$ denotes the convolution on $\mathbb{R}$ in the sense that
\begin{equation*}
	(u\star v)(y) \coloneq \int_{\mathbb{R}}u(y-\eta)v(\eta)\,\D\eta,\quad y\in\mathbb{R},
\end{equation*}
for any functions $u\colon \mathbb{R}\to\mathbb{R}$, $v\colon\mathbb{R}\to\mathbb{R}$ such that $u(y-\MTemptyplaceholder)v(\MTemptyplaceholder)\in L^1(\mathbb{R})$ for all $y\in\mathbb{R}$. Lemma \ref{lem:aux_WellPosedness} (\ref{num:aux_wellposed_c}) shows that $Z_t$ is also an analytically weak solution when testing with functions belonging to $C_c^\infty(\mathbb{R})$.

We consider the estimator given by \eqref{eq:Estimator}, only changing the index $\sigma$ to $\gamma$ to highlight its dependency on the size of the observation window $\gamma=\abs{\Gamma}$. 
		The methodology developed in Sections \ref{sec:Balanced_MainResults} and \ref{sec:Density} allows us to prove the following convergence result of the estimation error.
		\begin{theorem}\label{thm:Nonparametric_GrowingDomains}
			Fix $1\le\beta\le 2$ and $L>0$. Assume that $\sqrt{\gamma} h\to\infty$, then the estimation error of $\hat{f}(x_0)_{h,\gamma}$ from \eqref{eq:Estimator} satisfies
			\begin{equation*}
				\hat{f}(x_0)_{h,\gamma}-f(x_0) = \mathcal{O}(h^{\beta}) + \mathcal{O}_{\prob{}}(\gamma^{-1/2} h^{-1/2}),
			\end{equation*}
			uniformly in $f\in \Sigma(\beta,L)$. More precisely, we can decompose the estimation error as
			\begin{equation*}
				\hat{f}(x_0)_{h,\gamma}-f(x_0) = B_{h,\gamma} + A_{h,\gamma}^{\smash{-}}Z_{h,\gamma}^{\smash{-}} + A_{h,\gamma}^{\smash{+}}Z_{h,\gamma}^{\smash{+}},\label{eq:ErrorDecomposition2_GrowingDomains}
			\end{equation*}
			where
			\begin{enumerate}[(a)]
				\item $B_{h,\gamma}= \mathcal{O}(h^{\beta})$,\label{num:MainThm_f_GrowingDomains}
				\item $Z_{h,\gamma}^{\smash{-}}= \mathcal{O}_{\prob{}}(1)$ and $Z_{h,\gamma}^{\smash{+}}= \mathcal{O}_{\prob{}}(1)$,\label{num:MainThm_b_GrowingDomains}
				\item $Z_{h,\gamma}^{\smash{-}}$ and $Z_{h,\gamma}^{\smash{+}}$ are uncorrelated random variables,\label{num:MainThm_a_GrowingDomains}
				\item $(Z_{h,\gamma}^{\smash{-}},Z_{h,\gamma}^{\smash{+}})\transpose\xrightarrow{d} N(0,\operatorname{Id}_{2\times 2})$,\label{num:MainThm_d_GrowingDomains}
				\item $A_{h,\gamma}^{\smash{-}} \sim h^{-1/2}\gamma^{-1/2} + \smallo_{\prob{}}(h^{-1/2}\gamma^{-1/2})$ and $A_{h,\gamma}^{\smash{+}} \sim h^{-1/2}\gamma^{-1/2} + \smallo_{\prob{}}(h^{-1/2}\gamma^{-1/2})$.\label{num:MainThm_e_GrowingDomains}
			\end{enumerate}
		\end{theorem}
		\begin{proof}
			See Section \ref{sec:Proofs_GrowingDomain}.
		\end{proof}
		
		\begin{remark}
			Corollaries \ref{cor:OptimalRate}, \ref{cor:CLT} and \ref{cor:Test} can be restated in this setting without difficulties.
		\end{remark}

		Analogously to Section \ref{sec:Balanced_MainResults}, the main building block to analyse the estimation error of $\hat{f}(x_0)_{h,\gamma}$ is the concentration due to spatial ergodicity of functionals
		\begin{equation*}
			\mathcal{G}_{g,\gamma}\coloneqq \frac{1}{\abs{\Gamma}}\int_0^T \int_{\Gamma}g(Z_t(y))\,\D y\D t
		\end{equation*}
		for suitable functions $g\colon\mathbb{R}\to\mathbb{R}$ as $\gamma=\abs{\Gamma}\to\infty$. These are summarised in the following lemma.	
				
		\begin{lemma}\label{lem:Ergodicity_GrowingDomains}
			The functional $\mathcal{G}_{g,\gamma}$ satisfies the following properties as $\gamma=\abs{\Gamma}\to\infty$.
			\begin{enumerate}[(a)]
				\item\label{num:Growingobs_a} Let $g$ be locally Lipschitz-continuous with almost everywhere existing derivative $g'$. Then
			\begin{equation*}
				\var\left(\mathcal{G}_{g,\gamma}\right)= \gamma^{-1}\mathcal{O}\left(\norm{g'}_{L^1(\mathbb{R})}^2\wedge \norm{g'}_{L^2(\mathbb{R})}^2\wedge \norm{g'}_\infty^2\right),
			\end{equation*}
			where the constant in $\mathcal{O}$ does not depend on $\Gamma$ or $g$.
			\item Assume that $g\ge 0$ and $g\in L^1(\mathbb{R})$, then 
			\begin{equation*}
				\EV*{\mathcal{G}_{g,\gamma}}\le C \norm{g}_{L^1(\mathbb{R})}
			\end{equation*}
			for some constant $0<C<\infty$ that does not depend on $\Gamma$ or $g$.
			\item Assume that $g\ge 0$, $g\in L^1(\mathbb{R})$ and that $g$ has compact support. Then
			\begin{equation*}
				\EV*{\mathcal{G}_{g,\gamma}}\ge c\norm{g}_{L^1(\mathbb{R})}
			\end{equation*}
			for a constant $0<c<\infty$ that does not depend on $\Gamma$ or $g$.
			\item Let $g\ge 0$ be Lipschitz-continuous and have bounded support. Then
			\begin{equation*}
				\frac{\int_0^T\int_\Gamma g_h(Z_t(y))\,\D y\D t}{\int_0^T\int_\Gamma \EV*{g_h(Z_t(y))}\,\D y\D t}\xrightarrow{\prob{}}1,
			\end{equation*}
			as $\gamma=\abs{\Gamma}\to\infty$, provided $\sqrt{\gamma} h \to \infty$.
			\end{enumerate}
		\end{lemma}
		\begin{proof}
		See Section \ref{sec:Proofs_GrowingDomain}.
		\end{proof}
		\begin{remark}
			The bound for the variance in (\ref{num:Growingobs_a}) of Lemma \ref{lem:Ergodicity_GrowingDomains} in terms of $\norm{g'}_\infty$ is known from Theorem 1.6 of \citet{chenSpatialErgodicitySPDEs2021}. As mentioned in Remark \ref{rmk:Connection}, this bound is not sufficiently sharp for our purposes.
		\end{remark}

		\section*{Acknowledgements}
		Many fruitful discussions of the author with Markus Rei\ss, Gregor Pasemann, Randolf Altmeyer and Eric Ziebell are gratefully acknowledged. This research has been partially funded by Deutsche Forschungsgemeinschaft (DFG) -- Project-ID 410208580 -- IRTG2544 and DFG -- Project-ID 318763901 -- SFB1294.
\appendix

	\section{Results for Examples \ref{examp:main} (\ref{num_Examp_a})--(\ref{num_Examp_c})}\label{sec:Proofsexample*}

	This section contains the proofs for Examples \ref{examp:main} (\ref{num_Examp_a})-(\ref{num_Examp_c}). As $A_t$ does not depend on $t$ in Examples \ref{examp:main} (\ref{num_Examp_a})-(\ref{num_Examp_c}), we write $G_{\nu(t-s)}$ instead of $G_{\nu,t,s}$ in the following. We introduce the heat kernel on $\mathbb{R}^d$, which is given by $\phi_t(y)= (4\pi t)^{-d/2}e^{-\abs{y}^2/(4t)}$ for $y\in\mathbb{R}^d$ and $t>0$.
	
	\begin{remark}(Scaling of the heat kernel).\label{rmk:Scaling_Heat_Kernel}
	We first prove that if $G_{\nu \MTemptyplaceholder}(\MTemptyplaceholder,\MTemptyplaceholder)$ is the Dirichlet/Neumann heat kernel on $\Lambda$ with diffusivity level $\nu$, then
	\begin{equation}	 
		G_t^\nu(\tilde{y},\tilde{\eta})\coloneqq \nu^{d/2}G_{\nu t}(\nu^{1/2}\tilde{y},\nu^{1/2}\tilde{\eta}),\quad \tilde{y},\tilde{\eta}\in \Lambda_\nu\coloneqq \nu^{-1/2}\Lambda,\label{eq:Scaling_Heat_kernel}
	\end{equation}
	is the Dirichlet/Neumann heat kernel on $\Lambda_\nu$ with diffusivity level 1. To prove the claim, first note that
	\begin{equation}
		\frac{\partial}{\partial t} G_t^\nu(\tilde{y},\tilde{\eta}) = \frac{\partial}{\partial t}(\nu^{d/2}G_{\nu t}(\nu^{1/2}\tilde{y},\nu^{1/2}\tilde{\eta}))  = \Updelta(\nu^{d/2}G_{\nu t}(\nu^{1/2}\tilde{y},\nu^{1/2}\tilde{\eta})) = \Updelta G_t^\nu(\tilde{y},\tilde{\eta}),\label{eq:aux_green1}
	\end{equation}
	for any $\nu>0$, $t>0$ and $\tilde{y},\tilde{\eta}\in \Lambda_\nu$. Furthermore,
	\begin{align}
	\begin{split}
		\int_{\Lambda_\nu}G_t^\nu(\tilde{y},\tilde{\eta})\xi(\tilde{\eta})\,\D\tilde{\eta} &= \nu^{d/2}\int_{\Lambda_\nu} G_{\nu t}(\nu^{1/2}\tilde{y},\nu^{1/2}\tilde{\eta})\xi(\tilde{\eta}) \\
		&= \int_\Lambda G_{\nu t}(\nu^{1/2}\tilde{y},\eta)\xi(\nu^{-1/2}\eta)\,\D\eta\xrightarrow{t\to 0}\xi(\tilde{y}),\quad \xi\in C(\Lambda_\nu),
		\end{split}\label{eq:aux_green2}
	\end{align}
	for every $\nu>0$ and $\tilde{y}\in \Lambda_\nu$. The required behaviour of $G_t^\nu$ at the boundary of $\Lambda_\nu$ for Dirichlet/Neumann boundary conditions is inherited from $G_{\nu t}$. Together with \eqref{eq:aux_green1} and \eqref{eq:aux_green2}, this shows \eqref{eq:Scaling_Heat_kernel}.
	\end{remark}

	\subsection{Example \ref{examp:main} (\ref{num_Examp_a})}\label{subsubsec:Example_a}
	
	\begin{lemma}\label{lem:zoomingout}
		Consider the setting of Example \ref{examp:main} (\ref{num_Examp_a}) and define the spatially inflated process $Y_t(\tilde{y})\coloneqq X_t(\nu^{1/2}\tilde{y})$ for $\tilde{y}\in\Lambda_\nu=\nu^{-1/2}\Lambda$ and $0\le t\le T$. Then $(Y_t)_{t\in[0,T]}$ is a random field solution of the SPDE
		\begin{equation}
			\D Y_t = \Updelta Y_t\,\D t + F(Y_t)\,\D t + \nu^{-1/4}\sigma\,\D \bar{W}_t,\quad 0\le t\le T,\quad Y_0=X_0(\nu^{1/2}\MTemptyplaceholder),\label{eq:SPDE_Y}
		\end{equation}
		on $\Lambda_\nu$ with space-time white noise $\D\bar{W}_t/\D t$ on $\Lambda_\nu$.
	\end{lemma}
	\begin{proof}
	Using \eqref{eq:Scaling_Heat_kernel}, we obtain for all $\tilde{y}\in\Lambda_\nu$ and $0\le t\le T$ the distributional equality
	\begin{align*}
		Y_t(\tilde{y})&=X_t(\nu^{1/2}\tilde{y})\\
		&=\int_\Lambda G_{\nu t}(\nu^{1/2}\tilde{y},\eta)X_0(\eta)\,\D\eta + \int_0^t\int_\Lambda G_{\nu (t-s)}(\nu^{1/2}\tilde{y},\eta)f(X_s(\eta))\,\D\eta\D s\\
		&\quad + \sigma\int_0^t\int_\Lambda G_{\nu (t-s)}(\nu^{1/2}\tilde{y},\eta)\mathcal{W}(\D\eta,\D s)\\
		&\overset{d}{=}\int_{\Lambda_\nu} G_{t}^\nu(\tilde{y},\tilde{\eta})X_0(\nu^{1/2}\tilde{\eta})\,\D\tilde{\eta} + \int_0^t\int_{\Lambda_\nu} G_{t-s}^\nu(\tilde{y},\tilde{\eta})f(X_s(\nu^{1/2}\tilde{\eta}))\,\D\tilde{\eta}\D s\\
		&\quad+\nu^{-1/4}\sigma\int_0^t\int_{\Lambda_\nu} G_{t-s}^\nu(\tilde{y},\tilde{\eta})\bar{\mathcal{W}}(\D\tilde{\eta},\D s)\\
		&=\int_{\Lambda_\nu} G_{t}^\nu(\tilde{y},\tilde{\eta})Y_0(\tilde{\eta})\,\D\tilde{\eta} + \int_0^t\int_{\Lambda_\nu} G_{t-s}^\nu(\tilde{y},\tilde{\eta})f(Y_s(\tilde{\eta}))\,\D\tilde{\eta}\D s\\
		&\quad+ \nu^{-1/4}\sigma\int_0^t\int_{\Lambda_\nu} G_{t-s}^\nu(\tilde{y},\tilde{\eta})\bar{\mathcal{W}}(\D\tilde{\eta},\D s).
	\end{align*}
	Since the random field solution to \eqref{eq:SPDE_Y} is strongly (in a probabilistic sense) unique, the claim follows.
	\end{proof}

	\begin{remark}[On Lemma \ref{lem:zoomingout} and exploding spatial H\"older-norms]
	Viewing $X_t$ as a spatially squeezed version of an SPDE on a growing domain $\Lambda_\nu$ explains that the H\"older-norms of $X_t$ explode as $\nu\to 0$ for any fixed time $0<t\le T$, whereas its $\mathcal{H}$-norm remains finite (compare Proposition 3.14 of \citet{gaudlitzEstimationReactionTerm2023}).
	\end{remark}
		
	 	\begin{lemma}\label{lem:Variance_Control_Mainexample*}
			In the setting of Example \ref{examp:main} (\ref{num_Examp_a}) we have
			\begin{equation*}
				\int_0^t \norm*{G_{\nu,t,s}(y,\MTemptyplaceholder)}_{\mathbb{H}}^2\,\D s
				\begin{cases}
				\gtrsim t^{1/2}\nu^{-1/2},&\text{uniformly in }y\in\Gamma, 0\le t\le T, 0<\nu\le \bar{\nu},\\
				\lesssim t^{1/2}\nu^{-1/2},&\text{uniformly in }y\in\Lambda, 0\le t\le T, \nu>0,
				\end{cases}
			\end{equation*}
			for any fixed $\bar{\nu}>0$.
			\end{lemma}
		\begin{proof}
	Note that \eqref{eq:Scaling_Heat_kernel} implies for all time points $0<t\le T$ and locations $y\in\Lambda$ the equality
	\begin{align}
	\begin{split}
		\int_0^t \norm*{G_{\nu(t-s)}(y,\MTemptyplaceholder)}_{\mathbb{H}}^2\,\D s&= \int_0^t\int_\Lambda G_{\nu(t-s)}(y,\eta)^2\,\D \eta\D s= \nu^{1/2}\int_0^t\int_{\Lambda_\nu}G_{\nu(t-s)}(y,\nu^{1/2}\tilde{\eta})^2\,\D\tilde{\eta}\D s\\
		&=\nu^{-1/2}\int_0^t \int_{\Lambda_\nu}G_{t-s}^\nu(\nu^{-1/2}y,\tilde{\eta})^2\,\D\tilde{\eta}\D s.\end{split}
	\label{eq:G_estimate_bound_to_unbounded}
	\end{align}
	By interpreting $G_{t}^\nu$ as a transition density of a Brownian motion killed outside of $\Lambda_\nu$, the inequality 
	\begin{equation*}
		 G_t^{\bar{\nu}}(y,\eta)\le G_t^\nu(y,\eta)\le \phi_{t}(y-\eta),\quad t>0,\quad y,\eta\in\mathbb{R},\quad 0<\nu\le\bar{\nu}.
	\end{equation*}
	follows, see Equation (2.86) of \citet{nualartMalliavinCalculusRelated2006}.\\	
		For the claimed upper bound, we use \eqref{eq:G_estimate_bound_to_unbounded} to compute for any $y\in\Lambda$
		\begin{align*}
			\nu^{1/2}\int_0^t \norm*{G_{\nu(t-s)}(y,\MTemptyplaceholder)}_{\mathbb{H}}^2\,\D s&\le \int_0^t \int_{\Lambda_\nu}\phi_{t-s}(\nu^{-1/2}y-\eta)^2\,\D\eta\D s\le \int_0^t \int_{\mathbb{R}}\phi_{t-s}(\nu^{-1/2}y-\eta)^2\,\D\eta\D s\\
			&\lesssim \int_0^t (t-s)^{-1/2}\,\D s\lesssim t^{1/2}.
		\end{align*}
		We proceed with the claimed lower bound and fix $y\in\Gamma$, $0<t\le T$ and $\bar{\nu}>0$. An application of the lower bound $G_{\nu s}(y,\eta)\ge c_1 (s\nu)^{-1/2} \phi_{c_2}((s\nu)^{-1/2}(y-\eta))$ with constants $c_1,c_2>0$ from Equation (0.27) of \citet{varopoulosGaussianEstimatesLipschitz2003a} yields the lower bound
	\begin{align*}
		\int_0^t \norm*{G_{\nu,t,s}(y,\MTemptyplaceholder)}_{\mathbb{H}}^2\,\D s&\ge \int_0^t\int_{\Gamma}G_{\nu s}(y,\eta)^2\,\D \eta\D s\ge \frac{c_1}{\nu}\int_0^t\frac{1}{s}\int_{\Gamma}\phi_{c_2}((s\nu)^{-1/2}(y-\eta))^2\,\D \eta\D s\\
		&=c_1\nu^{-1/2}\int_0^t s^{-1/2} \int_{(s\nu)^{-1/2}(y-\Gamma)}\phi_{c_2}(\eta)\,\D \eta\D s\\
		&\ge c_1\nu^{-1/2}\int_0^t s^{-1/2} \D s\int_{(T\bar{\nu})^{-1/2}(y-\Gamma)}\phi_{c_2}(\eta)\,\D \eta\gtrsim \nu^{-1/2}t^{1/2},
	\end{align*}
	where we used $(T\bar{\nu})^{-1/2}(y-\Gamma) \subset(s\nu)^{-1/2}(y-\Gamma)$ for all $0<\nu\le\bar{\nu}$ and $0<s\le T$ since $y\in\Gamma$ and $\Gamma$ is connected in $\mathbb{R}$, hence convex.
		\end{proof}
	
\subsection{Example \ref{examp:main} (\ref{num_Examp_b})}

	\begin{lemma}\label{lem:VaryingCorrelationLength}
		In the setting of Example \ref{examp:main} (\ref{num_Examp_b}) we have 
			\begin{equation*}
				\int_0^t \norm*{G_{\nu,t,s}(y,\MTemptyplaceholder)}_{\mathbb{H}}^2\,\D s
				\begin{cases}
				\gtrsim t^{1-\rho/2}\nu^{-\rho/2},&\text{uniformly in }y\in\Gamma, 0\le t\le T, 0<\nu\le \bar{\nu},\\
				\lesssim t^{1-\rho/2}\nu^{-\rho/2},&\text{uniformly in }y\in\Lambda,0\le t\le T,\nu>0,
				\end{cases}
			\end{equation*}
			for any fixed $\bar{\nu}>0$.
	\end{lemma}
	\begin{proof}
	Fix $0<\nu\le\bar{\nu}$. Using the identity \eqref{eq:Scaling_Heat_kernel}, we find
	\begin{equation*}
		G_{\nu t}(y,\eta)=\nu^{-d/2}G_t^\nu(\nu^{-1/2}y,\nu^{-1/2}\eta)\le \nu^{-d/2}\phi_{t}(\nu^{-1/2}(y-\eta)),\quad y,\eta\in\mathbb{R},\quad t>0,
	\end{equation*}
	for the heat kernel $\phi_t(y)$ on $\mathbb{R}^d$. Consequently, we obtain the upper bound
	\begin{align*}
		\nu^{\rho/2}\int_0^t \norm*{G_{\nu,t,s}(y,\MTemptyplaceholder)}_{\mathbb{H}}^2\,\D s\hspace{-2em}&\hspace{2em}= \int_0^t\int_{\Lambda}\int_{\Lambda} \chi(\nu^{-1/2}(\eta_1-\eta_2))G_{\nu s}(y,\eta_1)G_{\nu s}(y,\eta_2)\,\D \eta_1\D\eta_2\D s\\
		&\le \frac{1}{\nu^{d}}\int_0^t\int_{\mathbb{R}^d}\int_{\mathbb{R}^d} \chi(\nu^{-1/2}(\eta_1-\eta_2))\phi_s(\nu^{-1/2}(y-\eta_1))\phi_s(\nu^{-1/2}(y-\eta_2))\,\D \eta_1\D\eta_2\D s\\
		&=\int_0^t\int_{\mathbb{R}^d}\chi(\tilde{\eta}_1)\int_{\mathbb{R}^d}\phi_s(\tilde{\eta}_1-\tilde{\eta}_2)\phi_s(\tilde{\eta}_2)\,\D \tilde{\eta}_2\D\tilde{\eta}_1\D s\\
		&= \int_0^t\int_{\mathbb{R}^d}\chi(\eta)\phi_{2s}(\eta)\,\D\eta\D s= \int_0^t s^{-\rho/2}\,\D s \int_{\mathbb{R}^d}\chi(\eta)\phi_{2}(\eta)\,\D\eta\lesssim t^{1-\rho/2}.
	\end{align*}
	For the lower bound, we combine the lower bound of (0.27) of \citet{varopoulosGaussianEstimatesLipschitz2003a} with $\lambda_1(y-\Gamma)\subset \lambda_2(y-\Gamma)$ for all $0\le\lambda_1\le \lambda_2<\infty$, $y\in\Gamma$, by convexity of $\Gamma$, to obtain
		\begin{align*}
		\nu^{\rho/2}\int_0^t \norm*{G_{\nu,t,s}(y,\MTemptyplaceholder)}_{\mathbb{H}}^2\,\D s&\ge \int_0^t\int_{\Gamma}\int_{\Gamma} \chi(\nu^{-1/2}(\eta_1-\eta_2))G_{\nu s}(y,\eta_1)G_{\nu s}(y,\eta_2)\,\D \eta_1\D\eta_2\D s\\
		&\ge \frac{c_1}{\nu^{d}}\int_0^t\frac{1}{s^d}\int_{\Gamma}\int_{\Gamma} \chi\left(\frac{\eta_1-\eta_2}{\sqrt{\nu}}\right)\phi_{c_2}\left(\frac{y-\eta_1}{\sqrt{s\nu}}\right)\phi_{c_2}\left(\frac{y-\eta_2}{\sqrt{s\nu}}\right)\,\D \eta_1\D\eta_2\D s\\
		&=\int_0^t s^{-\rho/2} \int_{(s\nu)^{-1/2}(y-\Gamma)}\int_{(s\nu)^{-1/2}(y-\Gamma)}\chi(\eta_1-\eta_2)\phi_{c_2}(\eta_1)\phi_{c_2}(\eta_2)\,\D \eta_1\D\eta_2\D s\\
		&\ge \int_0^t s^{-\rho/2} \D s\int_{(T\bar{\nu})^{-1/2}(y-\Gamma)}\int_{(T\bar{\nu})^{-1/2}(y-\Gamma)}\chi(\eta_1-\eta_2)\phi_{c_2}(\eta_1)\phi_{c_2}(\eta_2)\,\D \eta_1\D\eta_2\\
		&\gtrsim t^{1-\rho/2},
	\end{align*}
	where $0<c_1,c_2<\infty$ are universal constants.
	 \end{proof}
	 \pagebreak
	 \subsection{Example \ref{examp:main} (\ref{num_Examp_d})}\label{subsubsec:Example_d}
	 
	 	 To quantify the upper and lower bounds of $\int_0^t \norm*{G_{\nu,t,s}(y,\MTemptyplaceholder)}_{\mathbb{H}}^2\,\D s$ in the case of Example \ref{examp:main} (\ref{num_Examp_d}), the following condition on the eigenfunctions $(e_k)_{k\in\mathbb{N}}$ is sufficient.
	 
	 \begin{assumption}\label{assump:B_Aux}
	 	The eigenfunctions $(e_k)_{k\in\mathbb{N}}$ satisfy uniformly in $0<\nu\le \bar{\nu}$ for some $\bar{\nu}>0$ the bounds
	 	\begin{equation*}
				\sum_{k=1}^\infty e_k(y)^2 \sigma_k^2\int_0^t e^{2\nu\lambda_k s}\,\D s 
				\begin{cases}
					\gtrsim \sum_{k=1}^\infty \sigma_k^2\int_0^t e^{2\nu\lambda_k s}\,\D s,&\text{uniformly in }y\in\Gamma,\\
					\lesssim \sum_{k=1}^\infty \sigma_k^2\int_0^t e^{2\nu\lambda_k s}\,\D s,&\text{uniformly in }y\in\Lambda.\\
				\end{cases}
			\end{equation*}
	 \end{assumption}
	 
	 	\begin{lemma}\label{lem:Variance_Control_Spectral}
			Grant Assumption \ref{assump:B_Aux}. In the setting of Example \ref{examp:main} (\ref{num_Examp_d}) we have 
			\begin{equation*}
				\int_0^t \norm*{G_{\nu,t,s}(y,\MTemptyplaceholder)}_{\mathbb{H}}^2\,\D s
					\sim
					\begin{cases}
					\nu^{(1-2\rho_2)/\rho_1}t^{1+(2\rho_2-1)/\rho_1},& \rho_2<1/2,\\
					-t\log{(\nu t)},&\rho_2=2, \nu t<1,\\
					t,&\rho_2>1/2,
				\end{cases}
			\end{equation*}
			uniformly in $y\in\Lambda$, $0\le t\le T$ and $0<\nu\le \bar{\nu}$. $\lesssim$ holds uniformly in $y\in \Lambda$ and $\gtrsim$ uniformly in $y\in\Gamma$.
			\end{lemma}
			\begin{proof}
			Using Assumption \ref{assump:B_Aux} in the first line, we compute for $0\le t\le T$, $0<\nu\le \bar{\nu}$ and $y\in \Lambda$ (for $\lesssim$) or $y\in\Gamma$ (for $\gtrsim$)
			\begin{align*}
				\int_0^t\norm*{G_{\nu(t-s)}(y,\MTemptyplaceholder)}_{\mathbb{H}}^2\,\D s&\sim\sum_{k=1}^\infty \sigma_k^2\int_0^t e^{-2\nu\lambda_k s}\,\D s\sim \sum_{k=1}^\infty k^{-2\rho_2}((\nu k^{\rho_1})^{-1}\wedge t)\\
				&\sim \sum_{k=1}^{\lceil(\nu t)^{-1/\rho_1}\rceil}k^{-2\rho_2}t + \sum_{k=\lceil(\nu t)^{-1/\rho_1}\rceil+1}^\infty \nu^{-1} k^{-\rho_1-2\rho_2}\\
				&=\begin{cases}
					t^{1+2\rho_2/\rho_1 -1/\rho_1}\nu^{2\rho_2/\rho_1 -1/\rho_1},& \rho_2<1/2,\\
					t(1-\log{(\nu t)}/\rho_1),&\rho_2=1/2,\nu t<1,\\
					t,&\rho_2> 1/2.
					\end{cases}\qedhere
			\end{align*}
			\end{proof}

		 \subsection{Example \ref{examp:main} (\ref{num_Examp_c})}
	 
	 \begin{lemma}\label{lem:Neumann}
	 	Lemmas \ref{lem:Variance_Control_Mainexample*} and \ref{lem:VaryingCorrelationLength} hold true for $\Gamma=\Lambda$, if we endow the SPDE with Neumann boundary conditions.
	 \end{lemma}
	 \begin{proof}
		Theorem 3.2.9 of \citet{daviesHeatKernelsSpectral1990} implies the upper bound $G_{\nu t}(y,\eta)\lesssim ((\nu t)^{-d/2}\vee 1)e^{-\abs{y-\eta}^2/(C\nu t)}$ for all $y,\eta \in\Lambda$, $0<t\le T$, $0<\nu\le\bar{\nu}$ and a universal constant $0<C<\infty$. Additionally, the
		  domain monotonicity properties of the Neumann heat kernel imply the lower bound $G_{\nu t}(y,\eta)= \nu^{-d/2}G_t^\nu(\nu^{-1/2}y,\nu^{-1/2}\eta)\ge \nu^{-d/2}\phi_t(\nu^{-1/2}(y-\eta))$ for all $y,\eta\in\Lambda$, $0<t\le T$ and $0<\nu\le\bar{\nu}$, see Section 4 of \citet{kendallCoupledBrownianMotions1989}. Proceeding as in the proofs of Lemmas \ref{lem:Variance_Control_Mainexample*} and \ref{lem:VaryingCorrelationLength} yields the claimed bounds.
	 \end{proof}
			
	\section{Further results for Section \ref{sec:Balanced_MainResults}}\label{sec:AppendixMainResults}

	\subsection{Asymptotic behaviour of the random weights}
	
		The weights of the estimator $\hat{f}(x_0)_{h,\sigma}$ from \eqref{eq:Estimator} are of the form
		\begin{equation*}
			\mathcal{G}_{g,h,\sigma}\coloneqq \int_0^T\int_\Gamma g_h(X_t(y))\,\D y\D t = \int_0^T \mathcal{G}_{g,t,h,\sigma}\,\D t,\quad h>0,
		\end{equation*}
		for Lipschitz-continuous functions $g\colon \mathbb{R}\to\mathbb{R}$. The following lemma shows that the spatial ergodicity results for $\mathcal{G}_{g,t,h,\sigma}$ from Section \ref{sec:Density} also imply concentration for $\mathcal{G}_{g,h,\sigma}$ as $\nu\to 0$.
			\begin{lemma}\label{lem:Ergodicity}
			Grant Assumptions \ref{assump:WellPosedness} \hyperref[assump:WellPosedness]{(well-posedness)} and \ref{assump:key} \hyperref[assump:key]{(noise-scaling)}. The functional $\mathcal{G}_{g,h,\sigma}$ satisfies the following properties as $\nu\to 0$ and, consequently, $\sigma\to 0$.
			\begin{enumerate}[(a)]
				\item\label{num:Erg_a} Let $g$ be globally Lipschitz-continuous with almost everywhere existing derivative $g'\in L^\infty(\mathbb{R})$. Then
			\begin{equation*}
				\var\left(\mathcal{G}_{g,h,\sigma}\right)= \sigma^2\mathcal{O}\left(\norm{g'}_{L^1(\mathbb{R})}^2\wedge h^{-1}\norm{g'}_{L^2(\mathbb{R})}^2\wedge h^{-2}\norm{g'}_\infty^2\right),
			\end{equation*}
			where the constant in $\mathcal{O}$ does not depend on the diffusivity $\nu$, the function $g$ or the bandwidth $h$.
			\item\label{num:Erg_b} Assume that $g\in L^1(\mathbb{R})$, then 
			\begin{equation*}
				\EV*{\mathcal{G}_{g,h,\sigma}}\le C \norm{g}_{L^1(\mathbb{R})}h,\quad h>0,
			\end{equation*}
			for some constant $0<C<\infty$ that does not depend on the diffusivity $\nu$, the function $g$ or the bandwidth $h$.
			\item\label{num:Erg_c} Fix some compact set $\Xi\subset \mathbb{R}$. Assume that $g\ge 0$, $g\in L^1(\mathbb{R})$ and that $\operatorname{supp}(g)\subset\Xi$. Then
			\begin{equation*}
				\EV*{\mathcal{G}_{g,h,\sigma}}\ge c\norm{g}_{L^1(\mathbb{R})}h,\quad h>0,
			\end{equation*}
			for a constant $0<c<\infty$ that depends on $\Xi$, but not on the diffusivity $\nu$, the function $g$ or the bandwidth $h$.
			\item\label{num:Erg_d} Let $g\ge 0$ be locally Lipschitz-continuous and have bounded support. Then
			\begin{equation*}
				\frac{\int_0^T\int_\Gamma g_h(Z_t(y))\,\D y\D t}{\int_0^T\int_\Gamma \EV*{g_h(Z_t(y))}\,\D y\D t}\xrightarrow{\prob{}}1,
			\end{equation*}
			as $\nu\to0$, provided $\sigma=\sigma(\nu)=\smallo(h)$.
			\end{enumerate}
		\end{lemma}
		\begin{proof}
		Fix $h>0$.
		\begin{enumerate}[(a)]
			\item Since
			\begin{equation*}
				 \var\left(\mathcal{G}_{g,h,\sigma}\right) = \EV[][][\bigg]{\bigg(\int_0^T \left(\mathcal{G}_{g_h,t,\sigma} - \EV{\mathcal{G}_{g_h,t,\sigma}}\right)\,\D t\bigg)^2}\le T\int_0^T \var(\mathcal{G}_{g_h,t,\sigma})\,\D t
			\end{equation*}
			by the Cauchy-Schwarz inequality, the claim follows from Proposition \ref{prop:VarianceBound}, $\alpha<1$ and the scaling properties from Remark \ref{rmk:Scaling}.
			\item The upper bound for the density from Corollary \ref{cor:DensityBounds} implies
			\begin{align*}
				 \int_{0}^T\int_{\Gamma} \EV*{g_h(X_t(y))}\,\D y\D t&=\int_{0}^T\int_{\Gamma} \int_{\mathbb{R}} g_h(x)p_{\nu,t,y}(x)\,\D x\D y\D t\le \int_{0}^T\int_{\Gamma} \int_{\mathbb{R}} g_h(x)p_{\max}t^{-\alpha/2}\,\D x \D y\D t\\
				 &= \norm{g_h}_{L^1(\mathbb{R})}p_{\max}\abs{\Gamma}T^{1-\alpha/2}(1-\alpha/2)^{-1}.
			\end{align*}
			\item Since $g\ge 0$, the lower bound for the density from Corollary \ref{cor:DensityBounds} implies for every $0<\Updelta<t_0$ with $0<t_0\le T$ from \eqref{eq:lower_t_0} the bound
				\begin{align*}
					\int_{0}^T\int_{\Gamma} \EV*{g_h(X_t(y))}\,\D y\D t &\ge \int_{\Delta}^{ t_0}\int_{\Gamma} \EV*{g_h(X_t(y))}\,\D y\D t\ge \int_{\Delta}^{ t_0}\int_{\Gamma} \int_{\mathbb{R}} g_h(x)p_{\nu,t,y}(x)\,\D x\D y\D t\\
					&\ge \int_{\Delta}^{ t_0}\int_{\Gamma} \int_{\mathbb{R}} g_h(x)p_{\min,\norm{X_0}_\infty,\Xi,\Delta}\,\D x\D y\D t\\
					&= \norm{g_h}_{L^1(\mathbb{R})}p_{\min,\norm{X_0}_\infty,\Xi,\Delta}\abs{\Gamma}(t_0-\Delta).
				\end{align*}
			\item The claim follows from Properties (\ref{num:Erg_a}) and (\ref{num:Erg_c}), combined with Chebychev's inequality.\qedhere
		\end{enumerate}

		\end{proof}
	
		The following lemma collects the asymptotic behaviour of the statistical quantities appearing in the error decomposition \eqref{eq:ErrorDecomposition2}, which arise from the concentration results of Lemma \ref{lem:Ergodicity}.
		\begin{lemma}\label{lem:DetailsStochasticError}
			Grant Assumptions \ref{assump:WellPosedness} \hyperref[assump:WellPosedness]{(well-posedness)}, \ref{assump:key} \hyperref[assump:key]{(noise-scaling)} and \ref{assump:Kernel} \hyperref[assump:Kernel]{(kernel)}. Let $\nu\to 0$ and assume $\sigma(\nu) = \smallo(h)$, then
			\begin{enumerate}[(a)]
			\item \label{num:AuxLemma_a} the auxiliary quantities $\mathcal{T}_{h,\sigma}^{\smash{+,1}}$ ,$\mathcal{T}_{h,\sigma}^{\smash{+,2}}$, $\mathcal{T}_{h,\sigma}^{\smash{-,1}}$ and $\mathcal{T}_{h,\sigma}^{\smash{+,1}}$ satisfy
			\begin{align*}
			 	\EV{\mathcal{T}_{h,\sigma}^{\smash{\pm,1}}}&\sim h,\quad \mathcal{T}_{h,\sigma}^{\smash{\pm,1}}=\EV{\mathcal{T}_{h,\sigma}^{\smash{\pm,1}}} + \smallo_{\prob{}}\big(\EV{\mathcal{T}_{h,\sigma}^{\smash{\pm,1}}}\big),\\
			 	\EV{\mathcal{T}_{h,\sigma}^{\smash{\pm,2}}}&\sim h^2,\quad \mathcal{T}_{h,\sigma}^{\smash{\pm,2}}=\EV{\mathcal{T}_{h,\sigma}^{\smash{\pm,2}}} + \smallo_{\prob{}}\big(\EV{\mathcal{T}_{h,\sigma}^{\smash{\pm,2}}}\big).
			 	\end{align*}
			\item\label{num:AuxLemma_b} Furthermore, we find $\EV{\mathcal{T}_{h,\sigma}^{\smash{-,1}}}\EV{\mathcal{T}_{h,\sigma}^{\smash{+,2}}} + \EV{\mathcal{T}_{h,\sigma}^{\smash{+,1}}}\EV{\mathcal{T}_{h,\sigma}^{\smash{-,2}}}\sim h^3$ and
			\begin{equation*}
				\mathcal{J}_{h,\sigma} = \EV{\mathcal{T}_{h,\sigma}^{\smash{-,1}}}\EV{\mathcal{T}_{h,\sigma}^{\smash{+,2}}} + \EV{\mathcal{T}_{h,\sigma}^{\smash{+,1}}}\EV{\mathcal{T}_{h,\sigma}^{\smash{-,2}}} + \smallo_{\prob{}}\big(\EV{\mathcal{T}_{h,\sigma}^{\smash{-,1}}}\EV{\mathcal{T}_{h,\sigma}^{\smash{+,2}}} + \EV{\mathcal{T}_{h,\sigma}^{\smash{+,1}}}\EV{\mathcal{T}_{h,\sigma}^{\smash{-,2}}}\big).
			\end{equation*}
			\item\label{num:AuxLemma_c} Moreover, we have $\EV{\mathcal{I}_{h,\sigma}^{\smash{\pm}}}\lesssim h$.
			\item\label{num:AuxLemma_d} If, additionally, Assumption \ref{assump:CLT} \hyperref[assump:CLT]{(noise covariance function)} is satisfied, then
			\begin{equation*}
				 \EV{\mathcal{I}_{h,\sigma}^{\smash{\pm}}}\sim h,\quad  \mathcal{I}_{h,\sigma}^{\smash{\pm}} = \EV{\mathcal{I}_{h,\sigma}^{\smash{\pm}}} + \smallo_{\prob{}}\big(\EV{\mathcal{I}_{h,\sigma}^{\smash{\pm}}}\big).\\
			\end{equation*}
			\end{enumerate}
		\end{lemma}
			\begin{proof}
			To control $\mathcal{I}_{h,\sigma}^{\smash{+}}$ and $\mathcal{I}_{h,\sigma}^{\smash{-}}$ we introduce
			\begin{equation*}
				\mathcal{T}_{h,\sigma}^{\smash{\pm,3}}\coloneqq \int_0^T\int_{\Gamma} K_{\pm,h}(X_t(y))^2\,\D y\D t.
			\end{equation*}
			\textbf{Step 1} (Controlling $\mathcal{T}_{h,\sigma}^{\smash{+,1}},\mathcal{T}_{h,\sigma}^{\smash{+,2}}, \mathcal{T}_{h,\sigma}^{\smash{+,3}}, \mathcal{T}_{h,\sigma}^{\smash{-,1}},\mathcal{T}_{h,\sigma}^{\smash{-,2}},\mathcal{T}_{h,\sigma}^{\smash{-,3}}$ and proving (\ref{num:AuxLemma_a}), (\ref{num:AuxLemma_b})).\\
			 Proposition \ref{prop:VarianceBound} and Lemma \ref{lem:BoundsforExpectations} imply that
			 \begin{equation*}
			 	\EV{\mathcal{T}_{h,\sigma}^{\smash{\pm,\Set{1,3}}}}\sim h,\quad \operatorname{Var}(\mathcal{T}_{h,\sigma}^{\smash{\pm,\Set{1,3}}})\lesssim \sigma^2,\quad 
			 	\EV{\mathcal{T}_{h,\sigma}^{\smash{\pm,2}}}\sim h^2,\quad \operatorname{Var}(\mathcal{T}_{h,\sigma}^{\smash{\pm,2}})\lesssim \sigma^2h^2,
			 	\end{equation*}
			 	as $\nu\to 0$.
			 Property (\ref{num:AuxLemma_a}) follows from Chebychev's inequality and $\sigma=\smallo(h)$. Property (\ref{num:AuxLemma_b}) follows immediately from Property (\ref{num:AuxLemma_a}).\\			 
			\textbf{Step 2} (Controlling $\mathcal{I}_{h,\sigma}^{\smash{-}}$ and $\mathcal{I}_{h,\sigma}^{\smash{+}}$ and proving (\ref{num:AuxLemma_c}), (\ref{num:AuxLemma_d})).\\			
			The Property (\ref{num:AuxLemma_c}) follows from Lemma \ref{lem:BoundsforExpectations} and the computation
			\begin{equation*}
				\EV{\mathcal{I}_{h,\sigma}^{\smash{\pm}}} = \int_0^T\EV*{\norm*{\indicator_{\Gamma}K_{\pm,h}(X_t)}_{\mathbb{H}}^2}\,\D t
				\le \norm{B}\EV{\mathcal{T}_{h,\sigma}^{\smash{\pm,3}}} \lesssim h.
			\end{equation*}
			If Assumption \ref{assump:CLT} \hyperref[assump:CLT]{(noise covariance function)} is satisfied, then $\EV{\mathcal{I}_{h,\sigma}^{\smash{\pm}}}\ge \smash{\underbar{\Sigma}}\EV{\mathcal{T}_{h,\sigma}^{\smash{\pm,3}}}\gtrsim h$. Proposition \ref{prop:VarianceBound} yields $\operatorname{Var}(\mathcal{I}_{h,\sigma}^{\smash{\pm}})\lesssim \sigma^2$, and $\sigma=\smallo(h)$ yields Property (\ref{num:AuxLemma_d}).
		\end{proof}

	\subsection{Proofs of Corollary \ref{cor:CLT} and Proposition \ref{prop:LowerBound}}	\label{subsec:Proofs_Main}
	
				\begin{proof}[Proof of Corollary \ref{cor:CLT}]
			Recall the error decomposition \eqref{eq:ErrorDecomposition}, which implies
			\begin{align*}
				\frac{\mathcal{J}_{h,\sigma}}{\sigma\sqrt{(\mathcal{T}_{h,\sigma}^{\smash{+,2}})^2\mathcal{I}_{h,\sigma}^{\smash{-}}+(\mathcal{T}_{h,\sigma}^{\smash{-,2}})^2\mathcal{I}_{h,\sigma}^{\smash{+}}}}(\hat{f}(x_0)_{h,\sigma}-f(x_0)) \hspace{-20em}&\\
				&=\frac{\mathcal{J}_{h,\sigma}}{\sigma\sqrt{(\mathcal{T}_{h,\sigma}^{\smash{+,2}})^2\mathcal{I}_{h,\sigma}^{\smash{-}}+(\mathcal{T}_{h,\sigma}^{\smash{-,2}})^2\mathcal{I}_{h,\sigma}^{\smash{+}}}}B_{h,\sigma}+\frac{\mathcal{T}_{h,\sigma}^{\smash{+,2}}\EV{\mathcal{I}_{h,\sigma}^{\smash{-}}}^{1/2}Z_{h,\sigma}^{\smash{-}}+\mathcal{T}_{h,\sigma}^{\smash{-,2}}\EV{\mathcal{I}_{h,\sigma}^{\smash{+}}}^{1/2}Z_{h,\sigma}^{\smash{+}}}{\sqrt{(\mathcal{T}_{h,\sigma}^{\smash{+,2}})^2\mathcal{I}_{h,\sigma}^{\smash{-}}+(\mathcal{T}_{h,\sigma}^{\smash{-,2}})^2\mathcal{I}_{h,\sigma}^{\smash{+}}} }
			\end{align*}
			Lemma \ref{lem:DetailsStochasticError} (\ref{num:AuxLemma_b}) and (\ref{num:AuxLemma_d}) imply the convergence
			\begin{equation}
				\vphantom{\frac{\int_0^T}{\int_0^T}}\frac{(\mathcal{T}_{h,\sigma}^{\smash{+,2}})^2\mathcal{I}_{h,\sigma}^{\smash{-}}+(\mathcal{T}_{h,\sigma}^{\smash{-,2}})^2\mathcal{I}_{h,\sigma}^{\smash{+}} }{\EV{\mathcal{T}_{h,\sigma}^{\smash{+,2}}}^2\EV{\mathcal{I}_{h,\sigma}^{\smash{-}}}+\EV{\mathcal{T}_{h,\sigma}^{\smash{-,2}}}^2\EV{\mathcal{I}_{h,\sigma}^{\smash{+}}}}\xrightarrow{\prob{}}1,\label{eq:AuxConv}
			\end{equation}
			as $\nu\to 0$. In combination with Lemma \ref{lem:DetailsStochasticError} (\ref{num:AuxLemma_a}), (\ref{num:AuxLemma_b}), and $h=\smallo(\sigma^{2/(1+2\beta)})$, this yields the bound
			\begin{equation*}
				\frac{\mathcal{J}_{h,\sigma}}{\sigma\sqrt{(\mathcal{T}_{h,\sigma}^{\smash{+,2}})^2\mathcal{I}_{h,\sigma}^{\smash{-}}+(\mathcal{T}_{h,\sigma}^{\smash{-,2}})^2\mathcal{I}_{h,\sigma}^{\smash{+}}}} = \sigma^{-1}\mathcal{O}_{\prob{}}(h^{1/2})=\sigma^{-1}\smallo_{\prob{}}(\sigma^{1/(1+2\beta))}).
			\end{equation*}
			 Since Theorem \ref{thm:Nonparametric} (\ref{num:MainThm_a}) implies that $B_{h,\sigma}= \mathcal{O}(h^\beta)=\smallo(\sigma^{2\beta/(1+2\beta)})$, we see that the approximation term in the error decomposition is of the order of $\smallo_{\prob{}}(1)$.
			Combining Lemma \ref{lem:DetailsStochasticError} (\ref{num:AuxLemma_a}) with Theorem \ref{thm:Nonparametric} (\ref{num:MainThm_d}) and Slutsky's Lemma yields the convergence
			\begin{equation*}
				\begin{multlined}[t]
				\frac{\mathcal{T}_{h,\sigma}^{\smash{+,2}}\EV{\mathcal{I}_{h,\sigma}^{\smash{-}}}^{1/2}Z_{h,\sigma}^{\smash{-}}+\mathcal{T}_{h,\sigma}^{\smash{-,2}}\EV{\mathcal{I}_{h,\sigma}^{\smash{+}}}^{1/2}Z_{h,\sigma}^{\smash{+}}}{\sqrt{\EV{\mathcal{T}_{h,\sigma}^{\smash{+,2}}}^2\EV{\mathcal{I}_{h,\sigma}^{\smash{-}}}+\EV{\mathcal{T}_{h,\sigma}^{\smash{-,2}}}^2\EV{\mathcal{I}_{h,\sigma}^{\smash{+}}}}}\xrightarrow{d}N(0,1).
			\end{multlined}
			\end{equation*}
			Applying the convergence \eqref{eq:AuxConv} and Slutsky's Lemma, we obtain the claimed convergence for the stochastic error.
		\end{proof}
		
		\begin{proof}[Proof of Proposition \ref{prop:LowerBound}]
			We combine Theorems 2.1, 2.2 and Proposition 2.1 of \citet{Tsybakov2009}, similarly to Section 2.5 of \citet{Tsybakov2009}. To this end, take $f_0\equiv 0$ and $f_h(x)\coloneqq  Lh^\beta\phi_h(x)$, $h>0$, where $\phi\in \Sigma(\beta,1/2)\cap C^\infty(\mathbb{R})$, $\phi\ge 0$, $\phi(0)=1$ and $\operatorname{supp}(\phi)= B_{1/2}(0)$. Then
			\begin{enumerate}[(a)]
				\item $f_h(x_0) - f_0(x_0) = Lh^\beta$ and
				\item $f_0$, $f_h$ belong to $\Sigma(\beta,L)$.
			\end{enumerate}
			Denote by $\prob[f_h]{}$ the law induced by $X$ with $f=f_h$ on $C([0,T],\mathcal{H})$. The formula for the Girsanov density $\D\prob[f_h]{}/\D\prob[f_0]{}$ of Theorem 10.25 of \citet{DaPrato2014} allows us to compute the Kullback-Leibler divergence between $\prob[f_h]{}$ and $\prob[f_0]{}$ and implies for any $h>0$ and $0<\nu\le\bar{\nu}$ the bound
			\begin{align*}
				d_{\operatorname{KL}}\left(\prob[f_h]{}||\prob[f_0]{}\right) &= \EV[f_h][][\bigg]{\operatorname{log}\bigg(\frac{\D \prob[f_h]{}}{\D \prob[f_0]{}}\bigg)}= \EV[f_h][][\bigg]{\frac{1}{2\sigma^2}\int_{0}^T \int_{\Lambda} \frac{f_h(X_t(y))^2}{\Sigma(y)^2}\,\D y\D t-\frac{1}{\sigma}\int_{0}^T \iprod*{\Sigma^{-1}f_h(X_t)}{\D W_t}}\\
				&= \frac{1}{2\sigma^2}\int_{0}^T \int_{\Lambda}\EV*[f_h]{\Sigma(y)^{-2}f_h(X_t(y))^2}\,\D y\D t\le \frac{L^2h^{2\beta}}{2\sigma^2\smash{\underbar{\Sigma}}^2}\int_{0}^T \int_{\Lambda}\EV*[f_h]{\phi_h(X_t(y))^2}\,\D y\D t.
				\end{align*}
				The expectation $\EV[f_h]{}$ can be controlled using the upper bounds on the density $p_{\nu,t,y}$ provided by Lemma \ref{lem:BoundsforExpectations}. This shows
				\begin{equation*}
				d_{\operatorname{KL}}(\prob[f_h]{}||\prob[f_0]{})\le \frac{L^2h^{2\beta}}{2\sigma^2\smash{\underbar{\Sigma}}^2}hT^{1-\alpha/2}p_{\max}\abs{\Lambda}\norm{\phi}_{L^2(\mathbb{R})}^2(1-\alpha/2)=\colon h^{2\beta+1}\sigma^{-2}C'.
			\end{equation*}
			The choice $h=\sigma^{2/(2\beta+1)}L^{-1/\beta}$ ensures that
			\begin{enumerate}[(a)]
				\item $f_h(x_0)-f_0(x_0) = \sigma^{2\beta/(2\beta+1)}$ and
				\item $d_{\operatorname{KL}}(\prob[f_h]{}||\prob[f_0]{})\le C'L^{-2-1/\beta}$.
			\end{enumerate}						
			An application of Theorem 2.2 of \citet{Tsybakov2009} with $\alpha=C'/L$ and $C\coloneqq \max(e^{-\alpha}/4,(1-\sqrt{\alpha/2})/2)$ concludes the proof.
			\end{proof}

	 \section{Further results for Section \ref{sec:Density}}\label{sec:Other}
	 
	 This section contains the remaining proofs for Section \ref{sec:Density}.

		\subsection{Proof of Proposition \ref{prop:DensityBounds}}\label{subsec:Proofs_Density_Appendix}
		
		 We start with proving the density bounds provided in Proposition \ref{prop:DensityBounds}.
	
		\begin{proof}[Proof of Proposition \ref{prop:DensityBounds}]
		The proof uses the methodology of \citet{nourdinDensityFormulaConcentration2009}, which has been applied to SPDEs by \citet{nualartGaussianDensityEstimates2009a, nualartOptimalGaussianDensity2011}, and extends their methodology to cover more general operators $A_t$, domains $\Lambda$ and noise covariances $B$.
		
		\citet{nourdinDensityFormulaConcentration2009} establish an explicit formula for the density $p$ of a centred random variable $Z$ supported on $\mathbb{R}$ based on the auxiliary variable
		\begin{equation*}
			g_Z(z)\coloneqq\EV*{\iprod*{\mathcal{D}Z}{-\mathcal{D}L\inv Z}_{\mathfrak{H}}\given Z=z},\quad z\in\mathbb{R},
		\end{equation*}
		where $L$ is the generator of the Ornstein-Uhlenbeck semi-group (Theorem 3.1 of \citet{nourdinDensityFormulaConcentration2009}):
		\begin{equation*}
			p(z)=\frac{\EV*{\abs*{Z}}}{2 g_Z(z)}\exp\bigg(-\int_0^z \frac{x}{g_Z(x)}\,\D x\bigg),\quad z\in\mathbb{R}.
		\end{equation*}
		In Corollary 3.5, they find that if there exist $\sigma_{\min},\sigma_{\max}>0$ such that $\sigma_{\min}^2\le g_Z(Z)\le \sigma_{\max}^2$ $\prob{}$-almost surely, then
		\begin{equation}
			\frac{\EV*{\abs*{Z}}}{2\sigma_{\max}^2}\exp\left(-\frac{z^2}{2\sigma_{\min}^2}\right)\le p(z)\le \frac{\EV*{\abs*{Z}}}{2\sigma_{\min}^2}\exp\left(-\frac{z^2}{2\sigma_{\max}^2}\right),\label{eq:GeneralDensityBounds}
	\end{equation}		 
	for (Lebesgue-) almost all $z\in\mathbb{R}$. Using Mehler's formula as in Proposition 3.7 of \citet{nourdinDensityFormulaConcentration2009}, $g_Z$ can be rewritten as
		\begin{equation}
			g_Z(z) = \int_0^\infty e^{-u}\EV*{\iprod*{\Phi_Z(\mathcal{W})}{\Phi_Z\left(e^{-u}\mathcal{W} + \sqrt{1-e^{-2u}}\mathcal{W}^\circ\right)}_{\mathfrak{H}}\given Z=z}\,\D u,\label{eq:MehlerApplied}
		\end{equation}
		where $\mathcal{W}$ is the driving Gaussian process, $\mathcal{W}^\circ$ an independent copy of $\mathcal{W}$ and ${\Phi_Z(\mathcal{W}) \coloneqq \mathcal{D}Z|_{\mathcal{W}}}$ denotes the Malliavin derivative of $Z$ evaluated at the Gaussian process $\mathcal{W}$.
		
		We apply the bound \eqref{eq:GeneralDensityBounds} to $Z=X_{t}(y)-\EV[\smash(\underbar{t},\xi)]{X_{t}(y)}$ for fixed $0<t \le T-\smash{\underbar{t}} $, $y\in\Gamma$, and work under $\prob[\smash{(\underbar{t},\xi)}]{}$. Lemma \ref{lem:aux_densitybound} (below) yields the two-sided bound
		\begin{equation*}
		\iprod*{\Phi_{X_t(y)}(\mathcal{W})}{\Phi_{X_t(y)}\left(e^{-u}\mathcal{W} + \sqrt{1-e^{-2u}}\mathcal{W}^\circ\right)}_{\mathfrak{H}}
		\begin{cases}
			\ge c_2t^{\alpha},&y\in \Gamma, 0\le t\le T-\smash{\underbar{t}},u\ge 0,\\
			\le C_2t^{\alpha},& y\in\Lambda,0\le t\le T-\smash{\underbar{t}},u\ge 0,
			\end{cases}
	\end{equation*}
	$\prob[\smash{(\underbar{t},\xi)}]{}$-almost-surely, where the constants $0<c_2\le C_2<\infty$ only depend on $\underbar{C}$, $\bar{C}$, $C_0$, $\norm{f'}_\infty$, $\alpha$ and $T$.
	
	Substituting back into the reformulation \eqref{eq:MehlerApplied} and subsequently into the density bounds in \eqref{eq:GeneralDensityBounds}, we obtain
	\begin{align*}
		p_{\nu,\smash{\underbar{t}},t ,y}\left(x-\EV*[\smash{(\underbar{t},\xi)}]{X_{t}(y)}\right)&\ge \frac{\EV*[\smash{(\underbar{t},\xi)}]{\abs*{X_{t}(y) - \EV*[\smash{(\underbar{t},\xi)}]{X_{t}(y)}}}}{2 C_2 t^{\alpha}}\exp\left(-\frac{x^2}{2c_2t^{\alpha}}\right)\\
		p_{\nu,\smash{\underbar{t}},t ,y}\left(x-\EV*[\smash{(\underbar{t},\xi)}]{X_{t}(y)}\right)&\le \frac{\EV*[\smash{(\underbar{t},\xi)}]{\abs*{X_{t}(y) - \EV*[\smash{(\underbar{t},\xi)}]{X_{t}(y)}}}}{2 c_2t^{\alpha}}\exp\left(-\frac{x^2}{2C_2t^{\alpha}}\right)
	\end{align*}
	for all $x\in\mathbb{R}$, $0<\nu\le\bar{\nu}$, $y\in\Gamma$ and $0<t \le T-\smash{\underbar{t}}$. Combining this bound with the upper and lower bounds for $\EV[\smash{(\underbar{t},\xi)}]{\abs{X_{t}(y) - \EV[\smash{(\underbar{t},\xi)}]{X_{t}(y)}}}$ from Lemma \ref{lem:DeviationsX} (below) yields the claim with $t_0$ from \eqref{eq:lower_t_0}.
		\end{proof}
		
	\begin{lemma}\label{lem:aux_densitybound}
		Grant Assumptions \ref{assump:WellPosedness} \hyperref[assump:WellPosedness]{(well-posedness)} and \ref{assump:key} \hyperref[assump:key]{(noise-scaling)}. Then there exist constants $0<c_2\le C_2<\infty$, depending only on $\underbar{C}$, $\bar{C}$, $C_0$, $\norm{f'}_\infty$, $\alpha$ and $T$, such that for all diffusivity levels $0<\nu\le\bar{\nu}$, starting times $0\le \smash{\underbar{t}}<T$, deterministic initial conditions $\xi\in C(\Lambda)$, locations $y\in\Gamma$ and time points $0<t\le T-\smash{\underbar{t}}$ we have the two-sided bound
					\begin{equation*}
				\iprod*{\Phi_{X_t(y)}(\mathcal{W})}{\Phi_{X_t(y)}\left(e^{-u}\mathcal{W} + \sqrt{1-e^{-2u}}\mathcal{W}^\circ\right)}_{\mathfrak{H}}
				\begin{cases}
					\ge c_2 t^{\alpha},&u\ge 0,\\
					\le C_2t^{\alpha},&u\ge0,
					\end{cases}
			\end{equation*}
			where $\Phi_{X_{t}(y)}(\mathcal{W})$ is the Malliavin derivative $\mathcal{D}X_{t}(y)$ under $\prob[\smash{(\underbar{t},\xi)}]{}$ evaluated at $\mathcal{W}$ and $\mathcal{W}^\circ$ is an independent copy  of $\mathcal{W}$.
		\end{lemma}
		\begin{proof}
			For fixed $u\ge 0$ introduce $\mathcal{T}\coloneqq e^{-u}\mathcal{W} + \sqrt{1-e^{-2u}}\mathcal{W}^\circ$ and for $0<\delta\le t$ define the space $\mathfrak{H}_{\delta,t}\coloneqq L^2([t-\delta,t],\mathbb{H})$. Note that the representation of the Malliavin derivative from \eqref{eq:MalliavinDerivative} implies
			\begin{align}
			\begin{split}
				\iprod*{\Phi_{X_t(y)}(\mathcal{W})}{\Phi_{X_t(y)}\left(e^{-u}\mathcal{W} + \sqrt{1-e^{-2u}}\mathcal{W}^\circ\right)}_{\mathfrak{H}_{t,t}} \hspace{-5em}&\\
				&= \iprod*{\Phi_{X_t(y)}(\mathcal{W})}{\Phi_{X_t(y)}\left(e^{-u}\mathcal{W} + \sqrt{1-e^{-2u}}\mathcal{W}^\circ\right)}_{\mathfrak{H}}.
				\end{split}\label{eq:aux:DeltaCorrespondence}
			\end{align}
			\textbf{Step 1} (A decomposition). Fix $0<\delta\le t$. Using the expression for the Malliavin derivative from \eqref{eq:MalliavinDerivative} we find
	\begin{align}
		\iprod*{\Phi_{X_t(y)}(\mathcal{W})}{\Phi_{X_t(y)}\left(e^{-u}\mathcal{W} + \sqrt{1-e^{-2u}}\mathcal{W}^\circ\right)}_{\mathfrak{H}_{\delta,t}}\hspace{-13em}&\nonumber\\
		&=\begin{multlined}[t]
			\bigg\langle\sigma G_{\nu,\smash{\underbar{t}}+t,\smash{\underbar{t}}+\MTemptyplaceholder}(y,\MTemptyplaceholder) + \int_{t-\delta}^{t}\int_\Lambda G_{\nu,\smash{\underbar{t}}+t,\smash{\underbar{t}}+s}(y,\eta)f'(X_{s}^{\mathcal{W}}(\eta))\Phi_{X_{s}(\eta)}(\mathcal{W})\,\D\eta\D s,\\
		\sigma G_{\nu,\smash{\underbar{t}}+t,\smash{\underbar{t}}+\MTemptyplaceholder}(y,\MTemptyplaceholder) + \int_{t-\delta}^{t}\int_\Lambda G_{\nu,\smash{\underbar{t}}+t,\smash{\underbar{t}}+s}(y,\eta)f'(X_{s}^{\mathcal{T}}(\eta))\Phi_{X_{s}(\eta)}(\mathcal{T})\,\D\eta\D s\bigg\rangle_{\mathfrak{H}_{\delta,t}}
		\end{multlined}\nonumber\\
		&= \sigma^2\int_{t-\delta}^{t} \norm*{G_{\nu,\smash{\underbar{t}}+t,\smash{\underbar{t}}+\tau}(y,\MTemptyplaceholder)}_{\mathbb{H}}^2\,\D \tau+ R,\label{eq:Decompostion_MalliavinDerivative}
	\end{align}
	where
	\begin{align*}
		R \coloneqq& \iprod[\bigg]{\sigma G_{\nu,\smash{\underbar{t}}+t,\smash{\underbar{t}}+\MTemptyplaceholder}(y,\MTemptyplaceholder)}{\int_{t-\delta}^{t}\int_\Lambda G_{\nu,\smash{\underbar{t}}+t,\smash{\underbar{t}}+s}(y,\eta)f'(X_{s}^{\mathcal{T}}(\eta))\Phi_{X_{s}(\eta)}(\mathcal{T})\,\D\eta\D s}_{\mathfrak{H}_{\delta,t}}\\
		&\quad+\iprod[\bigg]{\sigma G_{\nu,\smash{\underbar{t}}+t,\smash{\underbar{t}}+\MTemptyplaceholder}(y,\MTemptyplaceholder)}{\int_{t-\delta}^{t}\int_\Lambda G_{\nu,\smash{\underbar{t}}+t,\smash{\underbar{t}}+s}(y,\eta)f'(X_{s}^{\mathcal{W}}(\eta))\Phi_{X_{s}(\eta)}(\mathcal{W})\,\D\eta\D s}_{\mathfrak{H}_{\delta,t}}\\
		&\quad+\bigg\langle\int_{t-\delta}^{t}\int_\Lambda G_{\nu,\smash{\underbar{t}}+t,\smash{\underbar{t}}+s}(y,\eta)f'(X_{s}^{\mathcal{W}}(\eta))\Phi_{X_{s}(\eta)}(\mathcal{W})\,\D\eta\D s,\\
		&\qquad~~ \int_{t-\delta}^{t}\int_\Lambda G_{\nu,\smash{\underbar{t}}+t,\smash{\underbar{t}}+s}(y,\eta)f'(X_{s}^{\mathcal{T}}(\eta))\Phi_{X_{s}(\eta)}(\mathcal{T})\,\D\eta\D s\bigg\rangle_{\mathfrak{H}_{\delta,t}}.
	\end{align*}
	\textbf{Step 2} (Controlling $R$ using Gronwall's inequality). We proceed by controlling $\abs{R}$. For any $\semigroup\in\Set{\mathcal{W},\mathcal{T}}$ and $t-\delta\le s\le t$ we can bound
	\begin{align*}
		\sup_{y\in\Lambda}\norm*{\Phi_{X_{s}(y)}(\semigroup)}_{\mathfrak{H}_{\delta,t}}&=\sup_{y\in\Lambda}\norm[\bigg]{\sigma G_{\nu,\smash{\underbar{t}}+s,\smash{\underbar{t}}+\MTemptyplaceholder}(y,\MTemptyplaceholder) + \int_{t-\delta}^{s} \int_\Lambda G_{\nu,\smash{\underbar{t}}+s,\smash{\underbar{t}}+v}(y,\eta)f'(X_{v}^{\semigroup}(\eta))\Phi_{X_{v}(\eta)}(\semigroup)\,\D \eta\D v}_{\mathfrak{H}_{\delta,t}}\\
		&\le \sup_{y\in\Lambda}\sigma\norm*{G_{\nu,\smash{\underbar{t}}+s,\smash{\underbar{t}}+\MTemptyplaceholder}(y,\MTemptyplaceholder)}_{\mathfrak{H}_{\delta,t}}\\
		&\quad + \norm{f'}_\infty\sup_{y\in\Lambda}\int_{t-\delta}^{s} \int_\Lambda G_{\nu,\smash{\underbar{t}}+s,\smash{\underbar{t}}+v}(y,\eta)\sup_{z\in\Lambda}\norm*{\Phi_{X_{v}(z)}(\semigroup)}_{\mathfrak{H}_{\delta,t}}\,\D \eta\D v\\
		&\le \bar{C}^{1/2}(s-(t-\delta))^{\alpha/2} + \norm{f'}_\infty C_0 \int_{t-\delta}^{s} \sup_{z\in\Lambda}\norm*{\Phi_{X_{v}(z)}(\semigroup)}_{\mathfrak{H}_{\delta,t}}\,\D v,
	\end{align*}
	where we used \eqref{eq:Upperbound} in the last line. An application of Gronwall's inequality yields
	\begin{equation*}
		\sup_{y\in\Lambda}\norm*{\Phi_{X_s(y)}(\semigroup)}_{\mathfrak{H}_{\delta,t}} \le \bar{C}^{1/2}\delta^{\alpha/2} e^{\norm{f'}_\infty C_0\delta},\quad t-\delta\le s\le t.
	\end{equation*}
	Using (\ref{num:AssumpWellposedness}) of Assumption \ref{assump:WellPosedness} \hyperref[assump:WellPosedness]{(well-posedness)} we obtain the bound
	\begin{equation*}
		\norm*{\int_{t-\delta}^{t}\int_\Lambda G_{\nu,\smash{\underbar{t}}+t,\smash{\underbar{t}}+s}(y,\eta)f'(X_{s}^{\semigroup}(\eta))\Phi_{X_{s}(\eta)}(\semigroup)\,\D\eta\D s}_{\mathfrak{H}_{\delta,t}}\le\norm{f'}_\infty C_0\bar{C}^{1/2}\delta^{1+\alpha/2}e^{\norm{f'}_\infty C_0 \delta}.
	\end{equation*}
	Recalling the upper bound \eqref{eq:Upperbound} and applying the Cauchy-Schwarz inequality shows
	\begin{align}
		\abs{R}&\le 2\bar{C}^{1/2}\delta^{\alpha/2}\norm{f'}_\infty C_0\bar{C}^{1/2}\delta^{1+\alpha/2}e^{\norm{f'}_\infty C_0 \delta} +\left(\norm{f'}_\infty C_0\bar{C}^{1/2}\delta^{1+\alpha/2}e^{\norm{f'}_\infty C_0 \delta}\right)^2\nonumber\\
		&=\bar{C}\delta^{1+\alpha}\norm{f'}_\infty C_0 e^{\norm{f'}_\infty C_0\delta}\left(2 + \delta \norm{f'}_\infty C_0e^{\norm{f'}_\infty C_0\delta}\right).\label{eq:auxcontrolremainder}
	\end{align}
	\textbf{Step 3} (The upper bound). Plugging the estimate \eqref{eq:auxcontrolremainder} back into \eqref{eq:Decompostion_MalliavinDerivative} and recalling \eqref{eq:Upperbound} yields for all $y\in\Lambda$ the upper bound
	\begin{align}
	\begin{split}
		\iprod*{\Phi_{X_t(y)}(\mathcal{W})}{\Phi_{X_t(y)}(e^{-u}\mathcal{W} + \sqrt{1-e^{-2u}}\mathcal{W}^\circ)}_{\mathfrak{H}_{\delta,t}}\hspace{-10em}&\\
		&\le \bar{C}\delta^{\alpha} + \bar{C}\delta^{1+\alpha}\norm{f'}_\infty C_0 e^{\norm{f'}_\infty C_0\delta}\left(2 + \delta \norm{f'}_\infty C_0e^{\norm{f'}_\infty C_0\delta}\right)\lesssim \delta^{\alpha},
		\end{split}\label{eq:UpperBoundAux}
	\end{align}
	where the hidden constant in $\lesssim$ only depends on $T$ and the constants in \eqref{eq:UpperBoundAux}, in particular not on $0<\nu \le \smash{\bar{\nu}}$, $y\in\Lambda$ or the initial condition $\xi$. Choosing $\delta = t$ and recalling \eqref{eq:aux:DeltaCorrespondence} yields the claimed upper bound.\\
	\textbf{Step 4} (The lower bound). We proceed to the proof of the lower bound for $y\in\Gamma$ and $0\le t\le T$. As already noted by \citet{nualartGaussianDensityEstimates2009a}, $\Phi\ge 0$ almost surely and we can bound
	\begin{align*}
		\iprod*{\Phi_{X_t(y)}(\mathcal{W})}{\Phi_{X_t(y)}(e^{-u}\mathcal{W} + \sqrt{1-e^{-2u}}\mathcal{W}^\circ)}_{\mathfrak{H}}\hspace{-10em}&\\
		&\ge \iprod*{\Phi_{X_t(y)}(\mathcal{W})}{\Phi_{X_t(y)}(e^{-u}\mathcal{W} + \sqrt{1-e^{-2u}}\mathcal{W}^\circ)}_{\mathfrak{H}_{\delta,t}}
	\end{align*}
	from below for all $0\le\delta\le t\le T-\smash{\underbar{t}}$. Choose $\delta = \delta_0 t$ with
	\begin{equation*}
			\delta_0 \coloneqq \min\left(1,\frac{1}{2T}\frac{\smash{\underbar{C}}}{\bar{C}\norm{f'}_\infty C_0 e^{\norm{f'}_\infty C_0T}(2 + T \norm{f'}C_0e^{\norm{f'}_\infty C_0T})}\right).
		\end{equation*}
		The decomposition \eqref{eq:Decompostion_MalliavinDerivative} combined with the upper bound \eqref{eq:auxcontrolremainder} and the lower bound \eqref{eq:Lowerbound} implies
		\begin{align*}
		\iprod*{\Phi_{X_t(y)}(\mathcal{W})}{\Phi_{X_t(y)}(e^{-u}\mathcal{W} + \sqrt{1-e^{-2u}}\mathcal{W}^\circ)}_{\mathfrak{H}}\hspace{-7em}&\hspace{7em}\ge \iprod*{\Phi_{X_t(y)}(\mathcal{W})}{\Phi_{X_t(y)}(e^{-u}\mathcal{W} + \sqrt{1-e^{-2u}}\mathcal{W}^\circ)}_{\mathfrak{H}_{\delta,t}}\\
		&\ge \smash{\underbar{C}}\delta_0^{\alpha}t^\alpha - \bar{C}\delta_0^{1+\alpha}t^{1+\alpha}\norm{f'}_\infty C_0 e^{\norm{f'}_\infty C_0\delta_0 t}\left(2 + \delta_0 t\norm{f'}C_0e^{\norm{f'}_\infty C_0\delta_0 t}\right)\\
		&\ge \underbar{C}\delta_0^{\alpha}t^\alpha - \bar{C}\delta_0^{1+\alpha}t^{1+\alpha}\norm{f'}_\infty C_0 e^{\norm{f'}_\infty C_0T}\left(2 + T \norm{f'}C_0e^{\norm{f'}_\infty C_0T}\right)\label{eq:LowerBoundAux}\\
		&\ge \frac{\underbar{C}}{2}\delta_0^{\alpha}t^\alpha.
	\end{align*}
	The proof is completed by noting that these bounds do not depend on $0<\nu \le \smash{\bar{\nu}}$, $y\in\Gamma$ or the initial condition $\xi$.
		\end{proof}
		
			\pagebreak
		\begin{lemma}\label{lem:DeviationsX}
		Grant Assumptions \ref{assump:WellPosedness} \hyperref[assump:WellPosedness]{(well-posedness)} and \ref{assump:key} \hyperref[assump:key]{(noise-scaling)}.
		For all starting times $0\le \smash{\underbar{t}} <T$ and deterministic initial conditions $\xi\in C(\Lambda)$, we have
			\begin{equation*}
				 \EV*[\smash{(\underbar{t},\xi)}]{\abs*{X_t(y)-\EV*[\smash{(\underbar{t},\xi)}]{X_t(y)}}}\le\frac{2(2\bar{C})^{1/2}}{\sqrt{\pi}} t^{\alpha/2}e^{\norm{f'}_\infty C_0 t},\quad y\in \Lambda, 0\le t\le T-\smash{\underbar{t}},0<\nu \le \smash{\bar{\nu}}.
			\end{equation*}
			Furthermore, with $0<t_0\le T-\smash{\underbar{t}}$ defined in \eqref{eq:lower_t_0} depending only on $\underbar{C}$, $\bar{C}$, $C_0$, $\norm{f'}_\infty$, $T$ and $\alpha$ and some constant $0<C<\infty$ depending only on $\smash{\underbar{C}}$ we have
			\begin{equation*}
				\EV*[\smash{(\underbar{t},\xi)}]{\abs*{X_t(y)-\EV*[\smash{(\underbar{t},\xi)}]{X_t(y)}}}\ge C t^{\alpha/2},\quad y\in \Gamma,0\le t\le t_0,0<\nu \le \smash{\bar{\nu}}.
			\end{equation*}
		\end{lemma}
		\begin{proof}
			\textbf{Step 1} (The upper bound). Denote by $X^{\circ}$ and $X^{\bullet}$ two independent copies of $X$ with the same initial condition $\xi$. Write $\EV[\smash{(\underbar{t},\xi)}][{\circ}]{}$ for the expectation with respect to $X^{\circ}$ and $\EV[\smash{(\underbar{t},\xi)}][{\bullet}]{}$ for the expectation with respect to $X^{\bullet}$. For all $y\in\Lambda$ and $0\le t\le T-\smash{\underbar{t}}$ we obtain the bound
			\begin{align*}
				\EV*[\smash{(\underbar{t},\xi)}]{\abs*{f(X_t(y))-\EV*[\smash{(\underbar{t},\xi)}]{f(X_t(y))}}}&= \EV*[\smash{(\underbar{t},\xi)}][{\circ}]{\abs*{\EV*[\smash{(\underbar{t},\xi)}][{\bullet}]{f(X_t^{\circ}(y))-f(X_t^{\bullet}(y))}}}\\
				&\le \EV*[\smash{(\underbar{t},\xi)}]{\abs*{f(X_t^{\circ}(y)) -f(X_t^{\bullet}(y))}}\\
				&\le \norm{f'}_\infty \EV*[\smash{(\underbar{t},\xi)}]{\abs*{X_t^{\circ}(y)-X_t^{\bullet}(y)}}.
			\end{align*}
			Denote by $\smash{\bar{X}^{\circ}}$ and $\smash{\bar{X}^{\bullet}}$ the corresponding linear parts (i.e.\ the Gaussian integrals). Applying the upper bound \eqref{eq:Upperbound} and (\ref{num:AssumpWellposedness}) of Assumption \ref{assump:WellPosedness} \hyperref[assump:WellPosedness]{(well-posedness)} yields
			\begin{align*}
				\sup_{y\in\Lambda} \EV*[\smash{(\underbar{t},\xi)}]{\abs*{X_t^{\circ}(y)-X_t^{\bullet}(y)}}&\le \sup_{y\in\Lambda}\EV*[\smash{(\underbar{t},\xi)}]{\abs*{\bar{X}_{t}^{\circ}(y)- \bar{X}_{t}^{\bullet}(y)}} \\
				&\quad + \sup_{y\in\Lambda}\norm{f'}_\infty \int_{0}^{t}\int_\Lambda G_{\nu,\smash{\underbar{t}}+t,\smash{\underbar{t}}+s}(y,\eta)\EV*[\smash{(\underbar{t},\xi)}]{\abs*{X_{s}^{\circ}(\eta)-X_{s}^{\bullet}(\eta)}}\,\D\eta\D s\\
				&\le 2\sup_{y\in\Lambda}\EV*[\smash{(\underbar{t},\xi)}]{\abs*{\bar{X}_{t}^{\circ}(y)}} + \norm{f'}_\infty C_0\int_{0}^{t} \sup_{\eta\in\Lambda}\EV*[\smash{(\underbar{t},\xi)}]{\abs*{X_{s}^{\circ}(\eta)-X_{s}^{\bullet}(\eta)}}\,\D s\\
				&\le \frac{2(2\bar{C})^{1/2}}{\sqrt{\pi}}t^{\alpha/2} + \norm{f'}_\infty C_0\int_{0}^{t} \sup_{\eta\in\Lambda}\EV*[\smash{(\underbar{t},\xi)}]{\abs*{X_{s}^{\circ}(\eta)-X_{s}^{\bullet}(\eta)}}\,\D s.
			\end{align*}
			An application of the Gronwall inequality shows
			\begin{align*}
				 \sup_{y\in\Lambda}\EV*[\smash{(\underbar{t},\xi)}]{\abs*{f(X_t(y))-\EV*[\smash{(\underbar{t},\xi)}]{f(X_t(y))}}}&\le \norm{f'}_\infty\sup_{y\in\Lambda} \EV*[\smash{(\underbar{t},\xi)}]{\abs*{X_t^{\circ}(y)-X_t^{\bullet}(y)}}\\
				 &\le \norm{f'}_\infty\frac{2(2\bar{C})^{1/2}}{\sqrt{\pi}} t^{\alpha/2}e^{\norm{f'}_\infty C_0t}.
			\end{align*}
			In particular, the claimed upper bound follows:
			\begin{equation*}
				\sup_{0<\nu \le \smash{\bar{\nu}}, y\in \Lambda}\EV*[\smash{(\underbar{t},\xi)}]{\abs*{X_t(y)-\EV*[\smash{(\underbar{t},\xi)}]{X_t(y)}}}\le \sup_{0<\nu \le \smash{\bar{\nu}}, y\in \Lambda}\EV*[\smash{(\underbar{t},\xi)}]{\abs*{X_t^{\circ}(y)-X_t^{\bullet}(y)}}\le \frac{2(2\bar{C})^{1/2}}{\sqrt{\pi}} t^{\alpha/2}e^{\norm{f'}_\infty C_0 t}.
			\end{equation*}
			\textbf{Step 2} (The lower bound). Using that $\abs{a+b}\ge \abs{a}-\abs{b}$ for real numbers $a,b$, we find
			\begin{align*}
				\EV*[\smash{(\underbar{t},\xi)}]{\abs*{X_t(y)-\EV*[\smash{(\underbar{t},\xi)}]{X_t(y)}}}\hspace{-3em}&\hspace{3em}=\mathrm{E}_{\smash{(\underbar{t},\xi)}}\Big[\sigma\Big|\int_{0}^{t} \int_\Lambda G_{\nu,\smash{\underbar{t}}+t,\smash{\underbar{t}}+s}(y,\eta) \mathcal{W}(\D\eta, \D s) \\
				&\quad+ \int_{0}^{t} \int_\Lambda G_{\nu,\smash{\underbar{t}}+t,\smash{\underbar{t}}+s}(y,\eta)[f(X_{s}(\eta))-\EV*[\smash{(\underbar{t},\xi)}]{f(X_{s}(\eta))}\,\D \eta\D s\Big|\Big]\\
				&\ge \mathrm{E}_{\smash{(\underbar{t},\xi)}}\Big[\sigma\Big|\int_{0}^{t} \int_\Lambda G_{\nu,\smash{\underbar{t}}+t,\smash{\underbar{t}}+s}(y,\eta) \mathcal{W}(\D\eta, \D s)\Big|\Big]\\
				&\quad - \int_{0}^{t} \int_\Lambda G_{\nu,\smash{\underbar{t}}+t,\smash{\underbar{t}}+s}(y,\eta)\EV*[\smash{(\underbar{t},\xi)}]{\abs*{f(X_{s}(\eta)) - \EV*[\smash{(\underbar{t},\xi)}]{f(X_{s}(\eta))}}}\,\D \eta\D s\\
				&\ge \frac{\sqrt{2}}{\sqrt{\pi}}\Big(\sigma^2\int_{0}^{t} \norm*{G_{\nu,\smash{\underbar{t}}+t,\smash{\underbar{t}}+s}(y,\MTemptyplaceholder)}_{\mathbb{H}}^2\,\D s\Big)^{1/2}-C_0\int_{0}^{t}\frac{2(2\bar{C})^{1/2}}{\sqrt{\pi}}\norm{f'}_\infty s^{\alpha/2}e^{\norm{f'}_\infty C_0 s}\,\D s\\
				&\ge t^{\alpha/2}\frac{\sqrt{2}}{\sqrt{\pi}}\underbar{C}^{1/2} - t^{1+\alpha/2}\frac{2(2\bar{C})^{1/2}}{\sqrt{\pi}}C_0\norm{f'}_\infty e^{\norm{f'}_\infty C_0 T}(1+\alpha/2)^{-1}.
			\end{align*}
			Consequently, for all $t\le \min(t_0,T-\smash{\underbar{t}})$ with
			\begin{equation}
				t_0 \coloneqq \min\left(T,\frac{1}{2}\Big(\frac{\smash{\underbar{C}^{1/2}}(1+\alpha/2)}{2\sqrt{\pi}\bar{C}^{1/2}\norm{f'}_\infty C_0 e^{\norm{f'}_\infty C_0 T}}\Big)\right)>0\label{eq:lower_t_0}
			\end{equation}				
			the lower bound
			\begin{equation*}
				\EV*[\smash{(\underbar{t},\xi)}]{\abs*{X_t(y)-\EV*[\smash{(\underbar{t},\xi)}]{X_t(y)}}}\ge t^{\alpha/2}\frac{\sqrt{2}}{2\sqrt{\pi}}\underbar{C}^{1/2}
			\end{equation*}
			holds uniformly in $y\in\Gamma$.
		\end{proof}
			\begin{lemma}\label{lem:UniformBoudnednessSolution}
		Grant Assumptions \ref{assump:WellPosedness} \hyperref[assump:WellPosedness]{(well-posedness)}, \ref{assump:key} \hyperref[assump:key]{(noise-scaling)}, fix an exponent $p\ge 1$, a starting time $0\le \smash{\underbar{t}}<T$ and a deterministic initial condition $\xi\in C(\Lambda)$. Then there exists a constant $0<C_{p,\norm{\xi}_\infty}<\infty$, depending only on $p$, $\norm{\xi}_\infty$, $T$, $f(0)$, $\norm{f'}_\infty$, $C_0$, $\alpha$ and $\bar{C}$, such that for all diffusivity levels $0<\nu \le \smash{\bar{\nu}}$, time points $0\le t\le T-\smash{\underbar{t}}$ and locations $y\in\Lambda$ we have the bound
		\begin{equation*}
			\EV*[\smash{(\underbar{t},\xi)}]{\abs*{X_{t}(y)}^p}\le C_{p,\norm{\xi}_\infty}<\infty.
		\end{equation*}
		\end{lemma}
		\begin{proof}
			Fix a time point $0\le t\le T-\smash{\underbar{t}}$. Note that the measure $G_{\nu,s,u}(y,\eta)\D\eta$ has mass of at most $C_0$ by Assumption \ref{assump:WellPosedness} \hyperref[assump:WellPosedness]{(well-posedness)}. Apply Jensen's inequality with respect to this measure and use that for $Z\sim N(0,\beta^2)$ there exists a constant $0<c_p<\infty$ depending on $p$ such that $\EV{\abs{Z}^p}^{1/p}\le c_p\EV{Z^2}^{1/2}$ to obtain
			\begin{align*}
				\sup_{0\le s\le t,y\in\Lambda}\EV*[\smash{(\underbar{t},\xi)}]{\abs*{X_{s}(y)}^p}\hspace{-5em}&\hspace{5em}\le 
					\sup_{0\le s\le t,y\in\Lambda}3^{p-1}C_0^{p-1}\int_\Lambda G_{\nu,\smash{\underbar{t}}+s,{\smash{\underbar{t}}}}(y,\eta)\abs*{\xi(\eta)}^p\,\D \eta\\
					&\quad +\sup_{0\le s\le t,y\in\Lambda} 3^{p-1}c_p\bigg(\sigma^2 \int_{0}^s \norm*{ G_{\nu,\smash{\underbar{t}}+s,\smash{\underbar{t}}+u}(y,\MTemptyplaceholder)}_{\mathbb{H}}^2\D u\bigg)^{p/2}\\
					&\quad +\sup_{0\le s\le t,y\in\Lambda} 3^{p-1}(C_0 s)^{p-1}\int_{0}^s \int_\Lambda G_{\nu,\smash{\underbar{t}}+s,\smash{\underbar{t}}+u}(y,\eta)\EV*[\smash{(\underbar{t},\xi)}]{\abs*{f(X_{u}(\eta))}^p}\,\D \eta\D u
			\end{align*}
			for any $0\le t\le T-\smash{\underbar{t}}$.
			Applying the upper bound from Assumption \ref{assump:key} \hyperref[assump:key]{(noise-scaling)} and $\abs{f(X_u)}\le \norm{f'}_\infty \abs{X_u}+\abs{f(0)}$ yields
			\begin{align*}
					\sup_{0\le s\le t,y\in\Lambda}\EV*[\smash{(\underbar{t},\xi)}]{\abs*{X_s(y)}^p}\hspace{-1em}&\hspace{1em}\le
						3^{p-1}C_0^p\norm{\abs{\xi}^p}_\infty + 3^{p-1}c_p\bar{C}^{p/2}t^{\alpha p/2}\\
						&\quad +6^{p-1} \norm{f'}_\infty^p\sup_{0\le s\le t,y\in\Lambda} (C_0s)^{p-1}\int_{0}^s \int_\Lambda G_{\nu,\smash{\underbar{t}}+s,\smash{\underbar{t}}+u}(y,\eta)\EV*[\smash{(\underbar{t},\xi)}]{\abs*{X_{u}(\eta)}^p}\,\D \eta\D u\\
						 &\quad +6^{p-1}\sup_{0\le s\le t,y\in\Lambda} (C_0s)^{p-1}\int_{0}^s \int_\Lambda G_{\nu,\smash{\underbar{t}}+s,\smash{\underbar{t}}+u}(y,\eta)\EV*[\smash{(\underbar{t},\xi)}]{\abs{f(0)}^p}\,\D \eta\D u
					\\
					&\le
						3^{p-1}C_0^p\norm{\abs{\xi}^p}_\infty + 3^{p-1}c_p\bar{C}^{p/2}t^{\alpha p/2} \\
						 &\quad + 6^{p-1}\norm{f'}_\infty^p t^{p-1}C_0^p \sup_{0\le s\le t} \int_{0}^s \sup_{y\in\Lambda}\EV*[\smash{(\underbar{t},\xi)}]{\abs{X_{u}(y)}^p}\D u +6^{p-1}t^{p}C_0^p\abs{f(0)}^p
				\\
					&\le
						 3^{p-1}C_0^p\norm*{\abs{\xi}^p}_\infty + 3^{p-1}c_p\bar{C}^{p/2}t^{\alpha p/2} \\
						&\quad + 6^{p-1}\norm{f'}_\infty^p T^{p-1}C_0^p\int_{0}^{t} \sup_{0\le u\le s,y\in\Lambda}\EV*[\smash{(\underbar{t},\xi)}]{\abs{X_u(y)}^p}\D s+6^{p-1}t^{p}C_0^p\abs{f(0)}^p.
			\end{align*}
			As the previous estimate is uniform in $0<\nu \le \smash{\bar{\nu}}$, an application of Gronwall's inequality yields the claimed bound.
			\begin{equation*}
				\EV*[\smash{(\underbar{t},\xi)}]{\abs{X_{t}(y)}^p}\le 
				3^{p-1} e^{6^{p-1}\norm{f'}_\infty^p C_0^p T^p}\left(C_0^p\norm{\abs{\xi}^p}_\infty + c_p\bar{C}^{p/2}t^{\alpha p/2} + 2^{p-1}t^{p}C_0^p\abs{f(0)}^p\right)
			\end{equation*}
			for all $0<\nu \le \smash{\bar{\nu}}$, $y\in\Lambda$, $0\le \smash{\underbar{t}} <T$ and $0\le t\le T-\smash{\underbar{t}}$.
		\end{proof}

\subsection{Proofs of Lemmas \ref{lem:CondExpectation_MalliavinIntegration} and \ref{lem:Occupationtime}}\label{subsec:Proofs_Concentration}

We first proof an auxiliary Lemma, from which Lemma \ref{lem:CondExpectation_MalliavinIntegration} follows readily.
		\begin{lemma}\label{lem:Lift}
			Grant Assumptions \ref{assump:WellPosedness} \hyperref[assump:WellPosedness]{(well-posedness)} and \ref{assump:key} \hyperref[assump:key]{(noise-scaling)}. Fix $0<\nu\le\bar{\nu}$, $0\le\tau <t\le T$ and take any $\phi\in L^\infty(\mathbb{R})$. Assume that there exists a family of random variables $(\kappa(t,s))_{s\in [\tau,t)}\subset \mathbb{R}_{\ge0}$, jointly measurable as a function $(\omega,s)\mapsto \kappa(t,s)(\omega)$ and potentially depending on $\nu$, $\tau$, $t$ and $\phi$, such that
		\begin{equation*}
			\sup_{y\in \Gamma}\abs{\EV{\phi(X_t(y))\given \mathcal{F}_s}}\le \kappa(t,s),\quad  \tau\le s<t,
		\end{equation*}
		almost surely. Then
		\begin{equation*}
			 \norm[\bigg]{\EV[][][\bigg]{\int_{\Gamma} \phi(X_t(y))\mathcal{D}_{\tau}X_t(y)\,\D y\given \mathcal{F}_\tau}}_{\mathbb{H}}\le \sigma C_0\norm{B}\abs{\Gamma}^{1/2}\bigg(\kappa(t,\tau) + e^{\norm{f'}_\infty t}\norm{f'}_\infty\int_\tau^t\EV*{\kappa(t,s)\given \mathcal{F}_\tau}\,\D s\bigg).
			\end{equation*}
		\end{lemma}
		\begin{proof}
		\textbf{Step 1} (Splitting into a linear and a non-linear part). Using the representation \eqref{eq:MalliavinDerivative} for the Malliavin derivative, we obtain
			\begin{align}
			\begin{split}
			\norm[\bigg]{\EV[][][\bigg]{\int_{\Gamma}\phi(X_t(y))\mathcal{D}_{\tau}X_t(y)\,\D y\given \mathcal{F}_\tau}}_{\mathbb{H}}\hspace{-8em}&\hspace{8em}\le \sigma\norm[\bigg]{\EV[][][\bigg]{\int_{\Gamma} \phi(X_t(y))G_{\nu,t,\tau}(y,\MTemptyplaceholder)\,\D y\given \mathcal{F}_\tau}}_{\mathbb{H}} \\
				 &\quad+\norm[\bigg]{\EV[][][\bigg]{\int_{\Gamma}\phi(X_t(y))\int_\tau^t\int_\Lambda G_{\nu,t,s}(y,\eta)f'(X_s(\eta))\mathcal{D}_{\tau}X_s(\eta)\,\D\eta\D s\D y\given \mathcal{F}_\tau}}_{\mathbb{H}}.
				 \end{split}\label{eq:aux_Gronwall_0}
		\end{align}
		We treat the two summands individually. First note that Fubini's Theorem implies
		\begin{align}
			\norm[\bigg]{\EV[][][\bigg]{\int_{\Gamma} \phi(X_t(y))G_{\nu,t,\tau}(y,\MTemptyplaceholder)\,\D y\given \mathcal{F}_\tau}}_{\mathbb{H}}&= \norm*{\EV*{\semigroup_{\nu,t,\tau}^\ast(\phi(X_t)\indicator_\Gamma)\given \mathcal{F}_\tau}}_{\mathbb{H}}= \norm*{\semigroup_{\nu,t,\tau}^\ast(\EV{\phi(X_t)\given \mathcal{F}_\tau}\indicator_\Gamma)}_{\mathbb{H}}\nonumber\\
			&\le C_0\norm{B} \abs{\Gamma}^{1/2}\kappa(t,\tau).
			\label{eq:aux_Gronwall_1}
		\end{align}
		Since $\mathcal{D}_\tau X_s$ and $f'(X_s)$ are $\mathcal{F}_s$-measurable, we use Fubini's Theorem and the tower property for conditional expectations to find
		\begin{align}
			\norm[\bigg]{\EV[][][\bigg]{\int_{\Gamma}\phi(X_t(y))\int_\tau^t\int_\Lambda G_{\nu,t,s}(y,\eta)f'(X_s(\eta))\mathcal{D}_{\tau}X_s(\eta)\,\D\eta\D s\D y\given \mathcal{F}_\tau}}_{\mathbb{H}}\hspace{-20em}&\nonumber\\
			&=\norm[\bigg]{\EV[][][\bigg]{\int_{\Gamma} \int_\tau^t \EV{\phi(X_t(y))\given \mathcal{F}_s} \int_\Lambda G_{\nu,t,s}(y,\eta)f'(X_s(\eta))\mathcal{D}_{\tau}X_s(\eta)\,\D\eta\D s\D y\given \mathcal{F}_\tau}}_{\mathbb{H}}\nonumber\\
			&\le \norm{f'}_\infty\int_\tau^t\norm[\bigg]{\EV[][][\bigg]{\int_{\Gamma}\abs{\EV{\phi(X_t(y))\given \mathcal{F}_s}}\int_\Lambda G_{\nu,t,s}(y,\eta)\abs{\mathcal{D}_{\tau}X_s(\eta)}\,\D\eta \D y\given \mathcal{F}_\tau}}_{\mathbb{H}}\D s\nonumber\\
			&\le \norm{f'}_\infty\int_\tau^t\EV[][][\bigg]{\kappa(t,s)\norm[\bigg]{\int_{\Gamma}\int_\Lambda G_{\nu,t,s}(y,\eta)\abs{\mathcal{D}_{\tau}X_s(\eta)}\,\D\eta \D y}_{\mathbb{H}}\given \mathcal{F}_\tau}\D s.\label{eq:aux_Gronwall_2}
		\end{align}
		\textbf{Step 2} (A Gronwall argument). We aim to show that
			\begin{equation}
				 L(s)\coloneqq\norm[\bigg]{\int_{\Gamma}\int_\Lambda G_{\nu,t,s}(y,\eta)\abs{\mathcal{D}_{\tau}X_s(\eta)}\,\D\eta \D y}_{\mathbb{H}} \le \sigma\norm{B}C_0\abs{\Gamma}^{1/2}e^{\norm{f'}_\infty(s-\tau)},\quad \tau \le s\le t.\label{eq:aux_Gronwall}
			\end{equation}
			To this end, fix some $\tau \le s\le t$ and use the representation \eqref{eq:MalliavinDerivative} for the Malliavin derivative $\mathcal{D}_\tau X_s(\eta)$ to obtain
		\begin{align}
		\begin{split}
				L(s)&\le \norm[\bigg]{\int_{\Gamma}\int_\Lambda G_{\nu,t,s}(y,\eta)\sigma G_{\nu,s,\tau}(\eta,\MTemptyplaceholder)\,\D\eta \D y}_{\mathbb{H}}\\
				&\quad + \norm[\bigg]{\int_{\Gamma}\int_\Lambda G_{\nu,t,s}(y,\eta)\int_\tau^s \int_\Lambda G_{\nu,s,r}(\eta,z)\abs{f'(X_r(z))}\abs{\mathcal{D}_\tau X_r(z)}\,\D z\D r\D\eta \D y}_{\mathbb{H}}.
			\end{split}\label{eq:aux_Gronwall_splitting}
			\end{align}
			For the linear term in \eqref{eq:aux_Gronwall_splitting}, we obtain
			\begin{align*}
				\norm[\bigg]{\int_{\Gamma}\int_\Lambda G_{\nu,t,s}(y,\eta)\sigma G_{\nu,s,\tau}(\eta,\MTemptyplaceholder)\,\D\eta \D y}_{\mathbb{H}} &= \sigma\norm[\bigg]{\int_\Lambda G_{\nu,s,\tau}(\eta,\MTemptyplaceholder) \int_\Lambda G_{\nu,t,s}(y,\eta)\indicator_{\Gamma}(y)\,\D y\D\eta}_{\mathbb{H}}\\
				&=\sigma \norm*{\semigroup_{\nu,s,\tau}^\ast\semigroup_{\nu,t,s}^\ast\indicator_{\Gamma}}_{\mathbb{H}} = \sigma \norm*{\semigroup_{\nu,t,\tau}^\ast\indicator_{\Gamma}}_{\mathbb{H}}\\
				&\le \sigma\norm{B}C_0\abs{\Gamma}^{1/2}.
			\end{align*}
			For the non-linear term in \eqref{eq:aux_Gronwall_splitting}, we obtain the bound
			\begin{align*}
			\norm[\bigg]{\int_{\Gamma}\int_\Lambda G_{\nu,t,s}(y,\eta)\int_\tau^s \int_\Lambda G_{\nu,s,r}(\eta,z)\abs{f'(X_r(z))}\abs{\mathcal{D}_\tau X_r(z)}\,\D z\D r\D\eta \D y}_{\mathbb{H}}\hspace{-14em}&\\
			&\le \norm{f'}_\infty\int_\tau^s\norm[\bigg]{\int_{\Gamma}\int_\Lambda G_{\nu,t,s}(y,\eta) \int_\Lambda G_{\nu,s,r}(\eta,z)\abs{\mathcal{D}_\tau X_r(z)}\,\D z\D\eta \D y}_{\mathbb{H}}\D r\\
			&= \norm{f'}_\infty \int_\tau^s\norm[\bigg]{\int_{\Gamma}\int_\Lambda G_{\nu,t,r}(y,\eta) \abs{\mathcal{D}_\tau X_r(\eta)}\,\D\eta \D y}_{\mathbb{H}}\D r.
			\end{align*}
			Substituting back into \eqref{eq:aux_Gronwall_splitting} yields
			\begin{equation*}
				L(s)\le \sigma \norm{B}C_0\abs{\Gamma}^{1/2} + \norm{f'}_\infty\int_\tau^s L(r)\,\D r.
			\end{equation*}
			An application of Gronwall's inequality yields the claimed bound \eqref{eq:aux_Gronwall}.\\
			\textbf{Step 3} (Conclusion).
				The claim follows by substituting the bound \eqref{eq:aux_Gronwall} into \eqref{eq:aux_Gronwall_2}, and subsequently \eqref{eq:aux_Gronwall_2} and \eqref{eq:aux_Gronwall_1} into \eqref{eq:aux_Gronwall_0}.
		\end{proof}
		
		\begin{proof}[Proof of Lemma \ref{lem:CondExpectation_MalliavinIntegration}]
		Using that
			\begin{equation*}
				\int_0^t \EV[][][\bigg]{\bigg(\int_\tau^t\EV*{\kappa(t,s)\given \mathcal{F}_\tau}\,\D s\bigg)^2}\,\D\tau \le \int_0^t (t-\tau)\int_\tau^t \EV{\kappa(t,s)^2}\,\D s\D\tau\le t^2 \int_0^t \EV{\kappa(t,s)^2}\,\D s,
			\end{equation*}
			the claim follows from squaring and integrating the bound from Lemma \ref{lem:Lift}.
		\end{proof}

		\begin{proof}[Proof of Lemma \ref{lem:Occupationtime}]
		We first note that Lemma \ref{lem:BoundsforExpectations} with $g=\indicator_A$ implies that
		\begin{equation*}
			\mu(A) = \int_\Gamma \EV*{\indicator_A(X_t(y))}\,\D y \sim \norm{\indicator_A}_{L^1(\mathbb{R})} = \abs{A}.
		\end{equation*}
		To prove the claimed concentration, let $a_0$ be the lower and $a_1$ the upper boundary point of $A$. For $0<\epsilon\le 1$ define the approximation
		\begin{equation*}
			g^{(\epsilon)}(x) = (1 - \epsilon^{-1}(a_0-x))\indicator_{(a_0-\epsilon,a_0)}(x) + \indicator_{[a_0,a_1]}(x)+(1-\epsilon^{-1}(x-a_1))\indicator_{(a_1,a_1+\epsilon)}(x)
		\end{equation*}
		of $\indicator_A(x)$, $x\in \mathbb{R}$. In particular, $g^{(\epsilon)}$ has the properties
		\begin{enumerate}[(a)]
			\item \label{num:aux_Occu_a} $\norm{\indicator_A - g^{(\epsilon)}}_{L^1(\mathbb{R})} = \epsilon$ and
			\item \label{num:aux_Occu_b} $\norm{(g^{(\epsilon)})'}_{L^1(\mathbb{R})} = 2$, $\norm{(g^{(\epsilon)})'}_{\infty}<\infty$, $\epsilon>0$.
		\end{enumerate}
		 We write
		\begin{align*}
			\frac{M(A)}{\mu(A)} = \frac{\int_\Gamma \EV{g^{(\epsilon)}(X_t(y))}\,\D y}{\int_\Gamma \prob{X_t(y)\in A}\,\D y}\frac{\int_\Gamma \indicator_A(X_t(y))\,\D y}{\int_\Gamma g^{(\epsilon)}(X_t(y))\,\D y}\frac{\int_\Gamma g^{(\epsilon)}(X_t(y))\,\D y}{\int_\Gamma \EV{g^{(\epsilon)}(X_t(y))}\,\D y}
		\end{align*}
		and consider the three factors individually.	We apply first Lemma \ref{lem:BoundsforExpectations} and then Property (\ref{num:aux_Occu_a}) to obtain
		\begin{align*}
			 \abs[\bigg]{\int_\Gamma \EV{g^{(\epsilon)}(X_t(y))}\,\D y -\int_\Gamma \prob{X_t(y)\in A}\,\D y} &&\le\int_\Gamma \EV{\abs{g^{(\epsilon)}(X_t(y)) - \indicator_A(X_t(y))}}\,\D y\\
			 &\le \abs{\Gamma}p_{\max}t^{-\alpha/2}\norm{g^{(\epsilon)}-\indicator_A}_{L^1(\mathbb{R})} = \mathcal{O}(\epsilon),
		\end{align*}
		uniformly in $0<\nu\le \bar{\nu}$. Similarly, an application of Markov's inequality yields
		\begin{equation*}
			\prob*[][][\bigg]{\abs[\bigg]{\int_\Gamma \indicator_A(X_t(y))\,\D y - \int_\Gamma g^{(\epsilon)}(X_t(y))\,\D y}\ge K}\le \frac{\int_\Gamma \EV{\abs{\indicator_A(X_t(y)) - g^{(\epsilon)}(X_t(y))}}\,\D y}{K}\le \mathcal{O}(\epsilon)K^{-1},
		\end{equation*}
		for all $K>0$, uniformly in $0<\nu\le \bar{\nu}$. We have shown that
		\begin{equation*}
					\frac{\int_\Gamma \EV{g^{(\epsilon)}(X_t(y))}\,\D y}{\int_\Gamma \prob{X_t(y)\in A}\,\D y}\frac{\int_\Gamma \indicator_A(X_t(y))\,\D y}{\int_\Gamma g^{(\epsilon)}(X_t(y))\,\D y} = (1+\mathcal{O}(\epsilon))(1+\mathcal{O}_{\prob{}}(\epsilon)),\quad 0<\nu\le \bar{\nu}.
		\end{equation*}
		For the last factor, we apply first Chebychev's inequality, then Proposition \ref{prop:VarianceBound} and finally Property (\ref{num:aux_Occu_b}) to obtain
		\begin{equation*}
			\prob[][][\bigg]{\abs*{\frac{\int_\Gamma g^{(\epsilon)}(X_t(y))\,\D y}{\int_\Gamma \EV{g^{(\epsilon)}(X_t(y))}\,\D y}-1}\ge K} \le \frac{\operatorname{Var}\left(\int_\Gamma g^{(\epsilon)}(X_t(y))\,\D y\right)}{K^2\left(\int_\Gamma \EV{g^{(\epsilon)}(X_t(y))}\,\D y\right)^2}\lesssim \sigma^2\frac{\norm{(g^{(\epsilon)})'}_{L^1(\mathbb{R})}^2}{K^2\norm{g^{(\epsilon)}}_{L^1(\mathbb{R})}^2}\lesssim \sigma^2K^{-2}.
		\end{equation*}
		Consequently, we have the decomposition
		\begin{equation*}
			\frac{\int_\Gamma \indicator_A(X_t(y))\,\D y}{\int_\Gamma \prob{X_t(y)\in A}\,\D y} =(1+\mathcal{O}(\epsilon))(1+\mathcal{O}_{\prob{}}(\epsilon))(1+\mathcal{O}_{\prob{}}(\sigma)).
		\end{equation*}
		We conclude by letting $\epsilon\to 0$ and $\sigma=\sigma(\nu)\to 0$.
		\end{proof}

\section{Proofs for Section \ref{sec:GrowingObs}}\label{sec:Proofs_GrowingDomain}

		As in the small diffusivity regime considered in Sections \ref{sec:Balanced_MainResults} and \ref{sec:Density}, the core of the proofs for the growing observation window asymptotic is given by spatial ergodicity results. Since $Z_t(y)$, $y\in\mathbb{R}$, $0\le t\le T$, does not depend on any asymptotic quantity, the proofs are considerably simpler than their counterparts in Sections \ref{sec:Balanced_MainResults} and \ref{sec:Density}, where $X_t(y)$ depends on $\nu\to 0$. We only give the main ideas and leave out details since they are only slight variations of the arguments for the small diffusivity regime. The key insight is that -- except for the (statistical) lower bound in Proposition \ref{prop:LowerBound} and the upper bound in Lemma \ref{lem:UniformBoudnednessSolution} -- the proofs of the results in Sections \ref{sec:Balanced_MainResults} and \ref{sec:Density} do not require $\Lambda$ to be bounded. The statement of Lemma \ref{lem:UniformBoudnednessSolution} carries over to $\Lambda=\mathbb{R}$ is we assume that the initial condition satisfies $\xi\in C_b(\mathbb{R})$, where $C_b(\mathbb{R})$ is the space of all bounded continuous functions from $\mathbb{R}$ to $\mathbb{R}$.
		
		The following result matches the Gaussian bounds found in Theorem 1.1 of \citet{nualartGaussianEstimatesDensity2012a}, which requires stronger smoothness assumptions on $f$.				
		\begin{lemma}\label{lem:DensityBounds_GrowingDomains}
		There exist constants $0<c_1\le C_1<\infty$, $0<c_2\le C_2<\infty$ and $0<t_0\le T$ depending only $\norm{f'}_\infty$ and $T$, such that for all starting times $0\le \smash{\underbar{t}}<T$, deterministic initial conditions $\xi\in C(\Lambda)$, locations $y\in\mathbb{R}$ and time points $0<t\le T-\smash{\underbar{t}}$ the Lebesgue-density $p_{\smash{\underbar{t}},t,y}$ of $Z_{t}(y)$ under $\prob*[\smash{(\smash{\underbar{t}},\xi)}]{}$ exists and satisfies the bound
		\begin{equation*}
			p_{\smash{\underbar{t}},t,y}\left(x-\EV[\smash{(\smash{\underbar{t}},\xi)}]{Z_{t}(y)}\right)\le C_1 t^{-1/4}\exp\left(-\frac{x^2}{2C_2t^{1/2}}\right),\quad x\in\mathbb{R}.
		\end{equation*}
		If, moreover, $0<t\le t_0$, then
		\begin{equation*}
			p_{\smash{\underbar{t}},t,y}\left(x-\EV[\smash{(\smash{\underbar{t}},\xi)}]{Z_{t}(y)}\right)\ge c_1 t^{-1/4}\exp\left(-\frac{x^2}{2c_2t^{1/2}}\right),\quad x\in\mathbb{R}.
		\end{equation*}
		\end{lemma}
	\begin{proof}
		The proof is analogous to the proof of Proposition \ref{prop:DensityBounds} by replacing the Green function $G_{\nu,t,\smash{\underbar{t}}}(y,\eta)$ with the heat kernel $\phi_{t-\smash{\underbar{t}}}(y-\eta)$, $0\le \smash{\underbar{t}}<t\le T$, $y,\eta\in\mathbb{R}$, and recalling $C_0=1$ in this setting.
	\end{proof}

		\begin{corollary}\label{cor:DensityBounds_GrowingDomains}~
		\begin{enumerate}[(a)]
			\item There exists a constant $0<p_{\max}<\infty$, depending only on $\norm{f'}_\infty$ and $T$ such that for all starting times $0\le \smash{\underbar{t}} <T$, deterministic initial conditions $\xi\in C_b(\mathbb{R})$, locations $y\in\mathbb{R}$ and times $0<t \le T-\smash{\underbar{t}}$ the Lebesgue-density $p_{\smash{\underbar{t}},t,y}$ of $Z_{t}(y)$ under $\prob[\smash{(\underbar{t},\xi)}]{}$ exists and satisfies
			\begin{equation*}
				p_{\smash{\underbar{t}},t,y}(x)\le p_{\max}t^{-1/4}<\infty,\quad x\in\mathbb{R}.
			\end{equation*}
			\item Consider the starting configuration $\smash{\underbar{t}}=0$ and $\xi\in C_b(\mathbb{R})$ deterministic. For every fixed bounded subset $\mathcal{N}\subset\mathbb{R}$ and $0<\Delta<t_0$ with $0<t_0\le T$ sufficiently small there exists a constant $0<p_{\min,\norm{\xi}_\infty,\mathcal{N},\Delta}<\infty$, depending on $\norm{\xi}_\infty$, $\mathcal{N}$, $\Delta$, $\norm{f'}_\infty$ and $T$, such that
			\begin{equation*}
				p_{\smash{\underbar{t}},t,y}(x) \ge p_{\min,\norm{\xi}_\infty,\mathcal{N}, \Delta}>0,\quad x\in\mathcal{N},
			\end{equation*}
			for all $y\in\mathbb{R}$ and $\Delta \le t\le t_0$.
		\end{enumerate}
		\end{corollary}
		\begin{proof}
			The claimed upper bound follows directly from Lemma \ref{lem:DensityBounds_GrowingDomains}. The lower bound follows from Lemma \ref{lem:DetailsStochasticError_GrowingDomain} and the bound $\sup_{0\le t\le T-\smash{\underbar{t}}, y\in\mathbb{R}}\EV[(\smash{\underbar{t},\xi})]{\abs{Z_t(y)}}<\infty$ if $\norm{\xi}_\infty<\infty$ (analogously to Lemma  \ref{lem:UniformBoudnednessSolution}).
		\end{proof}
		Having established the necessary density bounds, we can proceed to prove concentration of functionals $\mathcal{G}_{g,\gamma}$
		for suitable functions $g\colon\mathbb{R}\to\mathbb{R}$ as $\gamma=\abs{\Gamma}\to\infty$, which are stated in Lemma \ref{lem:Ergodicity_GrowingDomains}.

		\begin{proof}[Proof of Lemma \ref{lem:Ergodicity_GrowingDomains}]
		The proof is analogous to the proofs of Proposition \ref{prop:VarianceBound}, Lemma \ref{lem:BoundsforExpectations} and Proposition \ref{prop:SpatailErgodicity}, with Corollary \ref{cor:DensityBounds} replaced by Corollary \ref{cor:DensityBounds_GrowingDomains}.
		\end{proof}
		Given the concentration results provided by Lemma \ref{lem:Ergodicity_GrowingDomains}, we can proceed to prove Theorem \ref{thm:Nonparametric_GrowingDomains}.
 Similarly to Lemma \ref{lem:DetailsStochasticError}, the first step is to control the auxiliary quantities which constitute the estimator $\hat{f}(x_0)_{h,\gamma}$ from \eqref{eq:Estimator}.
		\begin{lemma}\label{lem:DetailsStochasticError_GrowingDomain}
			Assume $\sqrt{\gamma} h\to\infty$, then
			\begin{enumerate}[(a)]
			\item \label{num:AuxLemma_a_GD} the auxiliary quantities $\mathcal{T}_{h,\gamma}^{\smash{+,1}}$ ,$\mathcal{T}_{h,\gamma}^{\smash{+,2}}$, $\mathcal{T}_{h,\gamma}^{\smash{-,1}}$ and $\mathcal{T}_{h,\gamma}^{\smash{-,2}}$ satisfy
			\begin{align*}
			 	\EV{\mathcal{T}_{h,\gamma}^{\smash{\pm,1}}}&\sim \gamma h,\quad \mathcal{T}_{h,\gamma}^{\smash{\pm,1}}=\EV{\mathcal{T}_{h,\gamma}^{\smash{\pm,1}}} + \smallo_{\prob{}}\big(\EV{\mathcal{T}_{h,\gamma}^{\smash{\pm,1}}}\big),\\
			 	\EV{\mathcal{T}_{h,\gamma}^{\smash{\pm,2}}}&\sim \gamma h^2,\quad \mathcal{T}_{h,\gamma}^{\smash{\pm,2}}=\EV{\mathcal{T}_{h,\gamma}^{\smash{\pm,2}}} + \smallo_{\prob{}}\big(\EV{\mathcal{T}_{h,\gamma}^{\smash{\pm,2}}}\big).
			 	\end{align*}
			\item\label{num:AuxLemma_b_GD} Furthermore, we find $\EV{\mathcal{T}_{h,\gamma}^{\smash{-,1}}}\EV{\mathcal{T}_{h,\gamma}^{\smash{+,2}}} + \EV{\mathcal{T}_{h,\gamma}^{\smash{+,1}}}\EV{\mathcal{T}_{h,\gamma}^{\smash{-,2}}}\sim \gamma^2 h^3$ and
			\begin{equation*}
				\mathcal{J}_{h,\gamma} = \EV{\mathcal{T}_{h,\gamma}^{\smash{-,1}}}\EV{\mathcal{T}_{h,\gamma}^{\smash{+,2}}} + \EV{\mathcal{T}_{h,\gamma}^{\smash{+,1}}}\EV{\mathcal{T}_{h,\gamma}^{\smash{-,2}}} + \smallo_{\prob{}}\big(\EV{\mathcal{T}_{h,\gamma}^{\smash{-,1}}}\EV{\mathcal{T}_{h,\gamma}^{\smash{+,2}}} + \EV{\mathcal{T}_{h,\gamma}^{\smash{+,1}}}\EV{\mathcal{T}_{h,\gamma}^{\smash{-,2}}}\big).
			\end{equation*}
			\item Moreover, we have $\EV{\mathcal{I}_{h,\gamma}^{\smash{\pm}}}\sim \gamma h$ and $\mathcal{I}_{h,\gamma}^{\smash{\pm}} = \EV{\mathcal{I}_{h,\gamma}^{\smash{\pm}}} + \smallo_{\prob{}}(\EV{\mathcal{I}_{h,\gamma}^{\smash{\pm}}})$.
			\end{enumerate}
		\end{lemma}
		\begin{proof}
		The proof is analogous to the proof of Lemma \ref{lem:DetailsStochasticError} with Lemma \ref{lem:Ergodicity_GrowingDomains} instead of Proposition \ref{prop:VarianceBound} and Lemma \ref{lem:BoundsforExpectations}.
		\end{proof}
		\begin{proof}[Proof of Theorem \ref{thm:Nonparametric_GrowingDomains}]
			The proof is analogous to the proof of Theorem \ref{thm:Nonparametric} by using Lemma \ref{lem:DetailsStochasticError_GrowingDomain} instead of Lemma \ref{lem:DetailsStochasticError}.
		\end{proof}
		
	\section{Well-posedness results}
	
	This section contains the well-posedness results for the SPDE \eqref{eq:SPDE1} based on Assumption \ref{assump:WellPosedness} {\hyperref[assump:WellPosedness]{(well-posedness)}}.

		\begin{lemma}\label{lem:aux_WellPosedness}
			Grant Assumption \ref{assump:WellPosedness} \hyperref[assump:WellPosedness]{(well-posedness)}, fix any diffusivity level $0<\nu\le \bar{\nu}$, noise level $\sigma>0$ and initial configuration $(\smash{\underbar{t}},\xi)\in [0,T)\times C(\Lambda)$.
			\begin{enumerate}[(a)]
			\item\label{num:aux_wellposed_a} The semi-linear SPDE \eqref{eq:SPDE1} has a unique mild solution $X_t(y)$, $0\le t\le T$, $y\in\Lambda$, given by \eqref{eq:RandomField1}.
			\item\label{num:aux_wellposed_b} The Malliavin derivative $\mathcal{D}X_t(y)\in\mathfrak{H}$ of $X_t(y)$ from \eqref{eq:RandomField1} exists for all times $0< t\le T$, locations $y\in\Lambda$, and satisfies \eqref{eq:MalliavinDerivative}.
			\item\label{num:aux_wellposed_c} $X_t$ from \eqref{eq:RandomField1} is also an analytically weak solution to \eqref{eq:SPDE1} in the sense that
			\begin{equation*}
				 \iprod*{X_t}{\phi}-\iprod*{X_0}{\phi} = \int_0^t \left(\iprod*{X_s}{\nu A_s^\ast \phi} + \iprod*{F(X_s)}{\phi}\right)\,\D s + \sigma \iprod*{\phi}{\D W_s}
			\end{equation*}
			for all $\phi\in C_c^\infty(\Lambda)\subset \operatorname{dom}(A_0^\ast)$.
			\end{enumerate}
		\end{lemma}
		\begin{proof}
		Since the diffusivity $\nu>0$ and the noise level $\sigma>0$ are fixed, we set $\nu=\sigma=1$ without loss of generality and write $G_{\smash{\underbar{t}},t}$ instead of $G_{\nu,\smash{\underbar{t}},t}$.
		\begin{enumerate}[(a)]
			\item The proof follows the steps of Theorem 2.4.3 of \citet{nualartMalliavinCalculusRelated2006}, taking into account that $A_t$ is time-dependent. Consider the standard Picard iteration scheme with
			\begin{equation*}
				X_t^{(0)}(y) = \int_\Lambda G_{\smash{\underbar{t}}+t,\smash{\underbar{t}}}(y,\eta)\xi(\eta)\,\D\eta+\int_0^t \int_\Lambda G_{\smash{\underbar{t}}+t,\smash{\underbar{t}}+s}(y,\eta)\mathcal{W}(\D\eta,\D s)
			\end{equation*}
			for $0\le t\le T-\smash{\underbar{t}}$ and $y\in\Lambda$, where the stochastic integral is well-defined by Assumption \ref{assump:WellPosedness} \hyperref[assump:WellPosedness]{(well-posedness)} (\ref{num:AssumpWellposedNess_c}), and
			\begin{equation}
				X_t^{(n+1)}(y) = X_t^{(0)}(y) + \int_0^t\int_\Lambda G_{\smash{\underbar{t}}+t,\smash{\underbar{t}}+s}(y,\eta)f(X_s^{(n)}(\eta))\,\D\eta\D s\label{eq:Picard}
			\end{equation}
			for $0\le t\le T-\smash{\underbar{t}}$, $y\in\Lambda$ and $n\in\mathbb{N}_0 $. For every $p\ge 1$ we find
			\begin{equation*}
				 \EV*{\abs*{X_t^{(n+1)}(y)-X_t^{(n)}(y)}^p}\le \norm{f'}_\infty^p \EV[][][\bigg]{\bigg(\int_0^t \int_\Lambda G_{\smash{\underbar{t}}+t,\smash{\underbar{t}}+s}(y,\eta)\abs*{X_s^{(n)}(\eta)-X_s^{(n-1)}(\eta)}\,\D\eta\D s\bigg)^p}.
			\end{equation*}
			Using that $G_{\smash{\underbar{t}}+t,\smash{\underbar{t}}+s}(y,\eta)\,\D\eta$ is a measure with mass of at most $C_0$ by Assumption \ref{assump:WellPosedness} \hyperref[assump:WellPosedness]{(well-posedness)} (\ref{num:AssumpWellposedness}), we apply Jensen's inequality to find for all $0\le t\le T-\smash{\underbar{t}}$ the bounds
			\begin{align*}
				 V_n(t)&\coloneqq \sup_{y\in\Lambda}\EV*{\abs*{X_t^{(n+1)}(y)-X_t^{(n)}(y)}^p}\\
				 &\le (C_0t)^{p-1}\norm{f'}_\infty^p \sup_{y\in\Lambda}\EV[][][\bigg]{\int_0^t \int_\Lambda G_{\smash{\underbar{t}}+t,\smash{\underbar{t}}+s}(y,\eta)\abs*{X_s^{(n)}(\eta)-X_s^{(n-1)}(\eta)}^p\,\D\eta\D s}\\
				 &\le (C_0t)^{p-1}\norm{f'}_\infty^p \sup_{y\in\Lambda}\EV[][][\bigg]{\int_0^t \sup_{z\in\Lambda}\abs*{X_s^{(n)}(z)-X_s^{(n-1)}(z)}^p \int_\Lambda G_{\smash{\underbar{t}}+t,\smash{\underbar{t}}+s}(y,\eta)\,\D\eta\D s}\\
				 &\le C_0^p t^{p-1}\norm{f'}_\infty^p \int_0^t V_{n-1}(s)\,\D s.
			\end{align*}
			Gronwall's inequality in the form of Lemma 6.2 of \citet{SanzSole2005} with $k_1=k_2=0$ shows that
			\begin{equation*}
				\sup_{y\in\Lambda, 0\le t\le T} \EV[][][\bigg]{\abs[\bigg]{X_t^{(n+1)}(y)-X_t^{(n)}(y)}^p} = \sup_{0\le t\le T}V_n(t) \to 0
\end{equation*}
as $n\to\infty$. This yields convergence of $X^{(n)}$ in the Banach space
\begin{equation*}
	\Set*{u\colon \Lambda \times [0,T]\to \mathbb{R},\quad (y,t)\mapsto u_t(y)\given \tnorm{u}\coloneq \sup_{y\in\Lambda,0\le t\le T}\EV{\abs{u_t(y)}^p}^{1/p}<\infty}
\end{equation*}
to an adapted and unique stochastic process $(X_t(y))_{t\in [0,T]}$ satisfying \eqref{eq:RandomField1}. 
	\item The proof is analogous to the proof of Proposition 2.4.4 of \citet{nualartMalliavinCalculusRelated2006} and of Lemma 7.3 of \citet{SanzSole2005}, taking into account that $A_t$ is time-dependent. Consider the Picard iteration of \eqref{eq:Picard}. Using Assumption \ref{assump:WellPosedness} \hyperref[assump:WellPosedness]{(well-posedness)} (\ref{num:AssumpWellposedNess_c}), we find
			\begin{equation*}
				\sup_{0\le t\le T-\smash{\underbar{t}},y\in\Lambda}\EV{\norm{\mathcal{D}X_t^{(0)}(y)}_{\mathfrak{H}}^2} = \sup_{0\le t\le T-\smash{\underbar{t}},y\in\Lambda}\int_0^t \norm*{G_{\smash{\underbar{t}}+t,\smash{\underbar{t}}+\tau}(y,\MTemptyplaceholder)}_{\mathbb{H}}^2\,\D \tau \le C<\infty,
			\end{equation*}
			for some constant $0<C<\infty$. Assuming that $X_t^{(n)}(y)$ is Malliavin-differentiable, Proposition 1.2.4 of \citet{nualartMalliavinCalculusRelated2006} implies that $\mathcal{D}f(X_t^{(n)}(y))) = G_t(y)\mathcal{D}X_t^{(n)}(y)$ for a random variable $H_t(y)$ bounded by $\norm{f'}_\infty$ and $H_t(y)=f'(X_t^{(n)}(y))$ if $f$ is continuously differentiable. We write $\mathcal{D}f(X_t^{(n)}(y)) = f'(X_t^{(n)}(y))\mathcal{D}X_t^{(n)}(y)$ for notational convenience in both cases.
			Using Jensen's inequality for the finite measure $G_{\smash{\underbar{t}}+s,\smash{\underbar{t}}+r}(y,\eta)\,\D\eta\D r$ we compute for $0\le t\le T-\smash{\underbar{t}}$ and $n\in\mathbb{N}_0$ the bounds
			\begin{align*}
				V_{n+1}(t)&\coloneqq \sup_{0\le s\le t,y\in\Lambda}\EV[][][\bigg]{ \norm*{\mathcal{D} X_s^{(n)}(y)}_{\mathfrak{H}}^2}\\
				&= \sup_{0\le s\le t,y\in\Lambda}\EV[][][\bigg]{\norm[\bigg]{\mathcal{D} X_s^{(0)}(y) + \int_0^s\int_\Lambda G_{\smash{\underbar{t}}+s,\smash{\underbar{t}}+r}(y,\eta)f'(X_r^{(n)}(\eta))\mathcal{D} X_r^{(n)}(\eta)\,\D\eta\D r}_{\mathfrak{H}}^2}\\
				&\le 2 C + 2 \sup_{0\le s\le t,y\in\Lambda}\EV[][][\bigg]{\norm[\bigg]{\int_0^s\int_\Lambda G_{\smash{\underbar{t}}+s,\smash{\underbar{t}}+r}(y,\eta)f'(X_r^{(n)}(\eta))\mathcal{D} X_r^{(n)}(\eta)\,\D\eta\D r}_{\mathfrak{H}}^2}\\
				&\le 2 C + 2 \sup_{0\le s\le t,y\in\Lambda}sC_0\EV[][][\bigg]{\int_0^s\int_\Lambda G_{\smash{\underbar{t}}+s,\smash{\underbar{t}}+r}(y,\eta)\abs*{f'(X_r^{(n)}(\eta))}^2\norm*{\mathcal{D} X_r^{(n)}(\eta)}_{\mathfrak{H}}^2\,\D\eta\D r}\\
				&\le 2 C + 2 tC_0\norm{f'}_\infty^2\sup_{0\le s\le t,y\in\Lambda}\int_0^s \sup_{z\in\Lambda}\EV*{\norm*{\mathcal{D} X_r^{(n)}(z)}_{\mathfrak{H}}^2}\int_\Lambda G_{\smash{\underbar{t}}+s,\smash{\underbar{t}}+r}(y,\eta)\,\D\eta\D r\\
				&\le 2 C + 2 (T-\smash{\underbar{t}})C_0^2\norm{f'}_\infty^2\sup_{0\le s\le t}\int_0^s \sup_{0\le u\le r,z\in\Lambda}\EV*{\norm*{\mathcal{D} X_u^{(n)}(z)}_{\mathfrak{H}}^2}\D r\\
				&= 2 C + 2 (T-\smash{\underbar{t}})C_0^2\norm{f'}_\infty^2 \int_0^t V_n(s)\,\D s.
			\end{align*}
			Applying the Gronwall inequality (Lemma 6.2 of \citet{SanzSole2005}) with $k_1 = 2 C$ and $k_2 = 0$, we find $V_n(t)<\infty$, uniformly in $0\le t\le T-\smash{\underbar{t}}$, $y\in\Lambda$ and $n\in\mathbb{N}_0$. We conclude as in the proof of Proposition 2.4.4 of \citet{nualartMalliavinCalculusRelated2006}: Since $X_t^{(n)}(y)$ converges to $X_t(y)$ in $L^p(\Omega)$ for all $p\ge 1$, the Malliavin derivative $\mathcal{D}X_t(y)$ of $X_t(y)$ from \eqref{eq:RandomField1} exists. Even more, we find that $\mathcal{D}X_t^{(n)}(y)$ converges weakly to $\mathcal{D}X_t(y)$ in $L^2(\Omega,\mathfrak{H})$ by Lemma 1.2.3 of \citet{nualartMalliavinCalculusRelated2006} and the claim follows from applying $\mathcal{D}$ to \eqref{eq:RandomField1}.
			\item We follow the proofs of Proposition 3.2 of \citet{pardouxStochasticPartialDifferential2021} and of Theorem 2.1 of \citet{Ibragimov1999}. Take any $\phi\in C_c^\infty(\Lambda)\subset \operatorname{dom}(A_0^\ast)$, and $0\le s<t\le T-\smash{\underbar{t}}$. Using the stochastic Fubini Theorem for the stochastic integral, we find
			\begin{align*}
				\iprod*{X_t}{\phi}&= \iprod*{\semigroup_{\smash{\underbar{t}}+t,\smash{\underbar{t}}+s}X_s}{\phi} + \iprod[\bigg]{\int_s^t \semigroup_{\smash{\underbar{t}+t},\smash{\underbar{t}}+r}F(X_r)\, \D r}{\phi}+ \iprod[\bigg]{\int_s^t \semigroup_{\smash{\underbar{t}}+t,\smash{\underbar{t}}+r}\,\D W_r}{\phi}\\
				&= \iprod*{X_s}{\semigroup_{\smash{\underbar{t}}+t,\smash{\underbar{t}}+s}^\ast \phi}+ \int_s^t \iprod*{\semigroup_{\smash{\underbar{t}+t},\smash{\underbar{t}}+r}^\ast \phi}{F(X_r)}\, \D r+ \int_s^t \iprod*{\semigroup_{\smash{\underbar{t}}+t,\smash{\underbar{t}}+r}^\ast\phi}{\D W_r}.
			\end{align*}
			For fixed $n\in\mathbb{N}$ let $t_i = it/n$, $i=0,\dots ,n$. Using \eqref{eq:Commuting} we find
			\begin{align*}
				\iprod*{X_{t_i}}{\semigroup_{\smash{\underbar{t}}+t_{i+1},\smash{\underbar{t}}+t_i}^\ast\phi}-\iprod*{X_{t_{i}}}{\phi} &= \int_{t_i}^{t_{i+1}} \frac{\partial}{\partial r} \iprod*{X_{t_i}}{\semigroup_{\smash{\underbar{t}}+r,\smash{\underbar{t}}+t_i}^\ast \phi}\,\D r= \int_{t_i}^{t_{i+1}} \iprod[\bigg]{X_{t_i}}{ \frac{\partial}{\partial r}\semigroup_{\smash{\underbar{t}}+r,\smash{\underbar{t}}+t_i}^\ast\phi}\,\D r\\
				&=\int_{t_i}^{t_{i+1}}  \iprod*{X_{t_i}}{A_r^\ast \semigroup_{\smash{\underbar{t}}+r,\smash{\underbar{t}}+t_i}^\ast\phi}\,\D r=\int_{t_i}^{t_{i+1}}  \iprod*{X_{t_i}}{\semigroup_{\smash{\underbar{t}}+r,\smash{\underbar{t}}+t_i}^\ast A_r^\ast\phi}\,\D r.
			\end{align*}
			This implies
			\begin{align*}
				\iprod*{X_t}{\phi} - \iprod*{\xi}{\phi}&= \sum_{i=0}^{n-1}\left(\iprod*{X_{t_{i+1}}}{\phi} -\iprod*{X_{t_{i}}}{\phi} \right)\\
				&= \sum_{i=0}^{n-1} \left(\iprod*{X_{t_{i+1}}}{\phi} -\iprod*{X_{t_i}}{\semigroup_{\smash{\underbar{t}}+t_{i+1},\smash{\underbar{t}}+t_i}^\ast\phi}+\iprod*{X_{t_i}}{\semigroup_{\smash{\underbar{t}}+t_{i+1},\smash{\underbar{t}}+t_i}^\ast\phi}-\iprod*{X_{t_{i}}}{\phi} \right)\\
				&= \sum_{i=0}^{n-1}\bigg(\int_{t_i}^{t_{i+1}} \iprod*{\semigroup_{\smash{\underbar{t}+t_{i+1}},\smash{\underbar{t}}+r}^\ast\phi}{F(X_r)}\,\D r+ \int_{t_i}^{t_{i+1}} \iprod*{\semigroup_{\smash{\underbar{t}}+t_{i+1},\smash{\underbar{t}}+r}^\ast \phi}{\D W_r}\\
				&\quad + \int_{t_i}^{t_{i+1}}  \iprod*{X_{t_i}}{\semigroup_{\smash{\underbar{t}}+r,\smash{\underbar{t}}+t_i}^\ast A_r^\ast \phi}\,\D r\bigg).
			\end{align*}
			Since $\semigroup_{{r}+\Delta,r}$ converges to the identity operator on $\mathcal{H}$ for all $0\le r\le T$ as $\Delta\to 0$ and $X$ has almost surely continuous paths, the claim follows by letting $n\to\infty$.\qedhere
			\end{enumerate}
		\end{proof}
		\begin{remark}
			Note that Lemma \ref{lem:aux_WellPosedness} also applies to $\Lambda=\mathbb{R}$ if the bounds from Assumption \ref{assump:WellPosedness} hold true and $\xi\in C_b(\mathbb{R})$.
		\end{remark}

\printbibliography

\end{document}